\documentclass[10pt,reqno]{amsart}
\usepackage{amsthm}
\usepackage{amscd}
\usepackage[normalem]{ulem}
\usepackage{amsfonts}
\usepackage{amssymb}
\usepackage{mathrsfs}
\usepackage{cancel}
\usepackage{comment}
\usepackage{graphicx, wrapfig}
\usepackage{amsmath}
\usepackage[top=1.25in, bottom=1.25in, left=1.25in, right=1.25in]{geometry}
\usepackage{natbib}
\usepackage{lineno}
\usepackage{wrapfig}
\usepackage{enumerate}
\usepackage{setspace}
\usepackage{hyperref}

\usepackage{xcolor}

\definecolor{red}{rgb}{1.0,0.0,0.0}

\definecolor{blu}{rgb}{0.0,0.0,1.0}
 \definecolor{gre}{rgb}{0.0,0.5,0.2}

\definecolor{darkviolet}{rgb}{0.58, 0.0, 0.83}

\DeclareMathOperator*{\argmax}{arg\,max}
\newcommand{\ud}{{\, \mathrm{d}}}

  \newcommand{\R}{\mathbb{R}}
 
 \newcommand{\ee}{\end{equation}}
 \newcommand{\A}{\mathcal A}
  \newcommand{\X}{X}
\newcommand{\loc}{\textrm{loc}}
 \newcommand{\e}{\varepsilon}
\newtheoremstyle{mytheorem}
{6pt}
{6pt}
{\itshape}
{-0pt}
{\large \scshape}
{}
{1em}
{}

\newtheoremstyle{myremark}
{6pt}
{10pt}
{\rm}
{-0pt}
{\large \scshape}
{}
{1em}
{}

\newcommand{\D}{ D(\A^*)'}

\theoremstyle{mytheorem}

\newtheorem{Theorem}{Theorem}[section]
\newtheorem{Proposition}[Theorem]{Proposition}

\newtheorem{Corollary}[Theorem]{Corollary}
\newtheorem{Definition}[Theorem]{Definition}

 \newtheorem{Assumption}[Theorem]{Assumption}

\theoremstyle{myremark}
\newtheorem{Remark}[Theorem]{Remark}
\newtheorem{Example}[Theorem]{Example}

\newlength{\bibitemsep}
\newlength{\bibparskip}
\let\oldthebibliography\thebibliography
\renewcommand\thebibliography[1]{
	\oldthebibliography{#1}
	\setlength{\parskip}{\bibitemsep}
	\setlength{\itemsep}{\bibparskip}
}

\numberwithin{equation}{section}

\bibpunct{(}{)}{; }{a}{,}{,}
\selectfont

\usepackage[countmax]{subfloat}
\usepackage{caption}
\usepackage{subfig}
\usepackage{subfloat}

\begin{document}
	\title[Optimal Control in Infinite Dimensional Spaces
and Economic Modeling]{Optimal Control in Infinite Dimensional Spaces
and Economic Modeling: State of the Art and Perspectives}
	\author[G.Fabbri]{Giorgio Fabbri$^*$}
	\address{$^*$Univ. Grenoble Alpes, CNRS, INRA, Grenoble INP, GAEL, 38000	Grenoble, France.}
	\email{giorgio.fabbri@univ-grenoble-alpes.fr}
    	\author[S. Faggian]{Silvia Faggian$^{**}$}
	\address{$^{**}$ Dipartimento di Economia,  Università Ca' Foscari Venezia, Italy.}
	\email{faggian@unive.it}
	\author[S. Federico]{Salvatore Federico$^\dag$}
	\address{$^\dag$Dipartimento di Matematica, Università di Bologna, Bologna, Italy}
	\email{s.federico@unibo.it}
	\author[F. Gozzi]{Fausto Gozzi$^\ddag$}
	\address{$^\ddag$Université LUISS - G. Carli, Roma, Italy}
	\email{fgozzi@luiss.it}
	
	\begin{abstract}
This survey collects, within a unified framework, various results (primarily by the authors themselves) on the use of Deterministic Infinite-Dimensional Optimal Control Theory to address applied economic models. The main aim is to illustrate, through several examples, the typical features of such models (including state constraints, non-Lipschitz data, and non-regularizing differential operators) and the corresponding methods needed to handle them. This necessitates developing aspects of the existing Deterministic Infinite-Dimensional Optimal Control Theory (see, e.g., the book by \citealp{li2012optimal}) in specific and often nontrivial directions. Given the breadth of this area, we emphasize the Dynamic Programming Approach and its application to problems where explicit or quasi-explicit solutions of the associated Hamilton–Jacobi–Bellman (HJB) equations can be obtained. We also provide insights and references for cases where such explicit solutions are not available.

\vskip0.3truecm

  \medskip\ \newline
        \vskip0.1truecm

		\noindent\textit{JEL Classification}:

\vskip0.2truecm

		 
	\end{abstract}
	
	\maketitle

\tableofcontents



\section{Introduction}
\label{sec:intro}
This survey paper collects, in a unified setting, various results (primarily by the authors themselves) on the use of deterministic infinite-dimensional Optimal Control Theory to study applied models arising in Economics. The main aim is to explain, through several examples, the typical features of such models (such as state constraints, non-Lipschitz data, non-regularizing differential operators) and the methods to handle them, which require developing the existing theory in infinite dimesnion (see, e.g., the book \citealp{li2012optimal}) in specific and often nontrivial directions. Given the breadth of this area, we mainly focus on the Dynamic Programming Approach and its application to problems where explicit or quasi-explicit solutions of the associated Hamilton–Jacobi–Bellman (HJB) equations can be obtained. We also provide ideas and references on how to address cases where such explicit solutions are not available.

\subsection{Motivation and some ideas on examples}

Mathematical models are a very useful tool in economics, as they provide a solid background for a deeper understanding of economic phenomena.
By making a rough, yet potentially useful, simplification, one may say that the art of mathematical modeling lies between two opposing thrusts: on one hand the need to be closer to real phenomena, and on the other hand the need to retain some tractability of the model. This is particularly true when studying optimal control models in economics and social sciences due to the high complexity of phenomena arising from human decisions. In this context, a large portion of the related literature studies cases where the state variables are finite dimensional, i.e., their evolution is governed by Ordinary Differential Equations (ODEs) in $\R^n$.

However, when one deals with economic phenomena related to heterogeneity\footnote{For example, by heterogeneity in space, we mean that the economic variables do depend not only on time, but also on the spatial position.} or path-dependency of the state/control variables (capital, labor, consumption, investment, population, etc.), it becomes natural and somehow unavoidable, that the state equations become Partial Differential Equations (PDEs) or Differential Delay Equations (DDEs). As typical examples of this type we mention three cases below.
\begin{enumerate}[(i)]
\item 
Models of spatial growth (e.g.,
\cite{Brito04},
\cite{BoucekkineCamachoFabbri13},
\cite{boucekkine2019geographic}) or transboundary pollution
(e.g., \cite{boucekkine2022dynamic}, \cite{de2019spatial}), where the economic variables depend both on time $s$ and space $\theta$, and the state equation becomes a parabolic second order PDE. As an example, we may mention the model for pollution
\begin{equation}
\label{SE-transintro}
\begin{cases}
\displaystyle{\frac{\partial p}{\partial s}(s,\theta) = \frac{\partial}{\partial \theta} \bigg(\sigma(\theta) \frac{\partial p}{\partial \theta}  (s,\theta)\bigg) - \delta(\theta) p(s, \theta) + {\eta}(\theta) i(s,\theta),}
& \forall (s,\theta) \in \mathbb{R}^+ \times S^1, \\[10pt]
p(0,\theta) = p_0(\theta), & \forall \theta \in S^1,
\end{cases}
\end{equation}
where $p(s,\theta)$ denotes the pollution density at time $s$ and location $\theta$, $i(s,\theta)$ denotes the investment effort in economic activities (generating pollution), and 
$\sigma,\delta,\eta$ denotes given data depending on location with proper physical meaning (see the complete example in Subsection \ref{SSE:pollution}).

\medskip
\item Models with age structure, i.e., models where the economic variables depend on time $t$ and age $a$, such as optimal investment with vintage capital (e.g., \cite{BarucciGozzi2001}, \cite{Barucci1998}, \cite{faggian2010optimal}), models of forward interest rates (e.g.,
\cite{CarmonaTehranchi2006}), economic-epidemiological models (e.g., \cite{fabbri2021verification}).
The state equation here is a first order PDE of transport type. As an example, we may mention the model:
\begin{equation}
\label{eq:exstateintro}
\begin{cases}\displaystyle{\frac{\partial z(s, a)}{\partial s}+\frac{\partial z(s, a)}{\partial a}=u_1(s,a)},
& (s,a)\in(0,\infty)\times (0, \bar{a}], 
\\
z(s, 0)= u_0(s), & s \in(0,+\infty), 
\\
z(0, a)=z_0(a), & s \in[0, \bar{a}],\end{cases}    
\end{equation}
where $z(s,a)$ denotes the capital stock of age $a$ at time $s$, $u_0(s)$ denotes the investment in new capital goods  at time $s$ (a boundary control), $u_1(s,a)$ denotes the investment in goods of age $a$ at time $s$ (see the complete  example of Subsection \ref{sec:vcteexpl}.
\medskip
\item Models with path dependency, i.e., when the state/control variables at time $s$ depend on their values at times before $s$ (e.g., 
\cite{Boucekkine2002}, \cite{FabbriGozzi08}, 
\cite{federico2010hjb, federico2011hjb},
\cite{federico2014dynamic}). 
The state equation becomes a delay differential equation (DDE).
As an example, we may mention the model
\begin{equation}
\label{eqboucekkineintro} 
\left\{
\begin{array}{ll}
k'(s) = i(s) - i(s-T), &  s \geq 0, \\[3pt]
i(s) = i_0(s), & s \in [-T, 0), \\[3pt]
k(0) = \int_{-T}^0 i_0(s) \, {\mathrm{d}}s,
\end{array} 
\right.
\end{equation}
where $k(s)$ denotes the capital stock at time $s$, $i(s)$ denotes the investment in new capital goods at time $s$, $i_0$ represents the past pipeline stream of investments, and $T$ denotes the average duration of the investments.
\end{enumerate}





\subsection{Plan of the survey}
As previously mentioned, the main aim of this survey paper is twofold: \begin{enumerate} \item providing an introduction to Optimal Control Theory in infinite dimension, mainly viewed from the angle of the tools needed for its applications to economic models; 
\item developing various examples where this theory is successfully used to solve satisfactorily such models. \end{enumerate}
To accomplish carefully such tasks, one would need the space of an entire book.  Our choice is therefore to devote the core of the paper to the examples of application where explicit solutions of the HJB equations can be found. Indeed, while such cases are still very interesting for economic applications, the theory needed to develop them is more accessible to casual readers and can be exposed in a reasonable amount of page --- but still retaining the main features of infinite dimensional optimal control. 

Given the  considerations above, the plan of this survey paper is the following. 
\begin{enumerate}[--] 
\item Section \ref{sec:formulation} is devoted to present, in a concise and precise way, the formulation of Optimal Control Problems in infinite dimensional Hilbert spaces. Part of this section is devoted to explain how to deal with the case when the control operator is unbounded (typically arising in the case of boundary control or delay in the control), a case which is quite frequent in applications, even in our examples. It is divided in four subsections, three of them devoted to the main ingredients of our setting (state equation, admissible controls, objective functional) and the last one to the infinite horizon stationary problems with discount, which are the most frequent in economic applications.

\item In Section \ref{sec:DP}, we expose the basic theory on the Dynamic Programming approach to our family of problems in the case when classical solutions of the HJB equations are available. It is divided in three subsections; the first one is devoted to the HJB equation, the second one  to the verification theorem in the finite horizon case, and the last one to the verification theorem in the infinite horizon stationary case, including a sharp treatmente of the so-called transversality condition at infinity.

\item Section \ref{sec:non-linear} provides some ideas, on the qualitative side, on what can be done in the cases when classical solutions of the HJB are not available. It contains a subsection that illustrates  the concept of viscosity solution and another one that illustrates the concept of strong solution, both them used extensively in the literature.

\item Section \ref{sec:examples} presents some core examples arising in economic modeling, most of which are published in economic theory journals. All of them refer to cases where the explicit solution of the HJB equation can be found, hence solving completely the optimal control problem. They concern classical and novel  problems in the field of Economic Theory, such as spatial growth, environmental economics, vintage capital, time-to-build.

\item Section \ref{sec:exnonexpl} presents some other, more complex, examples arising in economic modeling, where explicit solutions of the HJB equation are not available. Here, we only give an idea of what one can do in applied examples like those.

\item Finally, Appendix \ref{sec:app} contains some basic material on operator and semigroups which are used throughout the paper. \end{enumerate}

\section{Formulation of optimal control problems in Hilbert spaces}\label{sec:formulation}




In this section, we give a precise formulation of a class of 
optimal control problems in Hilbert spaces which includes the delicate cases when the control operator is possibly unbounded (e.g., in the  cases of boundary control or delay in the control).


Regarding time, we consider two settings: the \emph{finite-horizon} case, where the time set is $[0,T]$ with $T\in(0,\infty)$, and the \emph{infinite-horizon} case, where the time set is $[0,\infty)$. 
To streamline notation, we will allow $T = \infty$, in which case $[0,T] = [0,\infty)$; the same convention applies to intervals of the form $[t,T]$.

\subsection{The state equation} 
\label{SSE:stateequation}
Assume that $X$ and $U$ are given real separable Hilbert  spaces, $X$ referred to as the \emph{state space} and $U$ as the \emph{control space}. When the context is clear, we denote by $\langle\cdot,\cdot\rangle$ the inner product and $|\cdot|$ for the norms on these spaces, specifying them with subscripts only when necessary. 
Given $t\in [0,T)$, $x\in X$, and\footnote{Note that, when $T<\infty$, we just have  $L^1_\loc([t,T];U)= L^1([t,T];U)$.} 
$u(\cdot)\in
L^1_\loc([t,T];U)$,
we consider the following controlled evolution equation (called \emph{state equation}) in the space $X$:
\begin{equation}\label{SE}
\begin{cases}
    y'(s)=\mathcal{A}y(s)
    +f(s,y(s),u(s)), \ \ \ s\in[t,T),\\
    y(t)=x,
\end{cases}
\end{equation}
 where 
$\mathcal{A}: D(\mathcal{A})\subseteq X\to X$ is a (possibly unbounded) linear operator;
$f:[0,T]\times X\times U\to X$ is a measurable function.
In many relevant examples one simply has $f(s,u,x)=\mathcal{B}u$, where $\mathcal{B}:U\to X$ is a linear operator.

We call the function $u(\cdot)$ the \emph{control} (or \emph{decision}) \emph{variable} and the function $y(\cdot)$ the \emph{state variable} of the system. 



We introduce the following assumptions for  $\mathcal{A}$ and $f$, which guarantee existence and uniqueness of a solution (in a proper sense) to \eqref{SE}, as  Proposition \ref{prop:mild} will show.

 \begin{Assumption}\label{ass:LU}
 \begin{enumerate}[(i)]
 \item[]
 \medskip
\item  $\mathcal{A}$ is a closed and densely defined linear operator on $X$ 
generating a $C_0$-semigroup of linear operators $e^{t\mathcal{A}}$ in $ \mathcal{L}(X)$. 
\item  $f:[0,T]\times X\times U\to X$ is measurable and Lipschitz-continuous with respect to $x$, uniformly in $(t,u)$, i.e., there exists $L>0$  such that 
$$
|f(t,x,u)-f(t,\hat x,u)|\leq L |x- \hat x|,  \ \ \ \forall t\in[0,T], \ \forall u\in U, \ \forall x,\hat x\in X.
$$
and 
$$|f(t,x,u)|\leq L(1+|x|+|u|), \ \ \ \ \forall t\in[0,T], \ \forall x\in X, \ \forall  u\in U.$$
\end{enumerate}
 \end{Assumption}

\noindent Note that, according to Proposition \ref{lem:Momega}, there exists constants $M>0$ and $\omega\in\mathbb R$ such that
\begin{equation}
    \label{eq:estsem}
\vert e^{t\A}\vert_{\mathcal{L}(X)}\le Me^{\omega t}, \quad \forall t\ge 0.
\end{equation}
  We rely on the concept of \emph{mild solution}\footnote{Clearly other notion of solution can be used, in particular strict, classical, strong and weak solutions, see e.g. \cite[Definition 3.1, p.129]{BDDM07} but we do not need to use them in the present context.} for the state equation a \eqref{SE}. 

  \begin{Definition}
  \label{sol:mild}
 Let $(t,x)\in[0,T)\times X$ and $u(\cdot)\in L^1_\loc([t,T];U)$. 
 We say that $y\in C( [t,T]; X)$ is 
 a \emph{mild solution} to \eqref{SE} if 
the integral equation 
 $$y(s)= e^{\mathcal{A}(s-t)} x
 +\int_t^s e^{(s-r)\mathcal{A}} f(r,y(r),u(r))\ud r
 $$
 holds for all $s\in[t,T]$. 
\end{Definition}
If needed, we point out the dependence of the solution on $t,x,u(\cdot)$ denoting it   by $y^{t,x,u(\cdot)}$.
\begin{Proposition}\label{prop:mild}
Let Assumption \ref{ass:LU} hold true. Then, for every $(t,x)\in [0,T)\times X$ and every  $u(\cdot)\in L^1_\loc([t,T];U)$, 
the state equation \eqref{SE} admits a unique mild solution $y(\cdot)$. The latter satisfies, for some $C:=C(L,\omega,M)>0$, 
\begin{align}
&|y(s)|\leq  C(1+|x|+|u(\cdot)|_{L^1([t,s];U)} )e^{C(s-t)},  \label{estimate-mild1}\\
& |y(s)-x|\leq  C(1+|x|+|u(\cdot)|_{L^1([t,s];U)} )\left(e^{C(s-t)}-1\right). \label{estimate-mild2}
\end{align}

Moreover, the  mild solution $y(\cdot):=y^{t,x,u(\cdot)}(\cdot)$ is also the unique weak solution (in the sense of \cite[Def.\,5.1\emph{(v)},\,p.\,129]{BDDM07}) to \eqref{SE}.
    \end{Proposition}
    \begin{proof}
This is a straightforward application of known results, see e.g. \cite[Chapter 6, Theorem 1.2]{pazy1983semigroups} or \cite[Ch 2, Proposition 5.3]{li2012optimal} for the existence and uniqueness of mild solutions and for the estimates \eqref{estimate-mild1}--\eqref{estimate-mild2}.


The proof of existence and uniqueness is based on the classical contraction fixed point theorem in the space $C([t,T];X)$, possibly endowed with an auxiliary norm. Moreover the estimates
\eqref{estimate-mild1}--\eqref{estimate-mild2} follow from an application of Gronwall's Lemma.
    \end{proof}
\noindent In view of Proposition \ref{prop:mild}, Assumptions \ref{ass:LU} will be in force from now on, and will not be stated again.

We now need a chain rule for mild solution to our state equation, which would apply to sufficiently smooth functions. Precisely, given $\mathcal{O}\subseteq X$ open and $t\in[0,T)$, let
$$
\mathcal{S}([t,T)\times \mathcal{O}):=\big\{\varphi:[t,T)\times \mathcal{O}\to\mathbb{R}: \  \varphi\in C^1([t,T)\times {\mathcal{O}}) \ \mbox{and} \ D\varphi\in C([t,T)\times \mathcal{O};D(\mathcal{A}^*))\big\},$$
and 
\begin{equation}\label{S0}
\mathcal{S}(\mathcal{O}):=\Big\{\varphi: \mathcal{O}\to \mathbb{R}: \ \varphi\in C^1(\mathcal{O}) \ \mbox{and} \ D\varphi\in C(\mathcal{O};D(\mathcal{A}^*))\Big\}.
\end{equation}
where ${D}(\mathcal{A}^*)$ is endowed with the graph norm.
We have the following.
\begin{Proposition}\label{prop:chain}
  The following chain's rules hold true.
 \begin{enumerate}[(i)]
 \item  Let $T<\infty$, 
 $(t,x)\in[0,T)\times X$, $u(\cdot)\in L^1([t,T];U)$.  Set  $y(\cdot):=y^{t,x,u(\cdot)}(\cdot)$, and assume that
 $y(s)\in\mathcal{O}$ for every $s\in[t,T]$. Then,
for every $\varphi\in \mathcal{S}([t,T)\times \mathcal{O})$ and $r\in[t,T]$
\begin{align}
    \label{chain}
    \varphi (r, y(r))&= \varphi(t,x) + \int_t^r \Big(\varphi_t(s,y(s))+\langle y(s),\mathcal{A}^*D\varphi(s,y(s))\rangle
    +\langle f(s,y(s),u(s)),
    D\varphi(s,y(s))
    \rangle \Big)\ud s.
\end{align}
\item \emph{(Stationary case)} Let $T=\infty$, 
$f(s,y,u)=f_0(y,u)$, $x\in X$, $u(\cdot)\in L^1_\loc([0,\infty);U)$.   Set  $y(\cdot):=y^{0,x,u(\cdot)}(\cdot)$, and assume that
 $y(s)\in\mathcal{O}$ for every $s\in[0,\infty)$. Then, for every $\varphi\in \mathcal{S}(\mathcal{O})$ and $r\in[0,\infty)$, 
\begin{align}
    \label{chainst}
    e^{-\rho r}\varphi (y(r))= \varphi(x) + \int_0^r e^{-\rho s}\Big(-\rho  \varphi(y(s))+\langle y(s),\mathcal{A}^*D\varphi(y(s))\rangle +
    \langle
    f_{0}(y(s),u(s)),
    D\varphi(y(s))\rangle \Big)\ud s.
\end{align}
\end{enumerate}
    \end{Proposition}
    \begin{proof}
     See, e.g., \citealp[Ch.2, Proposition 5.5]{li2012optimal}.
    \end{proof}

\subsubsection{The case of unbounded control operator}
\label{SSSE:Bunbounded}

We now look at the cases, which often arise in applied economic modeling, when the control function $u(\cdot)$ appears in the right hand side of the state equation through a linear term 
$\mathcal{B}u(\cdot)$\footnote{This does not exclude that the control also appears in possibly nonlinear map $f$ as in the previous subsection, see the example of Section \ref{sec:vcteexpl}.} where the operator $\mathcal{B}$ (which we will name the ``control operator" from now on) is unbounded in the sense that it is continuous from $U$ into 
$\overline{X}$ where $\overline{X}$ is a Hilbert space such that $X\subsetneq \overline{X}$ with a continuous inclusion.
Our aim here is to explain how to extend the above theory to this case.
Before to proceed with this task, we provide a concrete example where the need of considering unbounded control operators arises.




\begin{Example} We consider for given initial time $t\ge 0$ and for given maximum age $\overline{a}>0$, a controlled system of the following type (which a special case of the one of Section \ref{sec:vcteexpl}) \begin{equation}
\label{eq:exstate}
\begin{cases}\displaystyle{\frac{\partial z(s, a)}{\partial s}+\frac{\partial z(s, a)}{\partial a}=0},
& (s,a)\in(t,\infty)\times (0, \bar{a}], 
\\\\ 
z(s, 0)= u(s) & s \in(t,+\infty), 
\\\\
z(t, a)=x(a) & s \in[0, \bar{a}],\end{cases}    
\end{equation}
where $u(\cdot)\in L^1_{loc}(t,+\infty;\R)$ is a given boundary control (representing the investments in new capital goods in the example of Section \ref{sec:vcteexpl}).
   A natural reformulation in abstract terms is given here below choosing, as state space, the separable Hilbert space $X=L^2(0,\bar a;\mathbb R)$, as control space $U=\R$, and calling $y(s)$ the map  $z(s,\cdot):[t,T]\to X$: 
\begin{equation}
\label{eq:exstateriscritta}
\begin{cases}    y'(s)=\mathcal{A}y(s)+\mathcal{B}u(s), \ \ \ s\in[t,T),\\
    y(t)=x,
\end{cases}
\end{equation}
Here the operator $\mathcal{A}$
is the first order operator (with zero boundary Dirichlet condition) defined in Example \ref{ex:trsem} which generates the so called translation semigroup. Moreover the control operator $\mathcal{B}$ is, formally, written as
$$\mathcal B u=u\delta_0$$
where $\delta_0$ is the Dirac delta in $r=0$.
The above \eqref{eq:exstateriscritta}
is obtained following the approach developped, e.g., in
\cite[Section 13.2]{DaPratoZabczyk96}(see also \cite[Ch.9, Section 1.1]{li2012optimal}. Heuristically speaking, we can say that $\mathcal{B}=\mathcal{A}\mathcal{N}$ where $\mathcal{N}$ is the Dirichlet type operator $\mathcal{N}:\R\to X$
defined as $\mathcal{N}u=w$ iff $w(0)=u$ and $w'(a)=0$ for all $r\in [0,\overline{a}]$.

Clearly, $\mathcal B$ is unbounded in the sense that it takes its values out of the naturally chosen state space $X=L^2(0,\bar a;\mathbb R)$. More precisely $\mathcal{B}$ is a continuous linear operator if we take as arrival space 
$C([0,\bar{a}];\R)^\prime$ (i.e. the space of Radon measures on $[0,\overline{a}]$), but also if the arrival space is $H^{-1}(0,\bar{a};\R)$ or $D(\mathcal A^*)'$.
 \qed
\end{Example}

In the framework and notation given at the beginning of  the Subsection
\ref{SSE:stateequation}
we now consider the following state equation in the space $X$:
\begin{equation}\label{SEB}
\begin{cases}
    y'(s)=\mathcal{A}y(s)
 +   \mathcal{B}u(s)
    +f(s,y(s),u(s)), \ \ \ s\in[t,T),\\
    y(t)=x,
\end{cases}
\end{equation}
 where 
$\mathcal{A}$ and $f$ are as in equation \eqref{SE} 
while $\mathcal{B}$ is a continuous linear operator from the control space $U$ to another Hilbert space $\overline{X}$, with $X\subsetneq \overline{X}$ with continuous inclusion.
In many relevant examples one simply has $f(s,u,x)=0$.\footnote{There are also examples where the unbounded term in the state equation is non linear. This case can be treated with similar ideas, see e.g. \cite[Chapter 7]{lunardi1995analytic}.}

Our main focus here is to give sense to the notion of mild solution to \eqref{SEB} as given in Definition \ref{sol:mild}, i.e. the solution, in
$C([t,T];X)$ of the following integral equation
\begin{equation}
\label{eq:mildBunbounded}
y(s)= e^{\mathcal{A}(s-t)} x
 + \int_t^s e^{(s-t)\mathcal{A}} \mathcal{B}u(r)dr 
  +\int_t^s e^{(s-r)\mathcal{A}} f(r,y(r),u(r))\ud r
\end{equation}
Here the difference with respect to Definition \ref{sol:mild} is the presence of the term $\int_t^s e^{(s-t)\mathcal{A}} \mathcal{B}u(r)dr$. 
To give sense to the above term, and consequently to the mild solution equation we have two natural ways which are of course intertwined. One is to consider such integral equation in $X$ and prove that, under suitable assumptions, the above convolution term is well defined and lives to $X$; another is to extend the state space $X$ to the larger space 
$\overline{X}$ (extending also the semigroup $e^{t\mathcal{A}}$ on $\overline{X}$) where the control term is well defined and prove the existence of solution in that space.

We start first showing how to perform the extension of the semigroup.\footnote{This can be seen as a particular case of the notion of extrapolation space and Sobolev towers, treated, e.g. in 
\cite[Section II.5]{engel-nagel}; see also \cite[Section II.3.1]{BDDM07}.}
First of all, we recall, along Definition
\ref{df:adjointsemigroup} and Corollary
\ref{cr:adjointHilbert}, that the adjoint
$\mathcal{A}^*$ generates the adjoint semigroup $e^{tA^*}=(e^{tA})^*$ which is still a strongly continuous semigroup in $X$.
We consider  the subspace $D(\A^*)$ of $X$ endowed with the graph norm, and its dual space $D(\A^*)'$. Clearly we have the natural Gelfand triple
$D(\A^*)\subset X\equiv X'\subset D(\A^*)'$.\footnote{We recall that  we cannot identify contemporarily a proper subspace $V$ of $X$ with its dual $V'$.}
The inclusions are proper when $\A^*$ is not bounded.
If we endow $D(\A^*)$ with the graph norm, then $\A^*$ is a continuous linear operator from $D(\A^*)$ into $X$.
Since changing the topological structure of the spaces we change indeed the operator itself we call $\A^*_1$ such ``version'' of $\A^*$ and write
$\A^*_1\in \mathcal{L}(D(\A^*),X)$.

Now we take the adjoint of $\A_1^*$.  Using \eqref{eq:defadjointdomain}, we see that it must be, since we identify $X$ with its dual $X'$,
$$
(\A^*_1)^*\in \mathcal{L}(X,D(A^*)').
$$
We call $\mathcal{A}_{-1}$ this operator.
Using that $\mathcal{A}^{**}=\mathcal{A}$ (see e.g. \cite[Theorem VIII.1]{Yosida1980} or \cite[Theorem I.3.11]{ReedSimon1980}) and the theory developed in \cite[Section II.5.1, Theorem II.5.5]{engel-nagel} (see also \cite[Section II.3.1]{BDDM07}, one can easily prove that 
$$
\A_{-1}: D(\A_{-1})=X \subset D(\A^*)' \to D(\A^*)'
$$
is an extension of the operator
$\mathcal{A}$. Moreover, still using
\cite[Theorem II.5.5]{engel-nagel}
we can prove that
$\A^*_{-1}$ generates the strongly continuous semigroup 
$e^{t\A_{-1}}$ on $D(\A^*)'$ which
extends the semigroup $e^{t\A}$, i.e. $e^{t\A_{-1}}|_X = e^{t\A}$ for all $t \ge 0$.

\label{eq:domainA1}

Now we show how to extend the previous setting and results, namely Propositions \ref{prop:mild} and \ref{prop:chain} to the present setting. In the following,  we intend $D(\mathcal{A})$ endowed with the Hilbert structure induced by the inner product $\langle x,y\rangle_{D(\mathcal{A})}:=\langle x,y\rangle_{X}+\langle \A x,\A y\rangle_X$ and $D(\mathcal{A})'$ endowed with the Hilbert structure provided to it as dual of the Hilbert space $D(\mathcal{A})$.

\begin{Assumption}
    \label{ass:Bunbounded}
\begin{itemize}
\item[]
\item[(i)]
$\mathcal{B}\in\mathcal{L} (U;D(\A^*)')$;
    \item[(ii)] 
For each $t\in[0,T)$ and $R\in[t,T]$, the convolution
$$
(\mathcal{S}\mathcal{B}): [t,R] \to X, \ \ r\mapsto (\mathcal{S}\mathcal{B})(r):=\int_t^r e^{(s-t)\mathcal{A}} \mathcal{B}u(s) d s
$$
belongs to 
$ 
\mathcal{L}(L^1([t,R];U); C([t,R];X)).$
Moreover, there exists $C_{\mathcal{S}\mathcal{B}}>0$ such that 
\begin{equation}\label{stimaB}
\|\mathcal{S}\mathcal{B}\|_{\mathcal{L}(L^1([t,R];U); C([t,R];X))}\leq C_{\mathcal{S}\mathcal{B}} R.
\end{equation}
\end{itemize}
\end{Assumption}
Now we observe that, under Assumption
\ref{ass:Bunbounded}-(i), we can directly apply the above Proposition \ref{prop:mild} to the state equation \ref{SEB} simply substituting $X$ with $D(\A^*)'$. This provides well posedness of the state equation \ref{SEB} for all initial data $x\in D(\A^*)'$.

However, in many case this may not be satisfactory, as often one would like to know if and in which cases the mild solution still lies in $X$.
This is true if we also add
Assumption
\ref{ass:Bunbounded}-(ii) since this guarantee that the right hand side of 
\eqref{eq:mildBunbounded} lies in $X$ and that the contraction mapping theorem can be applied in $X$ itself.

Hence we have the following

\begin{Proposition}
\label{prop:mildunbounded}
Let Assumptions \ref{ass:LU} and 
\ref{ass:Bunbounded}\emph{(i)} hold true. Then, for every $(t,x)\in [0,T)\times D(\A^*)'$ and every  $u(\cdot)\in L^1_\loc([t,T];U)$, 
the state equation \eqref{SE} admits a unique mild solution $y(\cdot)$ in the sense of Definition \ref{sol:mild} where $X$ must be substituted with $D(\A^*)'$. Such a solution satisfies, for some $C:=C(L,\omega,M)>0$, 
\begin{align}
&|y(s)|_{D(\A^*)'}\leq  C(1+|x|_{D(\A^*)'}+|u(\cdot)|_{L^1([t,s];U)} )e^{C(s-t)},  \label{estimate-mild1unbounded}\\
& |y(s)-x|_{D(\A^*)'}\leq
C\left(1+|x|_{D(\A^*)'}+
|u(\cdot)|_{L^1([t,s];U)} \right)
\left(e^{C(s-t)}-1\right). \label{estimate-mild2unbounded}
\end{align}
Moreover, the  mild solution $y(\cdot):=y^{t,x,u(\cdot)}(\cdot)$ is also the unique weak solution (in the sense of \cite[Definition 5.1(v) p.129]{BDDM07}) to \eqref{SE}.

Finally, if also Assumption \ref{ass:Bunbounded}-\emph{(ii)} holds, then
the above solution also belongs to
$C([t,T];X)$ for every initial datum $x\in X$ and the estimates \eqref{estimate-mild1} and \eqref{estimate-mild2} hold.
    \end{Proposition}
    \begin{proof}
The result, except the last statement is exactly an application of the previous Proposition \ref{prop:mild} when $X$ is substituted by $D(\A^*)'$, thanks to Assumption \ref{ass:Bunbounded}-(i). 
The last statement follows applying the Contraction Mapping Principle to the integral equation \eqref{eq:mildBunbounded} in the space $C([t,T];X)$, which is made possible by 
Assumption \ref{ass:Bunbounded}-(ii). 
    \end{proof}

Now we state and prove the analogous of
Proposition \ref{prop:chain}.

\begin{Proposition}\label{prop:chainunbounded}
Let Assumptions \ref{ass:LU} and 
\ref{ass:Bunbounded} hold true.
The following chain's rule holds.
 \begin{enumerate}[(i)]
 \item Let $T<\infty$, 
 $(t,x)\in[0,T)\times X$, $u(\cdot)\in\ L^1([t,T];U)$.  Set  $y(\cdot):=y^{t,x,u(\cdot)}(\cdot)$, and assume that
 $y(s)\in\mathcal{O}$ for every $s\in[t,T]$. Then,
for every $\varphi\in \mathcal{S}([t,T)\times \mathcal{O})$ and $r\in[t,T]$
\begin{align}
    \label{chainunbounded}
    \varphi (r, y(r))&= \varphi(t,x) + \int_t^r \Big(\varphi_t(s,y(s))+\langle y(s),\mathcal{A}^*D\varphi(s,y(s))\rangle
    \\
    \nonumber
    &+   \langle u(s), \mathcal{B}^* D\varphi(s,y(s))\rangle
    +\langle 
    f(s,y(s),u(s)),
    D\varphi(s,y(s))\rangle \Big)\ud s.
\end{align}
\item  Let $T=\infty$, 
$f(s,y,u)=f_0(y,u)$, $x\in X$, $u(\cdot)\in L^1_\loc([0,\infty);U)$.   Set  $y(\cdot):=y^{0,x,u(\cdot)}(\cdot)$, and assume that
 $y(s)\in\mathcal{O}$ for every $s\in[0,\infty)$. Then, for every $\varphi\in \mathcal{S}(\mathcal{O})$ and $r\in[0,\infty)$, 
\begin{align}
    \label{chainstunbounded}
    \varphi (y(r))&= \varphi(x) + \int_0^r e^{-\rho s}\Big(\rho  y(s)+\langle y(s),\mathcal{A}^*D\varphi(y(s))\rangle 
    \\
    \nonumber
    &+   \langle u(s), \mathcal{B}^* D\varphi(y(s))\rangle
     +\langle 
   f_{0}(y(s),u(s)),
   D\varphi(y(s))\rangle \Big)\ud s.
\end{align}
\end{enumerate}
    \end{Proposition}
    \begin{proof}
     The proof goes exactly as the on of Proposition \eqref{chain} (see, e.g., \citealp[Ch.2, Proposition 5.5]{li2012optimal}) once we use that, thanks to Assumption \ref{ass:Bunbounded}-(ii), the mild solution of \eqref{SEB} lies in $X$ and that, thanks to Assumption \ref{ass:Bunbounded}-(i), the term 
     $\mathcal{B}^* D\varphi$ make sense and is continuous.
    \end{proof}

    \subsection{The set of admissible controls}\label{sec:adm} 
The set of admissible controls is chosen to take account of the structure of the specific problem and to avoid useless technical problem like non well-posedness of the state equation or infinite integrals in the objective functional. 
This means the typically such a set is strictly smaller than    
$L^1_\loc([t,T];U)$.
To take care of control and state constraints, we consider 
two measurable sets: $\mathcal{C} \subseteq  U$ (control constraint)
$\mathcal{D} \subseteq  X$ (state constraint).
Then we define  
$$
\widehat{\mathcal{U}}(t,x):= \Big\{ u(\cdot)\in  L^1_\loc([t,T];\mathcal{C}): \ y^{t,x,u(\cdot)}(s)\in \mathcal{D}, \ \forall s\in[t,T]\Big\}.  
$$
Clearly, if $x\notin \mathcal{D}$, then the above set is empty. 
The set of admissible controls
$\mathcal{U}(t,x)$ in each specific problem will be a (not necessarily proper) subset of
$\widehat{\mathcal{U}}(t,x)$, which taking account of further restrictions on the control strategies, like stronger integrability properties (e.g 
$u(\cdot)\in  L^p_\loc([t,T];\mathcal{C})$ for some $p>1$).

    
\subsection{The objective functional}
We introduce now the objective functional of our family of problems. Differently from what happens in many books on optimal control in the mathematical literature, here the objective functional has to be \textbf{maximized} and not minimized over the set of admissible controls. We choose to do this since this is the most common attitude used in the economic literature to which our examples refer to.
\begin{itemize}
    \item[($T<\infty$)] 
     Let $(t,x)\in[0,T)\times X$ and  $u(\cdot)\in \mathcal{U}(t,x)$. Given measurable functions $g:[0,T]\times \mathcal{D}\times \mathcal{C}\to\R$ and $\phi:\mathcal{D}\to \R$, the objective functional to be maximized is 
    $$\mathcal{J}(t,x;u(\cdot)):=\int_t^T g\Big(s,y^{t,x,u(\cdot)}(s),u(s)\Big)\ud s+\phi\big(y^{t,x,u(\cdot)}(T)\big).$$

    \item[($T=\infty$)]  Let $(t,x)\in \mathbb{R}^+\times \mathcal{D}$, $u(\cdot)\in \mathcal{U}(t,x)$.
    Given a measurable function $g:\mathbb{R}^+\times \mathcal{D}\times \mathcal{C}\to\R$, the objective functional to be maximized is 
    $$\mathcal{J}(t,x;u(\cdot)):=\int_t^\infty g\left(s,y^{t,x,u(\cdot)}(s),u(s)\right)\ud s.$$
\end{itemize}
In both cases, we define the value function as 
$$
V(t,x):=\sup_{u(\cdot)\in\mathcal{U}(t,x)} \mathcal{J}(t,x;u(\cdot)).
$$
In the case $T<\infty$, we set 
$$
V(T,x):=\phi(x).
$$

The above definition of $\mathcal{J}$ may raise well-posedness issues, especially when $T = \infty$ and the functional is given by an improper integral. 
Depending on the problem, one may require stronger integrability conditions on the functions in $\mathcal{U}(t,x)$ depending on the growth of $g(s,\cdot,\cdot)$ and
$\phi$ in the finite-horizon case;  on the growth and the asymptotic behavior of $g(s,\cdot,\cdot)$ as $s \to \infty$ in the infinite-horizon case.  
To avoid cumbersome formulations and overly restrictive conditions that may be unsuitable in relevant applications, 
we simply \emph{assume} that the problem is well-posed and finite.\footnote{
Typically, dealing with concrete problems, one performs a preliminary analysis and identifies sets of admissible controls in such a way that such assumption is verified.}

\begin{Assumption}\label{ass:wellpose} 
\begin{enumerate}[(i)] 
\item[]
\item[]
\vspace{-.3cm}
\item The family  $\{\mathcal{U}(t,x)\}_{(t,x)\in [0,T)\times \mathcal{D}}$ satisfies the following two properties:
\smallskip
    \begin{enumerate}[(P1)]
 \item  For every $s \in [t,T)$, $x\in \mathcal{D}$, we have 
$$
u(\cdot )\in {\mathcal{U}}\left(t,x\right)\  \Longrightarrow \ u(\cdot)|_{[s,T]}
\in {\mathcal{U}}( s,y^{t,x,u(\cdot)}(s)).$$

\smallskip
\item  For every $s \in[t,T)$, $x \in \mathcal{D}$, let
$$
u_{1}(\cdot)\in {\mathcal{U}}\left( t,x\right) ,\quad u_{2}(\cdot)\in {\mathcal{U}}(
s, y^{t,x,u_1(\cdot)}\left(s\right)),
$$
and set
$$
u(r) :=\begin{cases}
u_{1}( r),  \  \mbox{if}\quad r\in [t,s),  \\
u_{2}( r),  \ \mbox{if}\quad r\in [s,T],
\end{cases}
$$
Then,  $u(\cdot)\in {\mathcal {U}}\left( t,x\right).$
\end{enumerate}
\smallskip
\item  $\mathcal{J}(t,x;u(\cdot))$ is well defined and finite for every $(t,x)\in [0,T)\times \mathcal{D}$ and $u(\cdot)\in \mathcal{U}(t,x)$.
\smallskip
\item The value function $V(t,x)$ is finite for every $(t,x)\in [0,T)\times \mathcal{D}$.  
\end{enumerate}
\end{Assumption}
\smallskip
\begin{Remark}
    Roughly speaking (P1) can be rewritten saying that the tail of an admissible control is itself admissible; similarly, (P2) can be rephrased saying that the concatenation of two admissible controls  yields another admissible control. As we will show, these two properties are essential for the time-consistency of the control problems under exam, enabling the use of Dynamic Programming.\hfill$\square$
\end{Remark}


\subsection{The infinite horizon stationary case}\label{sec:infer}
An important class of infinite horizon problem arising in many application is when the system is autonomous, which allow to reduce the problem by getting rid of the time variable. This is the case when 
$$
f(s,x,u)=f_0(x,u), \ \ g(s,x,u)=e^{-\rho s}g_0(x,u),
$$
where $\rho>0$ is a discount factor.
In view of the structure of $f$ and by (P1), the  sets of admissible controls $\big\{\mathcal{U}(t,x)\big\}_{(t,x)\in [0,T)\times \mathcal{D}}$ are readily seen to verify the following property:
$$
u(\cdot)\in \mathcal{U}(0,x) \ \Longleftrightarrow \ u_t(\cdot):= u(\cdot-t)\in \mathcal{U}(t,x).
$$
Then,  it is straightforward to  see that 
$$
\mathcal{J}(t,x;u_t(\cdot))= e^{-\rho t} \mathcal{J}(0,x;u(\cdot)),
$$
which leads to 
$$
V(t,x)=e^{-\rho t} V(0,x).
$$
In this case, with abuse of notation we denote, 
$$y^{x,u(\cdot)}= y^{0,x,u(\cdot)}, \ \ \ \mathcal{U}(x)=\mathcal{U}(0,x),  \ \ \ \ \mathcal{J}(x;u(\cdot))=\mathcal{J}(0,x;u(\cdot)), \ \ \  V(x)=V(0,x).$$ 
\smallskip
\begin{Remark} To ensure conditions (ii) and (iii) of Assumption \ref{ass:wellpose} (with $t=0$), it is typically necessary that the discount factor be sufficiently large relative to the growth of $f_0$ and $g_0$.
\end{Remark}

\section{Dynamic Programming in the smooth case: classical solutions and verification}\label{sec:DP}

In this section, we develop the main tools of the Dynamic Programming approach for our optimal control problem. 
We will establish the connection between the value function $V$ and the Hamilton–Jacobi–Bellman (HJB) equation, and present verification theorems for classical solutions of the HJB equation.
\textbf{Assumption \ref{ass:wellpose} will be standing in the rest of the section.}

\subsection{Dynamic Programming Principle and Hamilton-Jacobi-Bellman equation}
The Dynamic Programming method to approach optimal control problems relies upon the so called \emph{Dynamic Programming Principle} (DPP), which can be stated as follows.
\begin{Proposition}[Dynamic Programming Principle]
    \label{DPP}
For each $(t,x)\in [0,T)\times \mathcal{D}$ and each $r\in[t,T]$, we have 
\begin{equation}\label{eq:DPP}
V(t,x)=\sup_{u(\cdot)\in \mathcal{U}(t,x)} \left[\int_t^{r} g(s,y(s),u(s))\ud s+V(r,y(r))\right].
\end{equation}

\end{Proposition}

\begin{proof}
    The proof does not suffer the dimension of the state space and relies, as in the finite dimensional case, on the time consistency of the problem, i.e. on the properties (P1)--(P2) of the sets of admissible controls $\mathcal{U}(t,x)$ for $(t,x)\in [0,T)\times \mathcal{D}$ (Assumption \ref{ass:wellpose}(i)). We omit it and refer to \cite{li2012optimal}.
\end{proof}
The formal passage to the limit of DPP (putting everything on one side, dividing by $r-t$ and sending $r\to t^+$) leads to consider the following PDE (Hamilton-Jacobi-Bellman, HJB)  associated to $V$:
\begin{equation}\label{eq:HJB}
    -v_t(t,x)=\langle \mathcal{A}x,Dv(t,x)\rangle+\mathcal{H}(t,x,Dv(t,x)), \ \ (t,x)\in [0,T)\times \mathcal{D},
\end{equation}
where 
\begin{equation}\label{Hsup}
\mathcal{H}(t,x,p):=\sup_{u\in \mathcal{C}} \mathcal{H}_{cv} (t,x,p;u), \ \ \  (t,x)\in [0,T)\times \mathcal{D}, \ p\in X,
\end{equation}
and 
$$
\mathcal{H}_{cv} (t,x,p;u):= \big\langle  f(t,x,u),\, p\big\rangle+g(t,x,u), \ \ \  \ \ \  (t,x)\in [0,T)\times \mathcal{D}, \ p\in X, \ u\in \mathcal{C}.
$$
Note that, in \eqref{eq:HJB}, the term $\langle \mathcal{A}x,Dv(t,x)\rangle$ is  defined only when $x\in\mathcal{A}$. To overcome this issue, one may deal with the following concept of solution.\footnote{Of course other concepts of solution are possible, see Section \ref{sec:non-linear} for an introduction to them.}

\begin{Definition}
Let
$\mathcal{O}:=\mbox{Int }\mathcal{D}$.
We say that a function $v\in \mathcal{S}([0,T)\times \mathcal{O})$ solves \emph{HJB} \eqref{eq:HJB} in classical sense at $(t,x)\in[0,T)\times
\mathcal{O}$ if 
\begin{equation}\label{eq:HJBbis}
    -v_t(t,x)=\langle x,\mathcal{A}^* Dv(t,x)\rangle+\mathcal{H}(t,x,Dv(t,x)).
\end{equation}
\end{Definition}
We provide the link between $V$ and HJB \eqref{eq:HJB}.


\begin{Theorem}\label{th:solreg}
    Let
$\mathcal{O}:=\mbox{Int }\mathcal{D}$ and assume that $\mathcal C$ is compact and that $V\in \mathcal{S}([0,T)\times\mathcal{O})$. Then, $V$ solves \emph{HJB} \eqref{eq:HJB}  in classical sense at each $(t,x)\in [0,T)\times \mathcal{O}$.
\end{Theorem}
\begin{proof}
    \emph{Step 1.} Let $(t,x)\in [0,T)\times \mathcal{O}$. We are going to show that for every $\bar u\in \mathcal{C}$,
    \begin{equation}\label{V1}
    -V_t(t,x)\geq \langle x, \mathcal{A}^* DV(t,x)\rangle+ \mathcal{H}_{cv} (t,x,DV(t,x),\bar u). 
    \end{equation}
    By arbitrariness of $\bar u$, this entails one inequality to be proved, that is 
    $$ -V_t(t,x)\geq \langle x,\mathcal{A}^* DV(t,x)\rangle+\mathcal{H}(t,x,DV(t,x)).$$
    Let  $\gamma>0$ be such that $B(x,\gamma)\subseteq \mathcal{O}$,
    fix $\bar u\in \mathcal{C}$, let $\bar y(\cdot):=y^{t,x,u(\cdot)\equiv \bar u}$, and  set 
    $$\bar r:=\inf\{s\in[t,T): y(s)\notin B(x,\gamma)\}\wedge T > t.$$
By  using \eqref{eq:DPP} and \eqref{chain}, we get, for every constant $r\in(t,\bar{r})$, 
\begin{align*}
V(t,x) & \geq  \int_t^r g(s,\bar y(s),\bar  u)\ud s+V(r,\bar y(r))\\
& =  \int_t^r g(s,\bar y(s),\bar  u)\ud s+
V(t,x)\\& \ \ \ \ \ + \int_t^r \Big(V_t(s,\bar y(s))+\langle \bar y(s),\mathcal{A}^*DV(s,\bar{y}(s))\rangle+ \langle DV(s,\bar{y}(s)),\,f(s,\bar{y}(s),\bar{u})\rangle \Big)\ud s.
\end{align*}
Diving by $r-t$, letting $r\to t^+$,   considering that $\bar y(s)\to x$ as $ t\leq s\leq r\to t^+$, and considering the continuity properties of the functions involved in the integral, we get by mean integral theorem \eqref{V1}.

    \emph{Step 2.} 
   We are going to show that 
   \begin{equation}\label{V2}
 -V_t(t,x)\leq \langle x,\mathcal{A}^* DV(t,x)\rangle+\mathcal{H}(t,x,DV(t,x)).
 \end{equation}
 Assume, by contradiction, that 
    \begin{equation}\label{V2bis}
 -V_t(t,x)> \langle x,\mathcal{A}^* DV(t,x)\rangle+\mathcal{H}(t,x,DV(t,x)).
 \end{equation}
Since $U$ is compact, the Hamiltonian $\mathcal{H}$ is (finite  and) continuous.
 By continuity of $\mathcal{H}$  and by the smoothness properties of $V$, there exist $\delta_o,\varepsilon_o>0$ such that 
  \begin{align}\label{V2tris}
0&\geq \delta_o +  V_t(s,y)+ \langle y,\mathcal{A}^* DV(s,y)\rangle+\mathcal{H}(s,y,DV(s,y))\\
 & \geq \delta_o  +V_t(s,y)+ \langle y,\mathcal{A}^* DV(s,y)\rangle+\mathcal{H}_{cv}(s,y,DV(s,y);u), \ \ \ \forall (s,y)\in[t,t+\varepsilon_o]\times B(x,\varepsilon_o), \ \forall u\in U. \nonumber 
 \end{align}
 Now, take $\varepsilon\in (0,\varepsilon_o)$, let  $u(\cdot)\in\mathcal{U}(t,x)$ be arbitrary,   and  set $y(s):=y^{t,x,u(\cdot)}(s)$ and 
 $$
r_\varepsilon:=\inf\big\{s\geq t: \ y(s)\notin B(x,\varepsilon)\big\}\wedge T.
 $$
 Clearly, $r_\varepsilon$ depends on $u(\cdot)$. However, due to the fact that $\mathcal{C}$ is compact (hence bounded) and due to \eqref{estimate-mild2}, we have the existence of $\alpha_o\in (0,1)$ such that 
 \begin{equation}\label{astbouona}
 r_\varepsilon \geq t+\alpha_o\varepsilon, \ \ \ \forall u(\cdot)\in \mathcal{U}(t,x).
 \end{equation}
 Hence, due to \eqref{V2tris}, 
 we have 
 \begin{align*}
     & V(t,x)- V\big(t+\alpha_o\varepsilon,\,y(t+\alpha_o\varepsilon)\big)\\
         &=-\int_t^{t+\alpha_o\varepsilon} \Big(V_t(s,y(s))+ \langle y,\mathcal{A}^* DV(s,y(s))\rangle+f(s,y(s);u(s))\Big)\ud s\\
     &=-\int_t^{t+\alpha_o\varepsilon} \Big(V_t(s,y(s))+ \langle y,\mathcal{A}^* DV(s,y(s))\rangle+\mathcal{H}_{cv}(s,y,DV(s,y(s));u(s))-g(s,y(s),u(s))\Big)\ud s\\&
     \geq \delta_o \alpha_o \varepsilon+\int_t^{t+\alpha_o\varepsilon} g(s,y(s),u(s))\ud s.
 \end{align*}
 On the other hand, by DPP, for every $\varepsilon>0$, we may take  $u_\varepsilon(\cdot)\in\mathcal{U}(t,x)$  such that, defining  $y_\varepsilon$  accordingly, we have  
 $$
 V(t,x)- V(t+\alpha_o\varepsilon, y(t+\alpha_o\varepsilon))\leq \int_t^{t+\alpha\varepsilon} g(s,y_\varepsilon(s), u_\varepsilon(s))\ud s+\varepsilon^2.
 $$
Combining the two inequalities above,
we get $\varepsilon ^2\geq \delta_o\alpha_o\varepsilon$. Since $\varepsilon$ was arbitrary, we get a contradiction and conclude.
\end{proof}
\begin{Remark}\label{rem:assV1}
We now comment on the assumption of compactness of $\mathcal{C}$ in the result above. A closer examination of the proof reveals that this assumption plays a role only in the second part (the proof of the subsolution property), specifically in ensuring:  
(i) the continuity of \( \mathcal{H} \);  
(ii) the uniform bound \eqref{astbouona}.

However, these two conditions can be established under significantly weaker assumptions. 
Furthermore, we emphasize that in certain cases, even \eqref{astbouona} may be unnecessary, provided the proof is appropriately adapted.

For simplicity, we assumed the compactness of $\mathcal{C}$  purely for convenience, avoiding unnecessary complications. It is worth noting that in many of the examples presented in Section \ref{sec:examples}, this assumption does not hold. Nevertheless, Theorem \ref{th:solreg} is introduced here mainly to motivate the HJB equation. The solution approach follows a verification-type method, which is carefully detailed in the next subsection. This approach \textbf{does not require compactness of $\mathcal{C}$.}  
\hfill $\square$
\end{Remark}

\begin{Remark}\label{rem:HJBunbounded}
We now look at what changes in the case presented in Subsection
\ref{SSSE:Bunbounded}, i.e. when 
in the state equation appears an additional term $\mathcal{B}u$ 
where is is unbounded (i.e. $\mathcal{B}:U\to \overline{X}$ with $X \subsetneq \overline{X}$).
In such a case the HJB equation \eqref{eq:HJB}
contains an additional term inside the Hamiltonian $\mathcal{H}$. Indeed 
we have, formally,
\begin{equation}\label{eq:HCVunbounded}
\mathcal{H}_{cv} (t,x,p;u):= 
\big\langle\mathcal{B}u, p\big\rangle+\big\langle  f(t,x,u),\, p\big\rangle+g(t,x,u), \ \ \  \ \ \  (t,x)\in [0,T)\times \mathcal{D}, \ p\in X, \ u\in \mathcal{C}.
\end{equation}
Note that the term 
$\big\langle\mathcal{B}u, p\big\rangle$ may not be well defined for all $p\in X$. However, under Assumption \ref{ass:Bunbounded}-(i) we have that $\mathcal{B}\in \mathcal{L}(U,D(\A^*)'$, hence, for $p\in D(\A^*)$, such term can be rewritten as
$\big\langle u, \mathcal{B}^*p\big\rangle$.
Moreover, under Assumption \ref{ass:Bunbounded}-(ii) the state trajectory still lives in $X$.
\\
Now, in the HJB equation \eqref{eq:HJB} $p=DV(t,x)$ and, under the assumptions of Theorem \ref{th:solreg} we have
$V\in \mathcal{S}([0,T)\times \text{Int }\mathcal{D})$, so $\A^*DV$ is well defined and continuous. This means that the term
$\big\langle u,
\mathcal{B}^*DV(t,x)\big\rangle$
is well defined for 
$(t,x) \in [0,T)\times X$. This, together with the fact that the state trajectory still lives in $X$ allows to say that the statement of Theorem \ref{th:solreg} holds true also in this case with the same proof.
\hfill $\square$
\end{Remark}

\subsection{Verification theorem --  finite horizon case}
In this subsection, we prove the verification theorem for classical solutions to HJB \eqref{eq:HJB} in the finite horizon case, i.e. when $T<\infty$. 
\begin{Theorem}[Verification Theorem --- Finite horizon]\label{teo:ver}
Let $\mathcal{O}:=\mbox{Int }\mathcal{D}$ and $v\in \mathcal{S}([0,T)\times \mathcal{O})\cap C^0([t,T]\times \mathcal{O})$ be a classical solution to \emph{HJB} \eqref{eq:HJB} at all $(t,x)\in [0,T)\times \mathcal{O}$ such that $v(T,\cdot)=\phi$. Then, given $(t,x)\in [0,T)\times \mathcal{O}$, we have the following statements. 
 \begin{enumerate}[(i)]
 \item  $v(t,x)\geq V(t,x)$.
 \item If there exists $u^*(\cdot)\in \mathcal{U}(t,x)$ such that, setting $y^*(\cdot):=y^{t,x,u^*(\cdot)}$, one has 
 $$
\mathcal{H}_{cv} (s,y^*(s), Dv(s,y^*(s));u^*(s))= \mathcal{H}(s,y^*(s), Dv(s,y^*(s))), \ \ \mbox{for a.e.} \ s\in[t,T],
 $$
 then $v(t,x)=V(t,x)$ and $u^*(\cdot)$ is optimal for the control problem starting at $(t,x)$.
 \end{enumerate}
\end{Theorem}
\begin{proof}
    \begin{enumerate}[(i)]
        \item Let $(t,x)\in [0,T)\times \mathcal{O}$ and  $u(\cdot)\in \mathcal{U}(t,x)$ and set $y(\cdot):=y^{t,x,u(\cdot)}$. By Proposition \ref{prop:chain}, using the definition of $\mathcal{H}_{cv}$, we have for each $r\in[t,T)$
        \begin{align*}
            &v(r,y(r))= v(t,x)\\
            &+\int_t^r \Big(v_t(s,y(s))+\langle y(s), \mathcal{A}^* Dv(s,y(s))\rangle + \mathcal{H}_{cv} \big(s,y(s), Dv(s,y(s));u(s)\big)-g(s,y(s),u(s))\Big)\ud s.   
        \end{align*}
        Rearranging and using the fact that $v$ solves HJB \eqref{eq:HJB}, we get
         \begin{align}\label{eq:fund}
            v(t,x)&= v(r,y(r))+\int_t^r g(s,y(s),u(s))\ud s\\
            &\ +\int_t^r \Big(\mathcal{H} \big(s,y(s), Dv(s,y(s))\big)- \mathcal{H}_{cv} \big(s,y(s), Dv(s,y(s));u(s)\big)\Big)\ud s.\nonumber   
        \end{align}
        Letting $r\to T^-$ and using the fact that $v(T,\cdot)=\phi$ and that the integrand above is nonnegative by definition, we get $v(t,x)\geq  \mathcal{J}(t,x;u(\cdot))$. By arbitrariness of $u(\cdot)$, we get the claim.
        \medskip
        \item Assume that $u^*(\cdot)$ satisfies the requirement of the claim. Then the integral in the second line of \eqref{eq:fund} vanishes. Arguing as before, we get  $v(t,x)=  \mathcal{J}(t,x;u^*(\cdot))$. Combining with item (i), we get the claim. 
    \end{enumerate}
    \vspace{-.5cm}
\end{proof}

As for the construction of optimal controls $u^*(\cdot)$ according to part (ii) of Theorem \ref{teo:ver}, this is typically done by relying on the concept of feedback  map and feedback control passing through the so called \emph{closed loop equation}. We illustrate these concepts.
A \emph{feedback map} is simply a measurable map $\phi:[0,T]\times \mathcal{D}\to \mathcal{C}$.
Given a feedback map $\phi$ and $(t,x)\in [0,T)\times \mathcal{D}$, we can associate to $\phi$ the \emph{closed loop equation}
\begin{equation}\label{eq:CLE}
y'(s)= \mathcal{A}y(s)+f(s,y(s),\phi(s,y(s)), \ \ \ \ \ y(t)=x.
\end{equation}
\medskip
By mild solution to \eqref{eq:CLE} we intend a function $y:[t,T]\to \mathcal{D}$ such that, for every $r\in[t,T]$, 
$$y(r)=e^{(r-t)\A} x + \int_t^re^{(r-s)\mathcal{A}} f(s,y(s),\phi(s,y(s)))\ud s.$$

\begin{Definition}
\label{def:optimalcontrol-optimalfeedback}
A feedback map $\phi:[0,T]\times\mathcal{D}\to \mathcal{C}$ is said 
\begin{enumerate}[(i)]
    \item 
\emph{admissible} starting at $(t,x)\in [0,T)\times \mathcal{D}$ if  \eqref{eq:CLE} has a unique mild solution $y^\phi(\cdot)$ and the associated \emph{feedback control} defined by $u^\phi(s):=\phi(s,y^\phi(s))$, $s\in[t,T]$, belongs to $\mathcal{U}(t,x)$;
\item \emph{optimal} starting at $(t,x)\in [0,T)\times \mathcal{D}$ if it is admissible and the associated \emph{feedback control} defined by $u^\phi(s)$ is optimal for the control problem starting at $(t,x)$.
\end{enumerate}
\end{Definition}
\begin{Corollary}
\label{cr:feedback}
Let $\mathcal{D}=\mathcal{O}:=\mbox{Int }\mathcal{D}.$
 Under the assumptions of Theorem \ref{teo:ver},  assume further that the supremum is achieved in \eqref{Hsup}, with $p=Dv(t,x)$; that is (changing the labels to the variables to avoid confusion in what follows), for all $(s,y)\in [0,T]\times \mathcal{O}$,  there exists $\phi(s,y)$ such that
\begin{equation}\label{Hsupbis}
\mathcal{H}_{cv} (s,y,Dv(s,y);\phi(s,y))
=\max_{u\in \mathcal{C}} \mathcal{H}_{cv} (s,y,Dv(s,y);u).
\end{equation}
Let $(t,x)\in [0,T)\times \mathcal{O}$. If $\phi$ is an admissible feedback map starting at $(t,x)$, then it is optimal starting at $(t,x)$. 
\end{Corollary}
\begin{proof}
By construction the feedback control $u^\phi(\cdot)$ fulfills the requirement of part (ii) of Theorem \ref{teo:ver}. So,  $u^\phi$ is an optimal (feedback) control and  $\phi$ is an optimal feedback map starting at $(t,x)$.
\end{proof}

\begin{Remark}\label{rem:VerificUnbounded}
We observe that verification Theorem \ref{teo:ver} and its Corollary \ref{cr:feedback} remain true also in  the case of Subsection
\ref{SSSE:Bunbounded}, i.e. when 
in the state equation appears an additional term $\mathcal{B}u$ 
where is is unbounded.
\\
Indeed in such a case the proof simply uses that:
\begin{itemize}
    \item the term 
$\langle u,\mathcal{B}^DV \rangle$
is well defined (which is guaranteed by Assumption \ref{ass:Bunbounded}-(i));
    \item the trajectory $y(\cdot)$ associated to any admissible control strategy remains in $X$ (which is guaranteed by Assumption \ref{ass:Bunbounded}-(ii));
\item 
the chain rule of Proposition \ref{prop:chainunbounded} holds.
\end{itemize}
\hfill $\square$
\end{Remark}

\subsection{The infinite horizon stationary case}\label{sub:infstat}
In this subsection, we discuss the results provided above in the infinite horizon stationary case described in Subsection \ref{sec:infer}. Considering that in this case $V_{t}(t,x)=-\rho e^{-\rho t} V(0,x)$, the HJB equation associated to $V(x):=V(0,x)$ reduces to 
\begin{equation}\label{eq:HJBst}
    \rho v(x)=\langle \mathcal{A}x,Dv(x)\rangle+\mathcal{H}(x,Dv(x)), \ \ \ \ x\in  \mathcal{D},
\end{equation}
where 
$$
\mathcal{H}(x,p):=\sup_{u\in \mathcal{C}} \mathcal{H}_{cv} (x,p;u), \ \ \  x\in\mathcal{D}, \ p\in X,$$
and 
$$
\mathcal{H}_{cv} (x,p;u):= \big\langle f_0(x,u),\, Dv(x)\big\rangle+g_0(x,u), \ \ \  \ \ \  x\in  \mathcal{D}, \ p\in X, \ u\in \mathcal{C}.
$$
\begin{Definition}
\label{def:classicalsolution}
 Let $\mathcal{O}:=\mbox{Int }\mathcal{D}$. We say that a function $v\in \mathcal{S}(\mathcal{O})$ solves \emph{HJB} \eqref{eq:HJBst} in classical sense at $x\in\mathcal{O}$ if 
\begin{equation}\label{eq:HJBbisst}
    \rho v(x)=\langle x,\mathcal{A}^* Dv(x)\rangle+\mathcal{H}(x,Dv(x)).
\end{equation}
\end{Definition}
\medskip
As for the finite horizon case, the following result holds true; the proof is exactly the same.
\begin{Theorem}
\label{th:HJBstaz}
   Let $\mathcal{O}:=\mbox{Int }\mathcal{D}$ and assume that $\mathcal{C}$ is compact and that  $V\in \mathcal{S} (\mathcal{O})$. Then, $V$ solves \emph{HJB} \eqref{eq:HJBst}  in classical sense at each $x\in \mathcal{O}$.
\end{Theorem}
Then, similarly to the finite horizon case, we have the following verification theorem.
\begin{Theorem}[Verification Theorem --- Infinite horizon stationary  case]\label{teo:verst}
$\mathcal{O}:=\mbox{Int }\mathcal{D}$ and 
let $v\in \mathcal{S}(\mathcal{O})$ be a classical solution to \emph{HJB} \eqref{eq:HJBst} at all $x\in  \mathcal{O}$.
 Then, given $x\in  \mathcal{O}$, we have the following statements. 
 \begin{enumerate}[(i)]
 \item  Assume that, for all $u(\cdot)\in\mathcal{U}(x)$,
there exists  $\hat u \left( \cdot \right)\in  \mathcal{U}(x)$ such that
$\mathcal{J}\left(x;\hat u\left( \cdot \right) \right) \geq \mathcal{J}\left(x;u \left( \cdot \right) \right)$
and such that, denoting $\hat y(\cdot):=y^{x,\hat u(\cdot)}$, 
\begin{equation}\label{tracont}
\limsup_{r\to\infty} e^{-\rho r}v(\hat y (r))\ge 0.
\end{equation}
 Then,
 $v(x)\geq V(x)$.
 \item In addition, assume that there exists $u^*(\cdot)\in \mathcal{U}(x)$ such that, setting $y^*:=y^{x,u^*(\cdot)}$, one has 
 \begin{equation}\label{feedverst}
\mathcal{H}_{cv}^0 (y^*(s), Dv(y^*(s));u^*(s))= \mathcal{H}^0(y^*(s), Dv(y^*(s))), \ \ \mbox{for a.e.} \ s\in \mathbb{R}^+,
 \end{equation}
 and 
 \begin{equation}\label{tracontbis}
\liminf_{r\to\infty} e^{-\rho r}v(y^* (r))\le 0.
\end{equation}
    Then, $v(x)=V(x)$ and $u^*(\cdot)$ is optimal for the control problem starting at $x$.
 \end{enumerate}
\end{Theorem}
\begin{proof}
    \begin{enumerate}[(i)]
        \item 
        Let  $x\in\mathcal{O}$ and 
        set
        $$\hat{\mathcal{U}}(x):=
        \{\hat u(\cdot)\in \mathcal{U}(x): \ \ \eqref{tracont} \  \mbox{holds}\}.
        $$
        Clearly,
        \begin{equation}\label{V00}
        V(x)=\sup_{\hat u(\cdot)\in\hat {\mathcal{U}}(x)}\mathcal{J}(x;\hat u(\cdot)).
        \end{equation}
        Now,   given  $\hat u(\cdot)\in \hat {\mathcal{U}}(x)$, set $\hat y(\cdot):=y^{x,\hat u(\cdot)}$.
        Arguing as in the proof of Theorem \ref{teo:ver}(i), we use  the stationary chain rule \eqref{chainst} getting
\begin{align*}
&e^{-\rho r}v(y(r))= v(x)\\
&+\int_t^r e^{-\rho s}\Big(-\rho v(y(s))+\langle y(s), \mathcal{A}^* Dv(y(s))\rangle + \mathcal{H}_{cv} \big(y(s), Dv(s,y(s));u(s)\big)-g_0(y(s),u(s))\Big)\ud s.   
        \end{align*}
Now we use the fact that $v$ solves the HJB equation \eqref{eq:HJBbisst}, getting
the analogous of \eqref{eq:fund} in this case:
\begin{align}\label{eq:fundstaznew}
v(x)&= e^{-\rho r}v(y(r))+\int_t^r g_0(y(s),u(s))\ud s\\
&\ +\int_t^r \Big(\mathcal{H} \big(y(s), Dv(s,y(s))\big)- \mathcal{H}_{cv} \big(y(s), Dv(s,y(s));u(s)\big)\Big)\ud s\nonumber   
        \end{align}
which implies
         \begin{align}\label{eq:fundst}
            v(x)&\geq  e^{-\rho r}{v(\hat y(r))}+\int_0^r e^{-\rho s}g_0(\hat y(s),\hat u(s))\ud s.
        \end{align}
        Letting $r\to \infty$ and using \eqref{tracont}, we get $v(x)\geq  \mathcal{J}(x;\hat u(\cdot))$. By arbitrariness of $\hat u(\cdot)$ and by \eqref{V00}, we get the claim.

\item Assume that $u^*(\cdot)$ satisfies the requirement of the claim. Then the integral in the second line of \eqref{eq:fundstaznew} vanishes. Arguing as in the  proof of Theorem \ref{teo:ver}(ii), we get  $v(t,x)\leq   \mathcal{J}(t,x;u^*(\cdot))$. Combining with item (i), we get the claim. \vspace{-.5cm}
    \end{enumerate}
\end{proof}
\medskip
\begin{Remark}

The limiting conditions \eqref{tracont}--\eqref{tracontbis} are known as \emph{transversality conditions}, and they commonly arise in infinite-horizon optimal control problems. 
The conditions we have provided are fairly sharp. 
One sufficient condition ensuring both is 
$$
\lim_{r\to \infty} e^{-\rho r} v(y(r))=0 \ \ \forall u(\cdot)\in \mathcal{U}(x).
$$
Unfortunately, the above condition is typically too restrictive in applications, requiring one to work instead with \eqref{tracont}--\eqref{tracontbis}.
\end{Remark}
As for the finite horizon case, 
 the construction of optimal controls $u^*(\cdot)$ according to part (ii) of Theorem \ref{teo:verst} relies on the concept of feedback map and requires the study of the  closed loop equation. In this case, 
a \emph{feedback map} is simply a measurable map $\phi: \mathcal{D}\to \mathcal{C}$; 
given a feedback map $\phi$ and $x\in  \mathcal{D}$, we can associate to $\phi$ the closed loop equation
\begin{equation}\label{eq:CLEst}
y'(s)= \mathcal{A}y(s)+f_0(y(s),\phi(y(s)), \ \ \ \ \ y(0)=x.
\end{equation}
\medskip
By mild solution to \eqref{eq:CLEst} we intend a function $y:\R^+\to \mathcal{D}$ such that, for every $r\in\R^+$, 
$$y(r)=e^{r\A} x+ \int_0^re^{(r-s)\mathcal{A}} f_0(y(s),\phi(y(s)))\ud s.$$

\begin{Definition}
\label{def:optimalcontrol-optimalfeedback2}
A feedback map $\phi:\mathcal{D}\to \mathcal{C}$ is said 
\begin{enumerate}[(i)]
    \item 
\emph{admissible} starting at $x\in \mathcal{D}$ if  \eqref{eq:CLEst} has a unique mild solution $y^\phi(\cdot)$ and the associated \emph{feedback control} defined by $u^\phi(s):=\phi(y^\phi(s))$, $s\in \R^+$, belongs to $\mathcal{U}(x)$;
\item \emph{optimal} starting at $x\in  \mathcal{D}$ if it is admissible and the associated \emph{feedback control} defined by $u^\phi(s)$ is optimal for the control problem starting at $x$.
\end{enumerate}
\end{Definition}
\begin{Corollary}
\label{cr:feedbackstaz}
Let $\mathcal{D}=\mathcal{O}:=\mbox{Int }\mathcal{D}$.
Under the assumptions of Theorem \ref{teo:ver},  assume further that the supremum is achieved in  $u\mapsto \mathcal{H}_{cv} (x,p;u)$, with $p=Dv(x)$; that is (changing the labels to the variable to avoid confusion in what follows), for all $y\in  \mathcal{O}$,  there exists $\phi(y)$ such that
\begin{equation}\label{Hsupbisbis}
\mathcal{H}_{cv} (y,Dv(y);\phi(y))
=\max_{u\in \mathcal{C}} \mathcal{H}_{cv} (y,Dv(y);u).
\end{equation}
Let $x\in  \mathcal{O}$. If $\phi$ is an admissible feedback map starting at $x$ such that \eqref{tracontbis} is satisfied by the feedback strategy $u^\phi$, then the latter is optimal starting at $x$. 
\end{Corollary}
\begin{proof}
By construction the feedback control $u^\phi(\cdot)$ fulfills the requirements of part (ii) of Theorem \ref{teo:verst}. So,  $u^\phi$ is an optimal (feedback) control and  $\phi$ is an optimal feedback map starting at $x$.
\end{proof}

\begin{Remark}\label{rem:VerificUnboundedSStaz}
As in Remarks \ref{rem:HJBunbounded} and
\ref{rem:VerificUnbounded} we observe that Theorems \ref{th:HJBstaz} and \ref{teo:verst} and its Corollary \ref{cr:feedbackstaz}
remain true also in  the case of Subsection
\ref{SSSE:Bunbounded}, for the same reasons.
\hfill $\square$
\end{Remark}

\section{Dynamic Programming in the nonsmooth case: weaker concepts of solutions}\label{sec:non-linear}

The theoretical results presented in Section \ref{sec:DP}, while quite elegant from a theoretical standpoint, have a limited range of applicability. Indeed, Theorem \ref{th:solreg} is mainly heuristic, because the assumption that $V \in C^{1}$ is essentially uncheckable beforehand. The verification approach  (see Theorem \ref{teo:ver}) relies on the existence of a classical solution $v \in C^{1}$  to the HJB equation: although it surely covers some important applications (see Section \ref{sec:examples}), it  is only practical when the model is specifically designed to have explicit solutions; indeed,  as a matter of fact, a systematic theory for classical solutions of first-order HJB equations is not available. We should note that this difficulty is intrinsic and  not only due to the infinite-dimensional setting, although that certainly adds complexity. In this perspective,   the scope of this section is to provide a quick informal introduction to weaker notions of solutions to HJB equations and more refined techniques that are applicable in more general situations.

\subsection{Viscosity solutions}
Back in the early 1980s, the concept of viscosity solutions to PDEs in finite dimension was introduced (see the seminal paper \citealp{crandall1983viscosity}). This notion was particularly well-suited  to address the well-posedness of HJB equations and proved to be very flexible since it allows to establish good comparison results for a wide class of equations, including HJB ones but not limited to them. From a control theory perspective, viscosity solutions have been very successful because, relying on the Dynamic Programming Principle but replacing the value function with smooth test functions, the value function $V$ naturally appears as a candidate viscosity solution to the associated HJB equation. This is possible without any smoothness assumptions (even continuity is not needed).
The analytical comparison results, often based on the so-called doubling of variables method (see, e.g., among many others, \citealp{crandall1992user}), then help to characterize the value function as the unique solution of the HJB. Also numerical schemes,  starting from the seminal paper \cite{souganidis1985approximation}, were successfully developed afterwards.
An extensive description of all these aspects in the deterministic, finite dimensional, framework can be found, e.g. in the monograph \cite{bardi1997optimal}.

Afterwards, 
the theory was successfully extended to infinite dimensional spaces in a series of papers --- see \cite{crandall1985viscosity, crandall1990viscosity,    crandall1991viscosity, crandall1994viscosity} --- and applied to optimal control theory in Hilbert spaces --- see the book \cite{li2012optimal} for an overview. The complete technical exposition of this matter goes out of the scopes of this survey; nonetheless, we provide a heuristic exposition of its main ideas in order to be able to discuss some examples without explicit solution in Section \ref{sec:examples}.


We may say that the heuristic idea behind the notion of viscosity solution is to treat separately  the ``convex singularities"
(i.e. the ones where the superdifferential is empty)
and the ``concave singularities"
(i.e. the ones where the superdifferential is empty).
To do this one splits the HJB equation into two inequalities, one for the
``convex singularities" and one for the ``concave singularities".
More precisely, one inequality is required to hold on the elements of the subdifferential and the other on the elements of the superdifferential.
This is equivalent to ask that such inequality hold when tested against functions that touch the candidate solution from below or from above.
For the sake of brevity, we illustrate this concept only for the infinite horizon stationary case and requiring continuity for the solution.

\begin{Definition}[Viscosity solution]
Let $\mathcal{O}:=\mathrm{Int}\, \mathcal{D}$.
\vspace{-.1cm}
\begin{enumerate}[(i)]
\item We say that $v\in C(\mathcal{D})$ is a viscosity supersolution (resp., subsolution) to \eqref{eq:HJBst} at $x\in\mathcal{O}$ if, whenever $\varphi\in \mathcal{S}(\mathcal{O})$ is such that $0=v(x)-\varphi(x)=\min (v-\varphi)$ (resp., $0=v(x)-\varphi(x)=\max (v-\varphi)$), we have
\begin{equation}\label{eq:HJBstviscsup}
\rho \varphi (x)\geq \langle x,\mathcal{A}^{*}D\varphi(x)\rangle+\mathcal{H}(x,D\varphi(x)),
\end{equation}
(resp.,
\begin{equation}\label{eq:HJBstviscsub}
\rho \varphi (x)\leq \langle x,\mathcal{A}^{*}D\varphi(x)\rangle+\mathcal{H}(x,D\varphi(x))).
\end{equation}

\item We say that $v\in C(\mathcal{D})$ is a viscosity solution to \eqref{eq:HJBst} at $x\in\mathcal{O}$ if it is both a viscosity supersolution and and a viscosity subsolution at $x$.
\end{enumerate}
\end{Definition}

As mentioned, the strength of this approach in optimal control theory is that, under fairly general assumptions, the value function $V$ can be shown to be a viscosity solution of the HJB equation --- see, e.g.,  \citealp[Ch.\,6]{li2012optimal}, and the references therein. Combining this with comparison principles for viscosity solutions provides a powerful way to characterize $V$, and numerical schemes can be employed to approximate it.
Nonetheless, from the perspective of solving optimal control problems, this approach is not entirely satisfactory. Ideally, one would like to eventually construct optimal controls using sufficient optimality conditions, in the spirit of Theorem \ref{teo:verst}(ii). Verification theorems in the non-smooth setting — particularly within the framework of viscosity solutions  — have been developed both in finite and infinite dimensions; see, for example, \cite{zhou1993verification}, and, for the infinite dimensional case \citealp{li2012optimal}, Ch.\,4, Th.\,5.4, p.262,   and \cite{fabbri2007verification}.
However, in practice, applying these results can be quite challenging. The main difficulty lies in producing a suitable candidate for the optimal feedback control and verifying the optimality conditions. A more promising middle ground involves leveraging some a priori regularity properties of the value function $V$, such as (semi)convexity or (semi)concavity. Textbooks like \citealp[Ch.\,II]{bardi1997optimal}, and \cite{cannarsa2004semiconcave} focus on how these features can be exploited in finite-dimensional settings.

More precisely, if we can establish this kind of regularity beforehand and use the viscosity property of $V$, we can attempt to:
\begin{enumerate}[(i)] \item Either construct implementable optimal feedback controls via nonsmooth verification theorems (see, e.g., in finite dimension \citealp[Th.\,5]{federico2024optimal}); \item Or prove that $V$ is differentiable at least along certain directions, and then build optimal feedback controls using \emph{ad hoc} verification theorems that leverage both the viscosity solution framework and the partial regularity of $V$. \end{enumerate}

Both approaches can be extended to infinite-dimensional problems. In Section \ref{sec:exnonexpl}, we will present some examples where the second method has been successfully applied.


\subsection{Strong solutions}

Similarly to viscosity solutions, \emph{strong solutions} generalize the concept of a classical solution, although satisfying stronger regularity requirements. More precisely a strong solution $v$ is such that its gradient $Dv$ exists in classical sense so that
the feedback map of Corollaries \ref{cr:feedback} and \ref{cr:feedbackstaz} is well defined (which does not happen in general in the case of viscosity solutions).

Roughly speaking, a strong solution is the limit, in a suitable sense,  of classical solutions (not to the original equation, but) to approximating HJB equations. 
For finite horizon, for example, one considers \eqref{eq:HJB}
$$-v_t(t,x)=\langle   x,\A^*Dv(t,x)\rangle +\mathcal H(t, x, Dv(t,x)), $$
assumes suitable hypotheses on 
$\mathcal H:[t,T]\times \X\times \X'\to \mathbb R$, 
(such as continuity and growth  bounds), and defines  a sequence of 
$\mathcal H_\e:[t,T]\times \X\times \X'\to \mathbb R$ 
so that: (1)  
$\mathcal{H} ^\e$ converges to $\mathcal H$, as $\e\to0$, in an appropriate sense; 
(2) the approximating equations
$$-v^\e_t(t,x) =\langle  x,\A^*Dv^\e(t,x)\rangle +\mathcal{H} ^\e(t, x, Dv^\e(t,x)) $$
have classical solutions $v^\e$;   
(3) the sequence $\{v^\e\}$ converges uniformly to a function $v$ and the gradients $\{\nabla v^\e\}$ converge to
$\nabla v$. Then $v$ is called a strong solution to \eqref{eq:HJB}. 

There are indeed few variants of such concept of strong solution, on this one can see the book \cite{BarbuDaPrato1983}
and, e.g., the papers \cite{Gozzi1991SICON, faggian2005regular,FaggianGozzi2010JME}.

We will give in more detail in Section \ref{vcteimpl} the definition of strong solution that adheres to that case.

\section{Examples with explicit solutions arising in economic theory}
\label{sec:examples}

\subsection{Spatial growth models}
A first class of examples of economic models to which optimal control techniques in infinite-dimensional spaces apply is that of a series of growth models where, instead of considering only a one-dimensional aggregate variable to represent capital, its spatial distribution is taken into account. The aim is to be able to formulate counterparts of benchmark growth models such as those by Solow or Ramsey. One possible idea to achieve this goal is the one initially introduced by \cite{Brito04}, where the author, following an idea borrowed from classical economic geography (see for example \citealp{BeckmannPuu85}, \citealp{Isard79}), assumes that capital tends to move from regions with greater capital intensity to regions with lower capital intensity.
\label{sub:example-spatial}
\subsubsection{The economic problem}

Let us assume, to begin with, that we represent the spatial structure as the unit circle \( S^1 \) in the plane:
\begin{equation}
\label{eq:defS1}
S^1 := \big\{(\sin \theta, \cos \theta) \in \mathbb{R}^2 \, : \, \theta \in [0,2\pi) \big\}.
\end{equation}
The fact that the chosen spatial structure is not particularly realistic may not be very relevant in an initial analysis, since the purpose of "benchmark" models is not to faithfully reproduce specific measurements, but rather to investigate qualitative behaviors and possible economic mechanisms. In any case, as we will show  generalizations of this structure are possible.

We denote any point in the space by \(\theta \in [0,2\pi)\) and, as usual, any time by \(t \geq 0\). In a linear specification of the model, which generalizes the classic \(AK\) model of endogenous growth (see \citealp{Acemoglu09}) to a world with spatial heterogeneity, we can, for example, postulate that for all time \(t \geq 0\) and any point in the space \(\theta \in [0,2\pi)\), the production is a linear function of the employed capital:
\[
Y(t,\theta) = A(\theta) K(t, \theta),
\]
where \(K(t,\theta)\) and \(Y(t,\theta)\) represent, respectively, the capital and output at the location \(\theta\) at time \(t\), while \(A(\theta)\) is the exogenous location-dependent technological level.
For each location, we suppose there is a budget constraint in the form
\[
Y(t,\theta) = I(t,\theta) + C(t,\theta) + \tau(t,\theta)
\]
where $C$, $I$, and $\tau$ are respectively the consumption, investment, and trade balance, which depend on the location and time. If we include a location-dependent depreciation rate $\delta(\theta)$, we get the following law of capital accumulation:
\[
{\frac{\partial K}{\partial t}}(t,\theta) = I(t, \theta) - \delta (\theta) K(t,\theta) = (A(\theta) - \delta(\theta)) K(t,\theta) - C(t,\theta) - \tau(t,\theta).
\]
We can incorporate the depreciation rate $\delta(\theta)$ into the coefficient $A(\theta)$ (that we suppose to be a bounded function of $\theta$), thus omitting the term $\delta$ in the subsequent discussion. According to \cite{Brito04}, the aim is to develop further the term $\tau(t,\theta)$ by assuming that the flow rate of capital from left to right at a point is the negative derivative of the capital level at that point. By ensuring that the net trade balance of the region $(\theta_1, \theta_2)\subset [0,2\pi)$ matches the capital outflow at the boundaries $\theta_1$ and $\theta_2$, we get
\begin{equation}
\label{eq:perottenretaudiBrito}
\int_{\theta_1}^{\theta_2} \tau(t,\theta) d \theta = \frac{\partial K}{\partial \theta}(t,\theta_1) - \frac{\partial K}{\partial \theta}(t,\theta_2). 
\end{equation}
Thus,
\[
\tau(t,\theta)=-\frac{\partial^2 K}{\partial \theta^2} (t,\theta).
\]
Combining everything, if for each $(t,\theta)$ we express the total consumption $C(t,\theta)$ as the product of the (equally distributed) per-capita consumption $c(t,\theta)$ and the time-independent exogenous (density of) population $N(\theta)$ (that we will suppose to be bounded), we obtain the state equation.
\begin{equation}\label{eq:state-findim-spatial}
\begin{cases}
\displaystyle{\frac{\partial K}{\partial t}(t,\theta)\!=\!\frac{\partial^2K}{\partial \theta^2}(t,\theta)\!+\!A(\theta)K(t,\theta)\!-\!c(t,\theta)N(\theta),\  t>0,\, \theta\in S^1,}\\\\
K(0,\theta)=K_0(\theta), \ \ \ \theta\in S^1,
\end{cases}
\end{equation}
where $K_0\in L^2(S^1; \R^+)$ is the initial capital distribution.  

We assume that the policy maker aims to maximize the following ``Benthamite'' functional (considering trajectories that ensure capital and consumption remain positive):
\begin{equation}\label{funct}
\int_0^\infty e^{-\rho t}\left( \int_0^{2\pi}  \frac{c(t,\theta)^{1-\sigma}}{1-\sigma}N(\theta) d\theta \right)\ud t,
\end{equation}
where $\rho>0$ represents the time-discount factor and $\sigma\in(0,1)\cup(1,\infty)$ serves as both a measure of (the inverse of the elasticity of) intertemporal substitution and the planner's aversion to inequality.

\subsubsection{Formal rewriting of the problem in Hilbert space}

We can represent \eqref{eq:state-findim-spatial} as an abstract dynamical system in infinite dimensions. Here, we will quickly show how to reduce the described problem to the general situation outlined in Section \ref{sec:formulation}. We will work in the Hilbert space
\begin{equation}
\label{eq:defL2}
X := L^2(S^1;\R):= \left\{ f: S^1 \to \mathbb{R} \text{ measurable}: \int_{S^1} |f(x)|^2 \, \mathrm{d}x < \infty \right\},
\end{equation}
endowed with the usual inner product $\langle f, g \rangle := \int_0^{2\pi} f(\theta) g(\theta) \, d\theta$, inducing the norm $|f| = \left(\int_0^{2\pi} | f(\theta) |^2 \, d\theta\right)^{1/2}$, and specify the generator of the semigroup with 
$$
\mathcal{A} u := \frac{\partial^2 u}{\partial \theta^2} + A(\cdot) u, \quad D(\mathcal{A}) = W^{2,2}(S^1)
$$
where the latter is the Sobiolev space 
\[
\displaystyle
W^{2,2}(S^1) := \big\{ f \in H: \ f \text{ is twice weakly differentiable, and } f', f'' \in X \big\}.
\]

The operator $\mathcal{A}$ is defined as the sum of the Laplacian operator on $S^1$ and the operator on $H$ defined by $u \mapsto A(\cdot)u$, which is bounded under the assumption that the function $A(\cdot)$ is bounded. The Laplacian operator is closed on the domain $D(\mathcal{A})$ and generates a $C_0$-semigroup on the space ${X}$ (see Theorem 4.9, page 115 of \cite{Grigoryan12}). Therefore (Theorem 1.1 page 190 of \cite{kato1980}), we conclude that $\mathcal{A}$ is also closed on the domain $D(\mathcal{A})$ and generates (Theorem 3.5 page 73 of \cite{engel-nagel}) a $C_0$-semigroup on the space ${X}$.

Finally, to reduce the considered example to the setting of Section \ref{sec:formulation}, we define $
U = X=L^2(S^1; \R)
$
and call $t \mapsto y(t) \in X$ the trajectory of the system and $t \mapsto u(t) \in X$ a generic control by formally identifying
\[
y(t)(\theta) = K(t,\theta) \ \ \ 
\mbox{and} \ \ \ 
u(t)(\theta) = c(t,\theta).
\]
In this way, we  can formally rewrite \eqref{eq:state-findim-spatial} as an abstract dynamical system in the space $X$ (see Proposition 3.2, page 131 of \citealp{BDDM07}):
\begin{equation}\label{stateH}
\begin{cases}
y'(t)=\mathcal{A}y(t)-u(t)N,\ \ \ t\in \mathbb{R}^+,\\
y(0)=x \in X.
\end{cases}
\end{equation}
Now, 
consider the positive cone in ${X}$, i.e. the set   
\begin{equation}
\label{eq:defH+}
{X}^+:=\big\{y\in {X}\,: \ y(\cdot)\geq  0 \big\},
\end{equation}
and ${X}^+_0:={X}^+\setminus\{0\}$. We consider as set of admissible  admissible strategies\footnote{In this formulation we require the slightly sharper state constraint $y^{x, u}(t)\in{X}^+_0 $ in place of the wider (original)  one $y^{x, u}(t)(\cdot)\geq 0$ almost everywhere. This is without  loss of generality: indeed, if $y^{x, u}(t)\equiv 0$ at some $t\geq 0$, the unique admissible (hence the optimal) control from $t$ on is the trivial one $u(\cdot)\equiv 0$, so we know how to solve the problem once we fall into this state and there is no need to define the Hamilton-Jacobi-Bellman equation at this point. The reason to exclude the zero function from the set ${X}^+$ and considering the set ${X}^+_0$ is that in this set
our value function is well defined and solves the Hamilton-Jacobi-Bellman equation, while this does not happen in ${X}^+$.}
$$
\mathcal{U}(x):= \big\{u\in L^1_\loc (\R^+;{X}^+)\,:\  \ y^{x, u(\cdot)}(t)\in{X}^+_0\   \ \forall t \geq 0\big\}.
$$
Notice that ${X}^+_0$ is not open in $X$.   

Then, we can formally rewrite the original optimization problem in the infinite dimensional setting as the optimal control problem consisting in   maximizing the objective functional
\begin{equation}\label{objective}
 J(x;u):=
\int_0^\infty e^{-\rho t}  \mathcal{U}(u(t))\ud t,
\end{equation}
over all $u(\cdot)\in\mathcal{U}(x)$ where
$$
\Psi:{X}^+\rightarrow \ \R^+, \ \ \Psi(u):= \int_0^{2\pi}\frac{u(\theta)^{1-\sigma}}{1-\sigma}N(\theta) d\theta.
$$
In the following, we call $\mathbf{(P)}$ this problem and we define the associated value function as
\begin{equation}\label{valuefunction}
V(x):= \sup_{u(\cdot)\in\mathcal{U}(x)}J(x;u(\cdot)).
\end{equation}

\subsubsection{The infinite-dimensional \emph{HJB} equation and its explicit solution}

We begin by recalling some spectral properties of the operator
\[
\mathcal{A} f := f'' + A(\cdot) f, \qquad f\in D(\mathcal{A}).
\]
The operator $\mathcal{A}$ is self-adjoint, i.e. $\A=\A^*$,  and has a compact resolvent. Its spectrum consists of a decreasing sequence of real eigenvalues
\[
\lambda_0 > \lambda_1 \geq \lambda_2 \geq \dots \to -\infty
\]
with associated orthonormal eigenfunctions $\{e_n\}_{n \geq 0}$ (see Chapter 7 and Chapter 8, Section 3, in \cite{CoddingtonLevinson1955}).  
The principal eigenfunction $e_0$ is strictly positive and can be normalized so that $\|e_0\|_{L^2} = 1$.  
Moreover
\[
\frac{1}{2\pi} \int_{0}^{2\pi} A(\theta)\,d\theta
\;\leq\;
\lambda_0
\;\leq\;
\sup_{S^1} A(\theta),
\]
so that $\lambda_0 > 0$ whenever $A(\theta) \geq 0$ and $A \not\equiv 0$.

\medskip

In a certain sense, in the model, $\lambda_0$ plays the role of total factor productivity $A$ in a one-sector $AK$ growth model. Therefore it is not surprising  that the standard assumption ensuring the finiteness of the objective functional takes the form
\begin{Assumption}
\label{hyp:rho}
  \(\displaystyle \rho > \lambda_0 \bigl(1-\sigma\bigr).\)
\end{Assumption}
It is not difficult to prove that, if this condition fails, there exist admissible trajectories along which the value functional is infinite, and the theory developed in the previous sections does not apply.

\bigskip

Recalling the state dynamics \eqref{stateH} and the objective functional \eqref{objective},  
the stationary infinite-horizon HJB equation can be written, in line with \eqref{eq:HJBbisst}, as
\begin{equation}\label{eq:HJB-22}
  \rho v(x)
  = \langle x, \mathcal{A} D v(x) \rangle
   + \sup_{u \in {X}^+} \Bigl\{\, \Psi(u) - \langle u N, D v(x) \rangle \Bigr\},
  \qquad x \in {X}^+,
\end{equation}
where ${X}^+$ is the positive cone introduced in \eqref{eq:defH+}.

We now proceed to search for an explicit solution to \eqref{eq:HJB-22}. According to Definition \ref{def:classicalsolution}, a solution can be defined on an open set $\mathcal{O}$ of $H$. In this case, we choose as $\mathcal{O}$ the set
\[
  {X}^+_{e_0}
  := \bigl\{ y \in {X}^+ : \langle y, e_0 \rangle > 0 \bigr\}.
\]
We now have all the ingredients to describe the explicit solution to the HJB equation:

\begin{Proposition}\label{prop:HJB-solution}
Let Assumption \ref{hyp:rho} hold.
Define
\[
  \alpha_{0}:=
  \left(
        \frac{\sigma}{\rho-\lambda_{0}(1-\sigma)}
        \int_{S^{1}} e_{0}(\eta)^{1-\frac{1}{\sigma}}N(\eta)\,d\eta
  \right)^{\!\frac{\sigma}{1-\sigma}}, \quad \text{and} \quad \beta(\theta):=\alpha_{0}\,e_{0}(\theta)
\]
so that $\beta\in H$ and $\beta(\theta)>0$.  
The function
\begin{equation}\label{eq:ValueFunction}
  v(x)\;=\;\frac{\langle x,\beta\rangle^{\,1-\sigma}}{1-\sigma},
  \qquad x\in {X}^+_{e_{0}},
\end{equation}
belongs to ${C}^{1}({X}^+_{e_{0}})$ and solves \eqref{eq:HJB-22} on ${X}^+_{e_{0}}$ in the sense of Definition \ref{def:classicalsolution}.
\end{Proposition}

\begin{proof}
We directly check that the proposed function is a solution in the sense of Definition \ref{def:classicalsolution}. 
The function $v$ is Fréchet differentiable on $X_{e_0}^+$ and its gradient is
\[
D v(x) = \langle x, \beta \rangle^{-\sigma}\, \beta.
\]
Since $\beta$ is regular and belongs to $D(\mathcal{A})$, all the regularity requirements in Definition \ref{def:classicalsolution} are satisfied.

For fixed $x\in X_{e_0}^+$, the supremum in \eqref{eq:HJB-22}
is achieved by the unique optimizer
\[
u^*(x) = \left( D v(x)(\theta) \right)^{-1/\sigma} = \langle x, \beta \rangle\, \beta(\theta)^{-1/\sigma}.
\]
We plug $v$ and its gradient into the HJB equation, we compute each term:
\begin{align*}
\langle x, \mathcal{A} D v(x) \rangle &= \langle x, \mathcal{A} (\langle x, \beta \rangle^{-\sigma} \beta) \rangle \\
&= \langle x, \beta \rangle^{-\sigma} \langle x, \mathcal{A} \beta \rangle \\
&= \langle x, \beta \rangle^{-\sigma} \lambda_0 \langle x, \beta \rangle \\
&= \lambda_0 \langle x, \beta \rangle^{1-\sigma},
\end{align*}
since $\mathcal{A} \beta = \lambda_0 \beta$ (by definition of the principal eigenfunction).

For the utility term, using the explicit maximizer:
\begin{align*}
\Psi(u^*(x)) &= \int_{S^1} \frac{(u^*(x)(\theta))^{1-\sigma}}{1-\sigma} N(\theta)\, d\theta \\
&= \frac{1}{1-\sigma} \left( \langle x, \beta \rangle \right)^{1-\sigma} \int_{S^1} \beta(\theta)^{1-\frac{1}{\sigma}} N(\theta)\, d\theta.
\end{align*}
Similarly,
\begin{align*}
\langle u^*(x) N, D v(x) \rangle &= \int_{S^1} u^*(x)(\theta) N(\theta) D v(x)(\theta)\, d\theta \\
&= \langle x, \beta \rangle^{1-\sigma} \int_{S^1} \beta(\theta)^{-\frac{1}{\sigma}} N(\theta) \cdot \beta(\theta) \, d\theta \\
&= \langle x, \beta \rangle^{1-\sigma} \int_{S^1} \beta(\theta)^{1-\frac{1}{\sigma}} N(\theta)\, d\theta.
\end{align*}

Combining all terms and factoring out $\langle x, \beta \rangle^{1-\sigma}$ gives
\begin{align*}
\rho v(x) &= \lambda_0 \langle x, \beta \rangle^{1-\sigma}
+ \frac{1}{1-\sigma} \langle x, \beta \rangle^{1-\sigma} \int_{S^1} \beta(\theta)^{1-\frac{1}{\sigma}} N(\theta)\, d\theta \\
&\quad - \langle x, \beta \rangle^{1-\sigma} \int_{S^1} \beta(\theta)^{1-\frac{1}{\sigma}} N(\theta)\, d\theta \\
&= \langle x, \beta \rangle^{1-\sigma} \left[ \lambda_0 + \frac{\sigma}{1-\sigma} \int_{S^1} \beta(\theta)^{1-\frac{1}{\sigma}} N(\theta)\, d\theta \right].
\end{align*}

Simplifying the term $\langle x, \beta \rangle^{1-\sigma}$ this is equivalent to
\[
\frac{\rho}{1-\sigma} = \lambda_0 + \frac{\sigma}{1-\sigma} \int_{S^1} \beta(\theta)^{1-\frac{1}{\sigma}} N(\theta)\, d\theta
\]
but using the definition of $\alpha_0$, we easily obtain
\[
\int_{S^1} \beta(\theta)^{-\frac{1-\sigma}{\sigma}} N(\theta)\, d\theta = \frac{\rho - \lambda_0(1-\sigma)}{\sigma}.
\]
If we substitute this expression in the previous equation we can complete the proof that $v$ is a classical solution of \eqref{eq:HJB-22}.
\end{proof}

\subsubsection{Verification and optimal feedback control}
Using Theorem \ref{teo:verst}, we can finally exploit the explicit solution of the model obtained in the previous section to characterize the optimal control for the problem.
The  candidate optimal feedback map is provided  by
\begin{equation}\label{eq:Feedback-22}
  u^{\ast}(x)=\phi(x) := \langle x,\beta\rangle\,\beta^{-1/\sigma}\in {X}^+.
\end{equation}
While it is easy to show that starting from any initial datum $x\in {X}^+_0$ the trajectory governed by this feedback remains in ${X}^+_{e_{0}}$, it is typically to show (see \cite[Sec.\,5]{calvia2021state}) that the trajectory remains in ${X}^+_0$ (note that ${X}^+_0 \subsetneq {X}^+_{e_{0}}$), hence that the feedback control is admissible\footnote{This is actually, in general, false, see \cite{ricci2024non}}. This condition is necessary to apply Theorem \ref{teo:verst}. If we take it as an assumption (that can be verified in some situations, for instance  near  steady states) we get the following result.

\begin{Theorem}
\label{th:optiaml-feedback-JEG}
Let  Assumption \ref{hyp:rho} hold.
Let $x\in {X}^+_0$ and let $y^*(\cdot)$ be the solution to \eqref{stateH} when the system is controlled through the feedback \eqref{eq:Feedback-22}, i.e., the unique mild solution  to the closed loop equation
\begin{equation}
\begin{cases}
y'(t) = \mathcal{A}y(t) - \phi(y(t))N = \mathcal{A}y(t) - \langle y(t), \beta\rangle\,\beta^{-1/\sigma}N, & t\in \mathbb{R}^+, \\
y(0) = x.
\end{cases}
\end{equation}
Assume that $y^*(\cdot)\in {X}^+_0$ and therefore that the control
\[
u^*(t) = \phi(y^*(t)) = \langle y^*(t), \beta\rangle\,\beta^{-1/\sigma}
\]
is admissible. Then $u^*(\cdot)$ is an optimal control and $\phi$ is an optimal feedback at $x$. Moreover, $v(x) = V(x)$.
\end{Theorem}
\begin{proof}
It follows from Theorem \ref{teo:verst} (the transversality condition can be easily checked thanks to Assumption \ref{hyp:rho}).
\end{proof}

\subsubsection{Variants and Literature}
The formulation we have used in this section is essentially the one used by \cite{boucekkine2019growth} where a rather general ``AK" version of a spatial model with ``diffusive'' capital spillover is analytical solved. As already mentioned, spatial models with diffusive terms (in the case of neoclassical production function specification) was first introduced by \cite{Brito04}. After \cite{Brito04}, a series of papers used variations of the same idea. In this stream, we can cite the articles by \cite{camacho2008dynamics, BrockXepapadeas08JEDC, boucekkine2009bridging, ballestra2016spatial}, which, at least partially, offered methodological contributions and highlighted some technical difficulties emerging from the model proposed by Brito. In the mentioned articles, the authors found necessary conditions for optimality via a version of the Pontryagin maximum principle, but none of them provided a complete solution to the model. Some of them proposed simplifications to approach the problem: \cite{Brito04} focuses on \emph{traveling waves} type solutions, \cite{boucekkine2009bridging} look at the linear utility case, and \cite{Brito11} studies the local dynamics of the system. To overcome the technical difficulties described above, a number of authors, including \cite{Quah02} and \cite{brock2014optimal-JME}, used immobile factors of production (knowledge or capital) and modeled the interaction between different geographic regions through spatial externalities. The article by \cite{BoucekkineCamachoFabbri13} is the first to explicitly solve a model of spatial growth with mobility and capital accumulation with a continuous spatial dimension. More recently, we mention \cite{xepapadeas2016spatial, camacho2022diffusion, brito2022dynamics}, the stochastic extension by \cite{gozzi2022stochastic}, the network versions by \cite{fabbri2023growth, calvia2024optimal, albeverio2022large}, extensions with population mobility by \cite{camacho2019model, ballestra2022modeling, ballestra2024modeling}, studies about the role of the social welfare function by \citealp{boucekkine2021control}, and more technical contributions (\citealp{santambrogio2020rational, ricci2024non}).\footnote{See also \cite{juchem2023versao, urena2023numerical}.} A particular mention goes to the contribution by \cite{xepapadeas2023spatial}, where the authors, taking a more challenging perspective, propose a model where capital moves toward locations where the marginal productivity of capital is relatively higher.

We finally spend some more words on a generalization of the proposed structure that is presented by \cite{Fabbri16}, where the role of geographical structure in the growth process is taken into account. The analysis by \cite{Fabbri16} is carried out through a model which, in several aspects, is similar to the one that we described in the previous subsection, but in a simplified setting where $A$ and $N$ are constant (say, equal to $1$). In this case, the spatial structure (the geography) is modeled as an $n$-dimensional compact, connected, oriented Riemannian manifold $M$, with metric $g$ and without boundary ($S^1$, like any $S^n$, are obviously particular examples of this structure).

The structure of the model is very similar to the proposed one. To obtain the expression for $\tau$, instead of (\ref{eq:perottenretaudiBrito}) and the following equations, we now have
\[
\int_B \tau(t,x) \, \ud x  = 
- \int_{\partial B} \frac{\partial K(t,x)}{\partial \hat{n}} \, \ud x 
= - \int_{B} \Delta_x K(t,x) \, \ud x,
\]
where we used the divergence theorem (see \citet{Lang95}, Theorem 3.1, page 329), so that, for almost every $x \in M$,
\begin{equation}
\label{eq:tau=-delta}
- \tau(t,x) = \Delta_x K(t,x),
\end{equation}
where $\Delta_x$ is the Laplace-Beltrami operator and the partial differential equation on $M$ that describes the evolution of the capital density $K(t,x)$ when we consider a spatially homogeneous productivity $A$, is now
\begin{equation}
\label{eq:pde-state}
\displaystyle
\left \{
\begin{array}{l}
\displaystyle \frac{\partial K(t,x)}{\partial t} = \Delta_x K(t,x) +  A K(t,x) - c(t,x)\\[10pt]
\displaystyle K(0,x) = K_0(x).
\end{array}
\right.
\end{equation}
and the policy maker chooses the consumption $c(\cdot,\cdot)$ to maximize
\begin{equation}
\label{eq:functional} 
J(c(\cdot, \cdot)) = \int_0^{+\infty} e^{-\rho t} \int_M \frac{\left ( c(t,x) \right )^{1-\sigma}}{1-\sigma} \ud x \ud t.
\end{equation}

The rewriting of the problem as a dynamic optimization problem in a Hilbert space, in line with the formulation presented in Section \ref{sec:formulation}, now relies on the operator $\mathcal{A}$ on $L^2(M)$ defined as follows\footnote{The Sobolev space $W^{2,2}(M)$ is defined, for example, in Section 4 of \cite{Grigoryan12}.}
\[
\left \{
\begin{array}{l}
D(\mathcal{A}) := W^{2,2}(M)\\[5pt]
\mathcal{A}(f) = \Delta_x f.
\end{array}
\right .
\]
$\mathcal{A}$ is the (self-adjoint) generator of the heat semigroup on $L^2(M)$, which is a $C_0$ semigroup on $L^2(M)$ (see Section 4.3 of \citealp{Grigoryan12}). The state equation (\ref{eq:pde-state}) can be rewritten as an evolution equation in the Hilbert space $L^2(M)$ as:
\begin{equation}
\label{eq:hilbert-state}
\left \{
\begin{array}{l}
{y}'(t) = \mathcal{A} y(t) + A y(t) - u(t)\\[5pt]
y(0) = x
\end{array}
\right .
\end{equation}
while the functional (\ref{eq:functional}) can be rewritten similarly to (\ref{objective}).

\subsection{Dynamic models with transboundary pollution dynamics}
\label{SSE:pollution}
We illustrate a class of models in which production is associated with the emission of pollutants (typically atmospheric pollutants) that exhibit spatial dynamics. These models focus on ``local'' pollutants that can be emitted during production processes, such as particulate matter (see \citealp{tiwary2010air}), nitrogen dioxide (\citealp{eea2020airquality}), or sulfur dioxide (\citealp{smith2011sulfur}). This type of model, which describes localized spatial dynamics, is not suitable for representing the effects of pollutants such as carbon dioxide (and greenhouse gases in general), which induce undesirable effects at a global scale, depending on their global concentration.
\label{sub:example-pollution}
\subsubsection{The economic problem}
The structure of the state equation is closely related to those presented in the previous Subsection \ref{sub:example-spatial} since we will have again a parabolic state dynamics.
Following the structure in \cite{boucekkine2022managing}, we suppose again to use $S^1$ defined in (\ref{eq:defS1}) as a space support and denote, respectively, by $c(t,\theta)$, $i(t,\theta)$, and $Y(t,\theta)$ the consumption, investment (input) and production at location $\theta$ and time $t$. At each location, the output is fully allocated between consumption and local investment, with no trade between locations. This is formalized by the equality
\begin{equation}
c(t,\theta) + i(t,\theta) = Y(t,\theta). \label{res}
\end{equation}
We suppose that the production $Y(t,\theta)$ is linear in the input, i.e.
\begin{equation}
\label{resbis}
Y(t,\theta) = a(\theta) i(t,\theta)
\end{equation}
for some location-specific productivity factor $a: S^1 \to (1, +\infty). $
The sole link between locations is transboundary pollution. Specifically, the spatial profile of pollution evolves according to the following parabolic partial differential equation (PDE):
\begin{equation}
\label{SE-trans}
\begin{cases}
\displaystyle{\frac{\partial p}{\partial t}(t,\theta) = \frac{\partial}{\partial \theta} \bigg(\sigma(\theta) \frac{\partial p}{\partial \theta}  (t,\theta)\bigg) - \delta(\theta) p(t, \theta) + {\eta}(\theta) i(t,\theta),}
& \forall (t,\theta) \in \mathbb{R}^+ \times S^1, \\[10pt]
p(0,\theta) = p_0(\theta), & \forall \theta \in S^1,
\end{cases}
\end{equation}
where   
$p_0, \delta, \sigma, \eta, w: S^1 \to \mathbb{R}^+
$, $
\gamma: S^1 \to (0, 1) \cup (1, +\infty)$, $p(t,\theta)$ represents the pollution stock at location $\theta$ and time $t$, $\frac{\partial}{\partial \theta} \bigg(\sigma(\theta) \frac{\partial p}{\partial \theta}  (t,\theta)\bigg)$ is the diffusion term, $\delta(\theta) p(t, \theta)$ is the natural regeneration capacity of the environment,  ${\eta}(\theta) i(t,\theta)$ the flow of new emissions accounting for the impact of pollution per unit of input used. 

Note that
\begin{enumerate}[(i)]
    \item Setting ${\eta}(\theta) = 1$ implies that emissions are equal to the input use;
    \item With ${\eta}(\theta) = a(\theta)\phi(\theta)$, emissions are proportional to output, as ${\eta}(\theta)i(t,\theta) = \phi(\theta)Y(t,\theta)$;
    \item The spatial heterogeneity in ${\eta}(\theta)$ reflects differences in regional production technologies.
\end{enumerate}

The diffusion term $\displaystyle{\frac{\partial}{\partial \theta} \big(\sigma(\theta) \frac{\partial p}{\partial \theta}(t,\theta)\big)}$ introduces spatial variability at the location dependent diffusion speed $\sigma(\theta)$: some studies, such as \cite{tiwary2010air}, highlight the role of local geographical and meteorological factors in shaping the dispersion of pollution. 

An initial spatial distribution $p_0(\theta)$ on $S^1$ is required to define the dynamics of the pollution. The state variable $p(t,\theta)$ then evolves according to the PDE in \eqref{SE-trans}, introducing infinite-dimensional complexity into the optimization problem.

We suppose that the instantaneous utility at location $\theta$ and time $t$ combines a CES utility from consumption and a linear disutility from pollution:
\[
g\big(c(t,\theta), p(t,\theta)\big) = \frac{c(t,\theta)^{1-\gamma(\theta)}}{1-\gamma(\theta)} - w(\theta)p(t,\theta),
\]
where $\gamma(\theta)$ measures the (location-specific) inverse of the elasticity of intertemporal substitution, and $w(\theta)$ reflects local environmental awareness. From (\ref{res}) and (\ref{resbis}) we have $c(t,\theta) = (a(\theta)-1)i(t,\theta)$ and then, if we consider a Benthamite planner with a discount rate $\rho$, she maximizes the following functional: 
\begin{equation}
\label{F}
J(p_0; i) = \int_0^\infty e^{-\rho t} \left( \int_{S^1} \left[ \frac{\big((a(\theta)-1)i(t,\theta)\big)^{1-\gamma(\theta)}}{1-\gamma(\theta)} - w(\theta)p(t,\theta) \right] \mathrm{d}\theta \right) \mathrm{d}t.
\end{equation}

\subsubsection{Rewriting the problem in a Hilbert space setting}

As in Subsection \ref{sub:example-spatial}, to reformulate the problem in an infinite-dimensional setting, we use the Hilbert  space $X:=L^2(S^1;\R)$, defined in \eqref{eq:defL2}, endowed with  the natural norm $| \cdot|$ and the inner product $\langle \cdot, \cdot \rangle$. The nonnegative cone ${X}^+$ is defined in \eqref{eq:defH+}. Additionally, we introduce the Sobolev space:
\[
\displaystyle
W^{2,2}(S^1; \mathbb{R}) := \left\{ f \in L^2(S^1; \mathbb{R}): f \text{ is twice weakly differentiable, and } f', f'' \in X \right\}.
\]
We assume that $\delta \in C(S^1; \mathbb{R}^+)$ and $\sigma \in C^1(S^1; (0, +\infty))$. Then, the operator $\mathcal{A}: D(\mathcal{A}) \subset X \to X$  defined as
\[
D(\mathcal{A}) = W^{2,2}(S^1; \mathbb{R}), \quad (\mathcal{A} \varphi)(\theta) = \left( \sigma(\theta) \varphi'(\theta) \right)' - \delta(\theta) \varphi(\theta), \quad \varphi \in D(\mathcal{A})
\]
 is a closed, densely defined, self-adjoint, unbounded linear operator on the space $H$ (see, e.g. \citealp{lunardi1995analytic}, p. 71-75, Sections 3.1 and 3.1.1). Hence,  $\mathcal{A}$ generates a strongly continuous contraction semigroup $(e^{t\mathcal{A}})_{t \geq 0} \subset \mathcal{L}(H)$ and the state dynamics is formally reformulated in $H$ as 
\[
y'(t) = \mathcal{A} y(t) + \Xi u(t), \quad y(0) = x,
\]
where $y(t) = p(t, \cdot)$, $\Xi = \eta(\cdot)$, $u(t) = i(t, \cdot)$ and $x = p(0, \cdot)$. We can then consider the following set of admissible controls
$$
\mathcal{U}(x)\equiv \mathcal{U}:=\bigg\{u\in
L^1_{loc}(\R^+;{X}^+): \  \int_0^\infty {e^{-\rho t} |\Xi(t)u(t)|}\, \ud t <\infty\bigg\}.
$$
If we suppose that $w\in C(S^1;(0,+\infty))$, that the function $\gamma:S^1\to(0,1)\cup(1,+\infty)$ is measurable and bounded with values in a compact subset of $(0,1)$ or of a compact subset of $(1, +\infty)$, and that $a(t,\theta)$ and ${\eta}(\theta)$ are bounded, 
then,
setting $A:=a(\cdot)$, $ \Gamma:=\gamma(\cdot)$, $\mathbf{1}$ the indicator function of $S^1$, and
$$ \left[\frac{\big((A-\mathbf{1}) u(t)\big)^{1- \Gamma}}{1-\Gamma}\right](\theta):=\frac{\big((a(\theta)-1) i(t, \theta)\big)^{1- \gamma(\theta)}}{1-\gamma(\theta)}, \ \ \ \theta\in S^1,$$
the functional \eqref{F} is rewritten in this formalism as
\begin{equation}\label{FH}
J(x,u)=\int_0^{\infty} e^{-\rho t} \left[ \left\langle 
\frac{\big((A-\mathbf{1}) u(t)\big)^{1- \Gamma}}{1-\Gamma},\mathbf{1}\right\rangle - \left\langle w, y(t)\right\rangle \right] d t.
\end{equation}

\subsubsection{The infinite-dimensional HJB equation and its explicit solution}



The stationary infinite-horizon HJB equation associated with our problem, following the general formulation in \eqref{eq:HJBbisst}, can be written as:
\begin{equation}\label{eq:HJB-pollution}
\rho v(x) = \langle x, \mathcal{A} Dv(x) \rangle - \langle w, x \rangle + \sup_{u \in {X}^+} \left\{ \Psi(u) + \langle \Xi i, Dv(x) \rangle \right\},
\end{equation}
where $x \in {X}^+$ is the state (the pollution stock), $Dv(x)$ is the Fréchet derivative of $v$ at $x$, and $\Psi(I)$ is the utility from consumption, which is given by
\[
\Psi(I) = \left\langle \frac{((A-\mathbf{1})I)^{1-\Gamma}}{1-\Gamma}, \mathbf{1} \right\rangle.
\]

\medskip

Inspired by the structure of the problem, we seek a solution to the HJB equation that is affine in the state $x$. Let us first define an auxiliary function $\alpha \in D(\mathcal{A})$ as the unique solution (Theorem 8.3 in \citealp{GilbargTrudinger2001}) to the following elliptic equation:
\begin{equation} \label{eq:alpha-def-pollution}
(\rho - \mathcal{A})\alpha(\theta) = w(\theta), \quad \forall \theta \in S^1.
\end{equation}
Given that $w(\theta)>0$ and $\rho>0$, it can be shown (Theorem 8.19 in \citealp{GilbargTrudinger2001}) that $\alpha(\theta) > 0$ for all $\theta \in S^1$. This function $\alpha(\theta)$ can be interpreted as the discounted total future disutility generated by a unit of pollutant emitted at location $\theta$.

We can now state the explicit solution to the HJB equation.

\begin{Proposition}\label{prop:HJB-solution-pollution}
Assume that $\delta \in C(S^1; \mathbb{R}^+)$ and $\sigma \in C^1(S^1; (0, +\infty))$,  that $\gamma: S^1\to(0,1)\cup(1,+\infty)$ is measurable and bounded with values in a compact subset of $(0,1)$ or in a compact subset of $(1, +\infty)$, and that  $a: S^1 \to (1, \infty)$ and $\eta, w: S^1 \to \mathbb{R}^+$ are measurable and  bounded.

Let $\alpha \in D(\mathcal{A})$ be the solution to \eqref{eq:alpha-def-pollution}. Define the constant $q$ as
\begin{equation} \label{eq:q-def-pollution}
q := 
\frac{1}{\rho} \int_{S^1} \sup_{u(\theta) \geq 0} \left\{ \frac{((a(\theta)-1)u(\theta))^{1-\gamma(\theta)}}{1-\gamma(\theta)} - \eta(\theta)u(\theta)\alpha(\theta) \right\} d\theta.
\end{equation}
The function
\begin{equation}\label{eq:ValueFunction-pollution}
  v(x) = - \langle \alpha, x \rangle + q, \qquad x \in H,
\end{equation}
is a classical solution to the \emph{HJB} equation \eqref{eq:HJB-pollution} on $X$.
\end{Proposition}

\begin{proof}
We verify that the affine function \eqref{eq:ValueFunction-pollution} solves the HJB equation \eqref{eq:HJB-pollution}. The differential of $v$ is constant with respect to $x$ and is given by
\[
Dv(x) = -\alpha.
\]
Since $\alpha \in D(\mathcal{A})$, the regularity requirements are met. Substituting $v(x)$ and $Dv(x)$ into the HJB equation, the left-hand side becomes:
\[
\rho v(x) =  -\rho \langle \alpha, x \rangle + \rho q.
\]
For the right-hand side, using the fact that $\mathcal{A}$ is self-adjoint, we have (the secondequility follows from (\ref{eq:alpha-def-pollution})):
\[
\langle x,  \mathcal{A} Dv(x) \rangle = -\langle x, \mathcal{A}\alpha \rangle =  -\langle x, \rho\alpha - w \rangle = -\langle x, \rho\alpha \rangle + \langle x, w \rangle
\]
The supremum term in the HJB equation becomes:
\[
\sup_{u \in {X}^+} \left\{ \Psi(u) + \langle \Xi u, Dv(x) \rangle \right\} - \langle w, x \rangle = \sup_{u \in {X}^+} \left\{ \Psi(u) - \langle \Xi u, \alpha \rangle \right\} - \langle w, x \rangle.
\]
The maximization problem can be solved pointwise for each $\theta \in S^1$:
\[
\sup_{u(\theta) \geq 0} \left\{ \frac{((a(\theta)-1)u(\theta))^{1-\gamma(\theta)}}{1-\gamma(\theta)} - \eta(\theta)i(\theta)\alpha(\theta) \right\}.
\]
The first-order condition yields the optimal investment function $i^*(\theta)$, which is independent of the state $x$:
\begin{equation} \label{eq:optimal-investment-pollution}
u^*(\theta) = \frac{1}{a(\theta)-1} \left( \eta(\theta) \alpha(\theta) \right)^{-\frac{1}{\gamma(\theta)}}.
\end{equation}
Plugging this back, the value of the supremum is precisely $\rho q$ by definition \eqref{eq:q-def-pollution}; so, 
\[
\sup_{u \in {X}^+} \left\{ \Psi(u) + \langle \Xi u, Dv(x) \rangle \right\} - \langle w, x \rangle = - \langle w, x \rangle + \rho q.
\]
Combining all parts, the HJB equation becomes:
\[
-\rho \langle \alpha, x \rangle + \rho q = -\langle x, \rho\alpha \rangle + \langle x, w \rangle - \langle w, x \rangle + \rho q
\]
which is an identity. This completes the proof.
\end{proof}

\subsubsection{The solution of the model}

The explicit solution of the HJB equation allows us to characterize the optimal policy for the planner. A key feature of this model, stemming from the affine structure of the value function, is that the optimal control is determined in an open-loop fashion; it does not depend on the current state of the system, the pollution stock.

If we express back the optimal control and the optimal trajectory in terms of original formulation/notation we have the following result.

\begin{Theorem}
\label{th:optimal-control-pollution}
Suppose that the assumptions of Proposition \ref{prop:HJB-solution-pollution} are verified and that $p(0, \cdot) \in {X}^+$. Then, the optimal control (the investment/input level) is constant over time and it is given by $i^*(\cdot)$ where
\begin{equation} \label{eq:optimal-I-pollution-final}
i^*(\theta) = \frac{1}{a(\theta)-1} \left( \eta(\theta) \alpha(\theta) \right)^{-\frac{1}{\gamma(\theta)}}, \quad \forall \theta \in S^1.
\end{equation}
The optimal (pollution) trajectory $p^*(t,\theta)$ is the solution to the linear non-homogeneous PDE:
\begin{equation}
\label{SE-transbis}
\begin{cases}
\displaystyle{\frac{\partial p}{\partial t}(t,\theta) = \frac{\partial}{\partial \theta} \bigg(\sigma(\theta) \frac{\partial p}{\partial \theta}  (t,\theta)\bigg) - \delta(\theta) p(t, \theta) + {\eta}(\theta) i^*(\theta),}
& \forall (t,\theta) \in \mathbb{R}^+ \times S^1, \\[10pt]
p(0,\theta) = p_0(\theta), & \forall \theta \in S^1,
\end{cases}
\end{equation}Furthermore, the optimal welfare for an initial pollution stock $p_0$ is given by
\[
- \langle \alpha, p_0 \rangle + q,
\]
where $q$ is the constant defined in \eqref{eq:q-def-pollution} and $\alpha$ is defined by (\ref{eq:alpha-def-pollution}).
\end{Theorem}
\begin{proof}
The candidate optimal control $i^*$ is derived directly from the maximization of the Hamiltonian in the proof of Proposition \ref{prop:HJB-solution-pollution}. It is admissible : (i) it is positive (ii) the integral in de definition of $J$ is bounded since $i^*$ is time-independent (iii)  By the maximum principle for parabolic equations $p$ remains positive since it is the solution of a parabolic equation with a positive source term ($\eta(\theta) i^*(\theta) \geq 0$ for all $\theta$) with non-negative intiial condition. 

The result then follows by applying Theorem \ref{teo:verst}. The transversality condition required by the theorem is satisfied due to the discounting factor $\rho > 0$.
\end{proof}

\subsubsection{Variants and Literature}
We followed here the structure of the model presented by \cite{boucekkine2022managing}, which generalizes the approaches of \cite{boucekkine2021firm, boucekkine2019geographic}, emphasizing the role of location-dependent parameters in shaping optimal strategies. 
The model draws on a classical idea in theoretical ecology, where pollutant diffusion is captured via parabolic equations that describe the spatial and temporal spread of contaminants 
\cite{Jacob1999,Thibodeaux1996}.

Previous contributions in the same spirit include 
\cite{anita2015dynamics}, which examined a spatially structured economic growth model that incorporates pollution diffusion and optimal taxation, focusing on the long-term behavior of the system, and the series of papers \cite{de2019spatial, de2019spatial-2, de2021equilibrium, de2022investment}, where the problem is addressed in the (more complex) context of a differential Nash game instead of the case of social optimum (and thus optimal control).

In a certain sense, pollution can be thought of as the "opposite" of a natural resource (the natural resource being clean air or water). In this perspective, these types of models can be interpreted as being closely related to optimal exploitation models of natural resources with a spatial dynamic, which makes them, in this specific sense, akin to optimal fishery management models: in the theoretical economics literature, optimal fishing models often contain an intertemporal optimization problem with a parabolic equation describing the spatio-temporal dynamics of the resource (see, e.g., \citealp{smith2009economics, brock2014optimal-ARRE}).

\subsection{Vintage capital: delay differential equation approach}
\label{sub:vintagecapitaldelay}

In this section, we will examine an optimal growth problem in the context of a vintage capital model. The model will feature a state equation with delay, which gives it a structure distinct from the models described earlier. \emph{Vintage capital models}, first introduced by \cite{johansen1959substitution} and \cite{Solow1960}, are models in which the stock of capital, at any point in time, is described as a heterogeneous set of different generations of machines. One of the intuitions behind this approach is the idea that technical progress is \emph{embodied} in capital goods: a new technology influences the production process only if new machines use that technology. This leads to distinguish generations of machines that are replaced over time, emphasizing the importance of stratification by age of capital\footnote{We refer the interested reader to \cite{BoucekkineDeLaCroixLicandro2011} for the (fascinating) history of the debate on the \emph{embodiment hypothesis} and the subsequent developments in the growth literature within the context of vintage capital models.}

\subsubsection{The economic model}

In this subsection, we describe a model presented by \cite{Boucekkineetal05}, characterized by a \emph{one-hoss shay} depreciation hypothesis, where capital loses its value abruptly: it is assumed that, at any time $t$, only investments made during the interval $(t-T, t)$ contribute to production, while those older than the exogenous and constant scrapping time $T$ become unproductive. In other words, the equation linking the evolution of capital dynamics to the flow of investments is:
\begin{equation}
\label{eq:VCM-accumulazionek}
k(t) = \int_{t-T}^t i(s) \, {\mathrm{d}}s.
\end{equation}
The rest of the model mirrors a standard $AK$ model: production linear and given by $Ak(t)$ where the total factor productivity $A$ is constant and exogenous. A planner decides how to allocate production between investment and consumption, according to the budget relation $Ak(t) = i(t) + c(t)$, in order to maximize an aggregate social welfare functional:
\begin{equation}
\label{eq:VCM-funzionaleinc}
\int_0^{\infty} e^{-\rho t} \frac{c(t)^{1-\sigma}}{1-\sigma} \, {\mathrm{d}}t,
\end{equation}
for some positive $\sigma \neq 1$. If we rewrite (\ref{eq:VCM-accumulazionek}) more formally, including an initial investment datum $\bar{\iota}(\cdot) \in L^2(-T, 0)$ over the interval $[-T, 0)$, the problem is characterized by the following state equation:
\begin{equation}
\label{eqboucekkine} 
\left\{
\begin{array}{ll}
k'(t) = i(t) - i(t-T), &  t \geq 0, \\[3pt]
i(s) = \bar{\iota}(s), & s \in [-T, 0), \\[3pt]
k(0) = \int_{-T}^0 \bar{\iota}(s) \, {\mathrm{d}}s,
\end{array} 
\right.
\end{equation}
where $i(t)$ for $t \geq 0$ is the control variable. Substituting the budget constraint $c(t) = y(t) - i(t)$ into (\ref{eq:VCM-funzionaleinc}), we can write the functional to maximize as:
\[
J(\bar{\iota}(\cdot); i(\cdot)) = \int_0^\infty e^{-\rho t} 
\frac{(Ak(t) - i(t))^{1-\sigma}}{1-\sigma} \, {\mathrm{d}}t,
\]
over the set:
\[
\mathcal{I}_{\bar{\iota}} = \big\{ i(\cdot) \in L^2_{\mathrm{loc}}([0,+\infty); \mathbb{R}^+)  : \ 
i(t) \in [0, Ak(t)] \ \forall t \in \mathbb{R}^+ \big\}.
\]
Note that the constraint $i(t) \in [0, Ak(t)]$ in the admissible control set includes the usual non-negativity constraints $c \geq 0$ and $i \geq 0$, which are inherent to the age structure (a negative amount of any specific vintage is not feasible). This, in turn, implies $k, y \geq 0$. The choice of $\mathcal{I}_{\bar{\iota}}$ ensures that $k(\cdot) \in W_{\mathrm{loc}}^{1,2}(0,+\infty; \mathbb{R}^+)$ for every $i(\cdot) \in \mathcal{I}_{\bar{\iota}}$, and, in particular, $k(\cdot)$ is continuous (as is evident from the integral form of (\ref{eq:VCM-accumulazionek})).

\subsubsection{Rewriting the problem in a Hilbert space setting}\label{sub:aas}
We introduce the infinite-dimensional space in which we reformulate the problem. It is defined as:
\[
{X} := {\mathbb{R}} \times L^2(-T,0;\R).
\]
A generic element $x\in{X}$ is denoted as a pair $x=(x^0, x^1)$. The scalar product on ${X}$ is the one defined for the product of Hilbert spaces, i.e.,
\[
\langle ({x}^0, {x}^1), ({z}^0, {z}^1) \rangle  := {x}^0 {z}^0 + \langle {x}^1, {z}^1 \rangle _{L^2},
\]
for every $({x}^0, {x}^1), ({z}^0, {z}^1) \in {X}$.
We define the operator $\mathcal{A}$ on ${X}$, which will be the adjoint of the generator of the $C_0$ semigroup that appears in the state equation of our problem:
\begin{equation}
\label{eq:defA-VCM}
\left\{
\begin{array}{ll}
D(\mathcal{A}) := \{ (\psi^0, \psi^1) \in {X} : \psi^1 \in  W^{1,2}(-T,0; {\mathbb{R}}), \; \psi^0 = \psi^1(0) \}, \\[4pt]
\mathcal{A}\colon D(\mathcal{A})\subseteq X \rightarrow {X}, \quad
\mathcal{A}(\psi^0, \psi^1) := (0, \frac{{\mathrm{d}}}{{\mathrm{d}}s} \psi^1).
\end{array}
\right.
\end{equation}

We also define the following three operators:
\[
\left\{
\begin{array}{ll}
B\colon D(\mathcal{A}) \rightarrow {\mathbb{R}}, \\[4pt]
B(\psi^0, \psi^1) := \psi^1(0) - \psi^1(-T),
\end{array}
\right.
\]
\begin{equation}
\label{defdiF}
\left\{
\begin{array}{ll}
F\colon L^2([-T,0); \mathbb{R}) \to L^2([-T,0); \mathbb{R}), \\[4pt]
F(z)(s) := -z(-T-s),
\end{array}
\right.
\end{equation}
and
\[
\left\{
\begin{array}{ll}
R\colon L^2([-T,0); \mathbb{R}) \to {\mathbb{R}}, \\[4pt]
R(z) := \int_{-T}^0 z(s) {\mathrm{d}}s.
\end{array}
\right.
\]

It will be useful to relax the relationship between $\iota$ and the initial datum of $k$ (we will use this for defining the solution of the HJB equation on an open set). Specifically, we will not necessarily impose $k(0) = \int_{-T}^0 \bar{\iota}(s) {\mathrm{d}}s$. The restriction will be re-introduced when we will come back to the applied problem. To achieve this, given $(k_0, \bar{\iota}) \in {X}$ and $i(\cdot) \in L^2_{\mathrm{loc}}([0,+\infty); \mathbb{R}^+)$, we define:
\begin{equation}
\label{eqintegralebis}
k_{k_0,\bar{\iota},i}(t) = k_0 - \int_{(-T+t) \wedge 0}^0 \bar{\iota}(s) {\mathrm{d}}s + \int_0^t i(s) {\mathrm{d}}s.
\end{equation}

Following \cite{VinterKwong}, we introduce now the structural state of the system, which connect the system's dynamics as a delay equation with the dynamics of a suitable evolution equation in ${X}$:
\begin{Theorem}
\label{thinmp}
Assume that $\bar{\iota} \in L^2([-T,0); \mathbb{R}^+)$, $i \in L^2_{\mathrm{loc}}([0,+\infty); \mathbb{R}^+)$, and $k_0 \in {\mathbb{R}}$, with $y = (k_0, F(\bar{\iota}))$. Then, for every $T > 0$, the \emph{structural state} 
\[
y_{y,i}(t) = (y_{y,i}^0(t), y_{y,i}^1(t)) = (k_{k_0,\bar{\iota},i}(t), F(i(t+\cdot)|_{[-r,0]}) 
\]
is the unique solution in
\begin{equation}
\label{defdiPi}
\Pi := \bigg\{ f \in C(0,T; {X}) \, : \,  f' \in L^2(0,T; D(\mathcal{A})') \bigg\},
\end{equation}
to the equation:
\begin{equation}
\label{eqBDDM}
\left\{
\begin{array}{ll}
 y'(t) = \mathcal{A}^* y(t) + {B}^* i(t), \quad t > 0,\\[4pt]
y(0) = x = (k_0, F(\bar{\iota})),
\end{array}
\right.
\end{equation}
where $\mathcal{A}^*$ and $B^*$ are the adjoints  of the operators $\mathcal{A}\colon D(\mathcal{A}) \rightarrow {X}$ and $B\colon D(\mathcal{A}) \rightarrow {\mathbb{R}}$.
\end{Theorem}
See  Theorem 5.1, p. 258 of \cite{BDDM07} for a proof.

\subsubsection{The infinite-dimensional HJB equation and its explicit solution}
\label{sub:VCM_HJB}

To find the solution to the optimal control problem, we formulate the associated Hamilton-Jacobi-Bellman (HJB) equation in the infinite-dimensional space ${X}$. The analysis requires two key hypotheses that parallel those found in standard growth models. The first ensures that sustained growth is possible, while the second guarantees the finiteness of the value function.

To this end, let $\xi$ be the unique strictly positive root of the characteristic equation 
\[
z = A(1 - e^{-zT}). 
\]
Such a root exists if and only if the following condition holds.
\begin{Assumption}
\label{hyp:VCM_growth}
$A T > 1$.
\end{Assumption}

The root $\xi$ represents the maximum possible asymptotic growth rate of the capital stock when all output is reinvested. The second assumption ensures that the planner's discount rate $\rho$ is sufficiently high to prevent the utility integral from diverging.

\begin{Assumption}
\label{hyp:VCM_finiteness}
$\rho > \xi(1-\sigma)$.
\end{Assumption}
This is the counterpart of the standard finite-utility assumption in a one sector $AK$ model

The HJB equation for the value function $V(y)$ associated with the structural state $y = (k, F(\bar{\iota})) \in {X}$ is:
\begin{equation}
\label{eq:VCM_HJB}
\rho v(x) = \sup_{i \in [0, A x^0]} \left\{ \frac{(A x^0 - i)^{1-\sigma}}{1-\sigma} + \langle  D  v(x), \mathcal{A}^* x + B^* i \rangle_{{X}} \right\}.
\end{equation}
We try now to look for an explicit solution of this equation.

First, we define the functional $\Gamma_0: {X} \to \mathbb{R}$, which acts as a scalar ``equivalent capital'' stock:
\begin{equation}
\label{eq:VCM_Gamma0}
\Gamma_0(x) := x^0 + \int_{-T}^0 e^{\xi s} x^1(s) \,{\mathrm{d}}s,
\end{equation}
where $x = (x^0, x^1)$ is a generic element of ${X}$. The solution to the HJB equation can be defined on the open set 
\[
\mathcal{O} := \left\{ x = (x^0, x^1) \in {X} \;:\; \  \Gamma_0(x) > 0 \quad \text{and} \quad A x^0 > \nu^{-1/\sigma} \left(\frac{A}{\xi}\right)^{1/\sigma} \Gamma_0(x) \right\}.
\]

\begin{Proposition}
\label{prop:VCM_HJB_solution}
Suppose that Assumptions \ref{hyp:VCM_growth} and \ref{hyp:VCM_finiteness} are verified. Define  $v: \mathcal{O} \to \mathbb{R}$ as the function
\begin{equation}
\label{eq:VCM_ValueFunction}
v(x) = \nu \frac{[\Gamma_0(x)]^{1-\sigma}}{1-\sigma},
\end{equation}
where 
\[
\nu = \left(\frac{\rho-\xi(1-\sigma)}{\sigma}\right)^\sigma \left(\frac{A}{\xi}\right)^{1-\sigma}.
\]
Then $v$ 
is a classical solution to the HJB equation \eqref{eq:VCM_HJB} on the open subset $\mathcal{O}$.
\end{Proposition}
\begin{proof}
We directly check that, for some positive $\nu$ that we will characterize, we can find a solution of the form \eqref{eq:VCM_ValueFunction}.

If $v$ is in the form \eqref{eq:VCM_ValueFunction} then it is continuously Fr\'echet differentiable on the open set $\mathcal{O}$ and its differential is given by
\begin{equation}
\label{eq:VCM_gradient}
 D  v(x) = \nu [\Gamma_0(x)]^{-\sigma}  D  \Gamma_0(x),
\end{equation}
where 
\[
D  \Gamma_0(x) = (1, \{s \mapsto e^{\xi s}\}) \in {X}.
\]
One can easily prove that $ D  \Gamma_0(x) \in D(\mathcal{A})$ (defined in \eqref{eq:defA-VCM}) and then the gradient $ D  v(x)$ also belongs to $D(\mathcal{A})$ for all $x \in \mathcal{O}$ and $D  v$ is continuous from $\mathcal{O}$ to $D(A)$ (endowed with the graph norm) satisfying the regularity requirements for a classical solution.

Since the function id concave, the maximizer of the Hamiltonian, which gives the candidate optimal control $i^*(x)$, is found by taking the first-order condition with respect to $i$ in the expression inside the supremum of \eqref{eq:VCM_HJB}:
\[
0= \frac{\partial}{\partial i} \left[ \frac{(A x^0 - i)^{1-\sigma}}{1-\sigma} + i B( D  v(x)) \right] = -(A x^0 - i)^{-\sigma} + B( D  v(x)).
\]
This gives the condition $(A x^0 - i)^{-\sigma} = B( D  v(x))$.
Since $B$ is the evaluation map $B(\psi) = \psi^1(0) - \psi^1(-T)$, its application to the gradient yields:
\[
B( D  v(x)) = \nu [\Gamma_0(x)]^{-\sigma} B( D  \Gamma_0(x)) = \nu [\Gamma_0(x)]^{-\sigma} (e^{\xi \cdot 0} - e^{\xi (-T)}) = \nu [\Gamma_0(x)]^{-\sigma} (1 - e^{-\xi T}).
\]
Using the characteristic equation for the root $\xi$, we have $1 - e^{-\xi T} = \xi/A$. Thus, we find
\[
B( D  v(x)) = \nu [\Gamma_0(x)]^{-\sigma} \frac{\xi}{A}.
\]
From the first-order condition, the optimal investment $i^*(x)$ is given by
\[
i^*(x) = A x^0 - [B( D  v(x))]^{-1/\sigma},
\]
which explicitly becomes:
\begin{equation}
\label{eq:VCM_i_star}
i^*(x) = A x^0 - \left( \nu [\Gamma_0(x)]^{-\sigma} \frac{\xi}{A} \right)^{-1/\sigma} = A x^0 - \nu^{-1/\sigma} \left(\frac{A}{\xi}\right)^{1/\sigma} \Gamma_0(x)
\end{equation}
(remark that $i^*(x)$ is in $(0, Ax^0)$ thanks to definition of $\mathcal{O}$),
\[
(A x^0 - i^*(x))^{1-\sigma} = B( D  v(x))^{-(1-\sigma)/\sigma} =\left( \nu [\Gamma_0(x)]^{-\sigma} \frac{\xi}{A} \right)^{(1-\sigma)/\sigma}
\]
and
\begin{align*}
i^*(x) B(\nabla v(x)) &= \left( A x^0 - \nu^{-1/\sigma} \left(\frac{A}{\xi}\right)^{1/\sigma} \Gamma_0(x) \right) \left( \nu [\Gamma_0(x)]^{-\sigma} \frac{\xi}{A} \right) \\&= \xi \nu x^0 [\Gamma_0(x)]^{\sigma} - \left( \nu [\Gamma_0(x)]^{\sigma} \frac{\xi}{A} \right)^{-(1-\sigma)/\sigma}.
\end{align*}
Let us compute the term $\langle  D  v(x), \mathcal{A}^*x \rangle_{{X}} = \langle \mathcal{A}( D  v(x)), x \rangle_{{X}}$. From the definition of operator $\mathcal{A}$, we have
\[
\mathcal{A}( D  v(x)) = \nu [\Gamma_0(x)]^{-\sigma} \mathcal{A}( D  \Gamma_0(x)) = \nu [\Gamma_0(x)]^{-\sigma} (0, \{s \mapsto \xi e^{\xi s}\}).
\]
Therefore,
\[
\langle \mathcal{A}( D  v(x)), x \rangle_{{X}} = \nu [\Gamma_0(x)]^{-\sigma} \int_{-T}^0 \xi e^{\xi s} x^1(s) \,{\mathrm{d}}s = \xi \nu [\Gamma_0(x)]^{-\sigma} (\Gamma_0(x) - x^0).
\]

Plugging the expressions for each term into the HJB equation 
\[
\rho v(x) = \langle \mathcal{A} D  v(x),  x \rangle + \sup_{i \in [0, A x^0]} \left\{ \frac{(A k - i)^{1-\sigma}}{1-\sigma} + \langle  D  v(y), \mathcal{B}^* i \rangle \right\}.
\]
we have that $v$ is a solution if and only if
\begin{align*}
\rho \nu \frac{[\Gamma_0(x)]^{1-\sigma}}{1-\sigma} = &\; \xi \nu [\Gamma_0(x)]^{-\sigma} (\Gamma_0(x) - x^0) \\
&+ \frac{\sigma}{1-\sigma}\left(\nu [\Gamma_0(x)]^{-\sigma} \frac{\xi}{A}\right)^{(\sigma-1)/\sigma} + x^0 \nu [\Gamma_0(x)]^{-\sigma} \xi.
\end{align*}
Now, we multiply the entire equation by $\frac{1-\sigma}{\nu [\Gamma_0(x)]^{-\sigma}}$ to simplify:
\begin{align*}
\rho [\Gamma_0(x)] = &\; \xi (1-\sigma) (\Gamma_0(x) - x^0) \\
&+ \frac{\sigma}{\nu} \left(\nu \frac{\xi}{A}\right)^{(\sigma-1)/\sigma} \Gamma_0(x) + (1-\sigma) \xi x^0.
\end{align*}
Let us expand and rearrange the terms:
\[
\rho \Gamma_0(x) = \xi (1-\sigma) \Gamma_0(x) - \xi (1-\sigma) x^0 + \frac{\sigma}{\nu} \left(\nu \frac{\xi}{A}\right)^{(\sigma-1)/\sigma} \Gamma_0(x) + (1-\sigma) \xi x^0.
\]
The terms involving $x^0$ cancel each other out so we are left with an equation where we can factor out $\Gamma_0(x)$:
\[
\rho \Gamma_0(x) = \left[ \xi(1-\sigma) + \sigma \nu^{-1/\sigma} \left(\frac{\xi}{A}\right)^{(\sigma-1)/\sigma} \right] \Gamma_0(x).
\]
which holds for any $x \in \mathcal{O}$ if and only if
\[
\nu = \left(\frac{\rho-\xi(1-\sigma)}{\sigma}\right)^\sigma \left(\frac{A}{\xi}\right)^{1-\sigma}.
\]
This confirms that the function $v(x)$ with this specific constant $\nu$ is indeed a solution to the HJB equation.
\end{proof}

\begin{Remark}
The (candidate, so far) value function is homogeneous of degree $1-\sigma$, not in the capital stock $k$ itself, but in the ``equivalent capital'' $\Gamma_0(x)$. This quantity discounts the past investments (contained in $x^1$) using the maximal growth rate $\xi$, effectively translating the entire vintage capital structure into a single scalar value.
\end{Remark}

\subsubsection{The solution of the model}
\label{sub:VCM_solution}

Before presenting the main results on the solution of the optimal control problem, we make some remarks on the state constraints. In our Hilbert space formulation, we have not explicitly included constraints such as the positivity of the capital stock, $k(t) \ge 0$, which corresponds to $x^0(t) \ge 0$. Moreover, the non-negativity of the investment control, $i(t) \ge 0$, is a natural requirement in a vintage capital model, where each past investment corresponds to a distinct generation of machines.

Let's consider an initial state $x=(k_0, F(\bar{\iota}))$ such that $k_0 > 0$ and the past investments are non-negative, i.e., $\bar{\iota}(s) \ge 0$ for $s \in [-T, 0)$. This implies that the initial structural state $x$ has a positive first component $x^0 > 0$ and a second component $x^1$ which is non-positive.
The candidate optimal feedback control is given by $\phi(x) = i^*(x)$ as defined in \eqref{eq:VCM_i_star}. By the very definition of the domain $\mathcal{O}$ in which we work, the feedback $\phi(x)$ is strictly positive for any $x \in \mathcal{O}$.
If a trajectory $y(t)$ starting from $x \in \mathcal{O}$ remains in $\mathcal{O}$ for all future times, then the investment control $i(t) = \phi(y(t))$ will always be positive. Consequently, the capital stock $k(t) = y^0(t)$, being an integral of past positive investments over the time window of length $T$, will also remain positive.
Therefore, the crucial step is to prove that the trajectory remains within the domain $\mathcal{O}$.

To make the argument cleaner, we state a first verification theorem, which directly follows from the general theory, where we suppose the trajectory to remain in $\mathcal{O}$. Then we will give a sufficient condition (on the parameters) to ensure that trajectories starting from economically meaningful initial data (i.e., with $k_0 > 0$ and $\bar{\iota}(s) \ge 0$) remain in $\mathcal{O}$. 

\begin{Theorem}[Verification Theorem]
\label{th:VCM_verification}
Suppose that Assumptions \ref{hyp:VCM_growth} and \ref{hyp:VCM_finiteness} hold. Let $\phi: \mathcal{O} \to \mathbb{R}$ be the feedback map defined by \eqref{eq:VCM_i_star}. Let $x \in \mathcal{O}$ be an initial state with $x^0>0$ and $x^1 \le 0$. If the unique solution $y(\cdot)$ to the closed-loop system
\[
 y'(t) = \mathcal{A}^* y(t) + \mathcal{B}^* \phi(y(t)), \quad y(0) = x,
\]
remains in $\mathcal{O}$ for all $t \ge 0$, then the control $u^*(t) = \phi(x(t))$ is optimal for the infinite-dimensional problem, and the value function is given by $V(x) = v(x)$, where $v$ is defined in \eqref{eq:VCM_ValueFunction}.
\end{Theorem}
\begin{proof}
The result follows directly from applying Theorem \ref{teo:verst}. Remark that the demanded transversality is satisfied due to Assumption \ref{hyp:VCM_finiteness}.
\end{proof}

The next result provides a sufficient condition on the model's parameters for ensuring that the trajectory driven by the (candidate optimal) feedback remains in $\mathcal{O}$, thus providing a complete solution to the problem. The condition is expressed in the following assumption.

\begin{Assumption}
\label{hyp:VCM_interior}
$\displaystyle \frac{\rho - \xi(1-\sigma)}{\sigma} < A$.
\end{Assumption}

\begin{Theorem}
\label{th:VCM_main_solution}
Suppose that Assumption \ref{hyp:VCM_growth}, \ref{hyp:VCM_finiteness}, and \ref{hyp:VCM_interior} hold.
Let $x = (k_0, F(\bar{\iota}))$ be an initial state such that $k_0 > 0$ and $\bar{\iota}(s) \ge 0$ for a.e. $s \in [-T,0)$.
Then the closed-loop system driven by the feedback map $\phi(x)$ from \eqref{eq:VCM_i_star} has a unique solution $y(\cdot)$ starting from $x$, and this solution remains in $\mathcal{O}$ for all $t \ge 0$, moreover $y^0$ and $i$ remains positive.
Consequently, the control $i^*(t) = \phi(y(t))$ is optimal for the original economic/DDE problem (\ref{eqboucekkine})-(\ref{eq:VCM-funzionaleinc}) with positivity contraints for capital, investment (and consumption), the related value function is $V(x)=v(x)$.
\end{Theorem}
\begin{proof}
We need to show that if we start from $x \in \mathcal{O}$, the trajectory $y(t)$ remains in $\mathcal{O}$ (our discussion before Theorem \ref{th:VCM_verification} already shows that, in this case, if one starts from positive $k=x^0$ and $\bar\iota$ than the system will maintain positive capital and investment over time). This means we need to prove that $i(t) = \phi(y(t)) > 0$ for all $t \ge 0$. 

Along the candidate optimal trajectory, from \eqref{eq:VCM_i_star} and using the expressions for the state components from Theorem \ref{thinmp} and equation \eqref{eqintegralebis}, we have:
\begin{multline}
i(t) = A y^0(t) - \left( \nu [\Gamma_0(y(t))]^{-\sigma} \frac{\xi}{A} \right)^{-1/\sigma} = A y^0(t) - \nu^{-1/\sigma} \left(\frac{A}{\xi}\right)^{1/\sigma} \Gamma_0(y(t))\\
= \left(A - \frac{\rho-\xi(1-\sigma)}{\sigma\xi/A} \right) \left( \int_{t-T}^t i(s) \ud s \right) - \frac{\rho-\xi(1-\sigma)}{\sigma\xi/A} \int_{-T}^0 e^{\xi s} F(i(t+\cdot))(s) \ud s \\
=\int_{-T}^{0} \left[ A - \left( \frac{\rho-\xi(1-\sigma)}{\sigma} \frac{A}{\xi} \right) \left( 1 - e^{\xi(-T -s)} \right) \right] i(t+s) \ud s.
\end{multline}
Thanks to Assumption \ref{hyp:VCM_interior}, the weight multiplying $i(s)$ inside the integral is strictly positive for every $s$. This implies that if the investment path $i(\cdot)$ has been positive in the past (or non-negative and not identically zero), then the resulting $i(t)$ must be strictly positive. It is therefore impossible for the investment to reach zero (or become negative) for the first time.

Therefore, the trajectory remains in $\mathcal{O}$, and by Theorem \ref{th:VCM_verification}, the proof is complete.
\end{proof}

\subsubsection{Variants and Literature}\label{vandl:delay}
The study of vintage capital models with delay dynamics in an intertemporal optimization context began with the optimal replacement model proposed by \cite{Malcomson1975}, while one of the first contributions in the realm of optimal growth theory comes from \cite{Benhabib1991}.

In the model considered here, the scrapping time $T$ is exogenous and constant, whereas in several other models (such as those already cited, or also \cite{Boucekkine1997} and \cite{boucekkine2011revisiting}), it is endogenous (for example, due to the cost of maintenance, which increases with the age of the machines). In the model by \cite{boucekkine2011revisiting}, where the lifetime $T$ is endogenous but constant at the optimum, the authors solve the model using an infinite-dimensional formulation. For additional references, see also \cite{Boucekkine2008, Fabbri2008, Fabbri2008a, Boucekkine2011} and the works cited therein.

In a series of other models using a similar formalization (optimal control problem with e DDE state equation), attention has been focused on human capital rather than physical capital accumulation. Examples include the works of \cite{deLaCroix1999, Boucekkine2002, Boucekkine2011, Boucekkine2014}. The last two articles provide explicit feedback solutions to the optimization problem, obtained through the formulation of the problem in Hilbert spaces.

Another strand of literature where delays naturally appear assumes that agents modify their preferences according to \emph{internal habits}, i.e., based on their past consumption habits. One of the classical models in this context is \cite{Constantinides1990}, originally formulated in discrete time and later extended to a continuous-time setting with delay differential equations by several authors, including \cite{Detemple1991} and, more recently, \cite{AugeraudVeron2015}. Models with habits have also been studied through a reformulation in Hilbert spaces via dynamic programming by \cite{Augeraud-Veron2017}.

Finally, we mention the articles by \cite{Fabbri2017} and \cite{Boucekkine2018}, where some international macroeconomic problems are modeled as optimal control problems with neutral-type state equations (i.e., including delay terms in both the state and its derivative). In both cases, the problems are studied using a Hilbert space formulation introduced by \cite{Burns1983}.

\subsection{Vintage capital with transport state equation}
\label{sec:vcteexpl}
We here analyse a model for  {vintage capital} where the heterogeneity in the vintage, or equivalently the age,  of capital goods is represented by a separate variable $s$ evolving jointly with the time variable $t$. Here, capital goods $k$ depend on the two variables $t$ and $s$, and their evolution becomes a PDE of  transport type.
\subsubsection{The economic problem}  Differently from the model analysed in \ref{sub:vintagecapitaldelay},  The instance of the model that we describe here is that of \cite{BarucciGozzi2001}, where 
the capital-accumulation process is described by the system
\begin{equation}\label{eq:esvintage}
\begin{cases}
\displaystyle{\frac{\partial z(t, s)}{\partial t}+\frac{\partial z(t, s)}{\partial s}=-\mu z(t, s)+u_1(t, s)}, & t \in(0,\infty), s \in(0, \bar{s}], \\\\ z(t, 0)=u_0(t), & t \in(0,\infty), \\\\ z(0, s)=z_0(s), & s \in[0, \bar{s}],\end{cases}    
\end{equation}
where $z(t, s)$ is the amount of capital goods of vintage $s$ accumulated at time $t$, $u_1(t,s)$ represents the gross investment rate at time $t$ in capital goods of vintage $s$, $u_0(t)$ represents gross investment in new capital goods at time $t$, it also defines the boundary condition for the evolution of the stock of capital $z$. Moreover $\mu>0$ is the depreciation factor,  $\bar{s} \in(0,+\infty]$ is the maximum vintage/age considered (older capital goods are considered nonproductive), and  $s^{\prime}>s$ means that capital goods of vintage $s^{\prime}$ are older than capital goods of vintage $s$.\footnote{Note that the partial differential equation \eqref{eq:esvintage} generalizes the classical dynamic system describing the firm capital accumulation, $z'(t)=$ $u(t)-\mu z(t)$.}
Time and vintage share the same unit measure, meaning that, net of the decrease $\mu z(t,s)$ and of the investment $u_1(t,s)$ (i.e. for $\mu=0$ and $u_1\equiv 0$), the stock of capital at time $t$ of vintage $s$ becomes after the time period $\Delta t$ the stock of capital of vintage $s+\Delta t$, i.e., $z(t, s)=z(t+\Delta t, s+\Delta t)$.

The entrepreneur maximizes over $u(\cdot)=(u_0(\cdot),u_1(\cdot))$   the discounted intertemporal profits (revenues minus costs) 
\begin{align}\label{fp}
&
    \mathcal{J}\left(z_0 ; u(\cdot)\right)\\&=\int_0^{\infty} e^{-\rho t}\bigg[\int_0^{\bar{s}}\left[\alpha(t, s) z(t, s)-q_1(t, s) u_1(t, s)-\beta_1(t, s) u_1(t, s)^2\right] \mathrm{d} s-q_0(t) u_0(t)-\beta_0(t) u_0(t)^2\bigg] \mathrm{d} t,\nonumber
    \end{align}
 where  $\alpha (t,s)k(t,s)$ represent the revenues  at time $t$ of capital goods of vintage $s$, meaning we assume  constant returns to scale,  while costs are quadratic, with the $q_i$'s representing unit costs, the $\beta_i$'s adjustment costs.
 
We can make further assumptions about the payoff coefficients to give them economic meaning:
\begin{enumerate}[(i)]
    \item  $\alpha(t, s)\ge 0$, with $\alpha(t, s)=0$,  $\forall t \geq 0$ if $s \geq \bar{s}$;  moreover,  $\frac{\partial\alpha}{\partial s} (t, s) \leq 0, \forall s\in[0,\bar s]$, for every fixed $t$ (technology productivity is decreasing in $s$, so that young capital goods are more productive than old ones);
    \item similarly, $\beta_1(t, s), q_1(t, s)>0$, $\beta_0(t), q_0(t)\ge 0$; and $\frac{\partial\beta_1}{\partial s} (t, s), \frac{\partial q_1}{\partial s} (t, s) \leq 0, \forall s\in[0,\bar s]$, for every fixed $t$ (old capital goods are less expensive and easier to
install than young ones);  moreover  $$q_0(t) \geq \lim\limits_{s \rightarrow 0} q_1(t, s)=:q_1\left(t, 0^{+}\right), \ \ \ \beta_0(t) \geq \lim\limits _{s \rightarrow 0} \beta_1(t, s)=:\beta_1\left(t, 0^{+}\right).$$    
\end{enumerate}
  
It is crucial to observe that the linearity of the payoff with respect to the state variable is essential for deriving an explicit solution to the problem. The case where the revenue is a more general concave function of the state variable produces no explicit formula, and will be later analysed in Subsection \ref{vcteimpl}.

\subsubsection{Rewriting the problem in a Hilbert space setting}
We rephrase the equation \eqref{eq:esvintage} in the Hilbert space $X=L^2(0,\bar s;\R)$ through the following linear closed operator $\A$:
\begin{equation*}
    D(\A)=\left\{f \in H^1(0,\bar s): f(0)=0\right\}, \quad \A f=-f^{\prime}-\mu f,
\end{equation*}
where $H^1(0,\bar s)$ is the usual Sobolev space of order $1$.  The operator $\A$ generates a $C_0$-semigroup on the Hilbert space $X$ (see e.g. \cite{engel-nagel}, Sec. 1.4), which is a modification of the well-known translation semigroup analysed in Example \ref{ex:trsem}. In detail
$$
[e^{\A t} f](s)=e^{-\mu t} f(s-t) \chi_{[t, \bar{s}]}(s).
$$
The adjoint semigroup $\{e^{t\A^*}\}_{t\ge0}$ and its generator $\A^*$, the adjoint of $\A$, will also play a role in the sequel.
When defined in $X$, they read as
$$[e^{t\A^*} f](s)=e^{-\mu t} f(s+t) \chi_{[0, \bar{s}-t]}(s)$$ with
$$
D\left(\A^*\right)=\left\{f \in H^1(0,\bar s): f(\bar{s})=0\right\}, \quad\A^* f=f^{\prime}(s)-\mu f,
$$
when $\bar{s}<+\infty$; when $\bar{s}=+\infty$ we have that $\lim _{s \rightarrow+\infty} f(s)=0$ for $f \in H^1(0,+\infty)$, so that
$D\left(\A^*\right)=H^1(0,\infty)$.


Formally setting $y(t):=z(t,\cdot)$, 
the state equation \eqref{eq:esvintage} can be formally rewritten as an abstract equation as
\begin{equation}\label{eq:clsolvint}
    \begin{cases}
        y'(t)=\A y(t)+\mathcal{B} u(t),&t>0,\\
        y(0)=x=z_0(\cdot),
    \end{cases}
\end{equation}
where, with abuse of notation, $u_1(t):=u_1(t,\cdot)$ and, letting $U=\mathbb R\times X$ be the control space endowed with norm $|u|=| u_0|_{\mathbb R}+|u_1|_X$, the control operator is
$$\mathcal{B}:U\to D(\A^*)', \quad \mathcal{B}u=\mathcal{B}(u_0,u_1)= u_0\delta_0+u_1.$$
Note that, it is 
hence unbounded -- a common trait in boundary value problems, where the control is acting on the boundary.
We are therefore within the framework of unbounded control described in Subsection \ref{SSSE:Bunbounded}. In particular $\mathcal{B}$ satisfies \ref{ass:Bunbounded}, so the claim of Proposition \ref{prop:mildunbounded} holds in this case and therefore the 
 mild solution
\begin{equation}\label{mildvintage}
y(t)=e^{t\A  }x+\int_0^t e^{(t-\tau)\A} \mathcal{B}u(\tau) \mathrm{d} \tau
\end{equation}
lies in $X$.
The profit functional can also be rewritten in abstract terms 
with abuse of notation as 
\begin{equation}\label{eq:vintagefp2}
\begin{aligned}
&    \mathcal{J}\left(x ; u(\cdot)\right)\\&=\int_0^{\infty}  e^{-\rho t}\bigg[\langle\alpha(t), k(t)\rangle-\langle q_1(t), u_1(t)\rangle-\langle \beta_1(t) u_1(t), u_1(t)\rangle -q_0(t) u_0(t)-\beta_0(t) u_0^2(t)\bigg] \mathrm{d} t.
    \end{aligned}
\end{equation}
The functional $\mathcal{J}\left(k_0, u\right)$ has to be minimized with respect to $u$ ranging over a suitable set of admissible controls. For instance, we may choose the set
\begin{align}\label{U(t,x)rho}
 L_\rho^2(0,\infty; U)\equiv\left \{u:[0,\infty)\to U: \int_0^\infty e^{-\rho \tau}|u(\tau)|^2<\infty\right \}.
\end{align}
We do not impose constraints on the controls and on the trajectory, but provide conditions under which the positivity of the trajectory is obtained \emph{a posteriori} for the solution of the unconstrained problem.
For the sake of simplicity, in what follows, we will assume that the coefficients $\alpha, q_1, \beta_1, q_0, \beta_0$ are independent of $t$, leading to the stationary problem, although in \cite{BarucciGozzi2001} the authors discuss a case with a particular dependence on the time variable $t$ of such coefficients.

\subsubsection{The infinite-dimensional HJB equation and its explicit solution} We discuss in detail the case in which the data satisfy the following assumptions.
\begin{Assumption} \label{hypovint1} 
\begin{enumerate}[(i)]
\item[]
\item$\alpha, \beta_1, q_1\in H^1(0,\bar s)$. 
\item For every $s \geq 0$, we have $\alpha(s) \geq 0, q(s) \geq 0, \beta(s) \geq \epsilon$ for a given $\epsilon>0$
\item $\alpha^{\prime}(s) \leq 0, q_1^{\prime}(s) \leq 0$, $\beta_1^{\prime}(s) \leq 0$.
\item $\alpha(\bar{s})=0$, and
$$
q_0 \geq \lim _{s \rightarrow 0^+} q_1(s)=:q\left(0^{+}\right), \quad \beta_0 \geq \lim _{s \rightarrow 0^+} \beta_1(s)=:\beta_1\left(0^{+}\right).
$$
\end{enumerate}
\end{Assumption}
\begin{Assumption}  \label{hypovintmu} When $\bar s=+\infty$, we assume $\rho>-\mu$.\end{Assumption}
\noindent 
In line with the notation of Section \ref{sec:DP}, for $p\in D(\A^*)$, $x\in X$, we 
define\footnote{Operations with functions must be intended pointwise.}
\begin{align}\mathcal H(x,p)&:=\langle\alpha, x\rangle+\sup_{u_0\in \mathbb R} \Big\{(\delta_0p-q_0)  u_0-\beta_0 u_0^2\Big\}  
 +\sup_{u_1\in X} \Big\{\langle (p-q), u_1\rangle
-\left\langle \beta_1 u_1, u_1\right\rangle\Big\}  \label{eq:hamvitage} \\
& = \langle\alpha, x\rangle+\frac{(p(0)-q_0)^2}{4\beta_0}+\Big\langle {\frac 1{4\beta_1}}(p-q),p-q\Big\rangle\nonumber
\end{align} 
Note that the above function is well defined as $p$ is continuous and $(\delta_0p)=p(0)$. Moreover, the suprema in the above formula are attained at
\begin{equation}\label{eq:maxconvintage}
  u_0^*=\frac{p(0)-q_0}{2\beta_0}, \quad u_1^*=u_1^*(s)=\frac{p(s)-q(s)}{2\beta_1(s)}.  
\end{equation}
Consequently, the
 HJB equation associated to the problem is 
 \begin{equation} \label{eq:HJBvintage}
  \rho v(k)= \langle\alpha+\A^*Dv(k), x\rangle+\frac{([Dv(x)](0)-q_0)^2}{4\beta_0}+\Big\langle \frac 1{4\beta_1}Dv(x)-q_1,Dv(x)-q_1\Big\rangle.   
 \end{equation}

\begin{Proposition}
\label{prop:vintageHJBsol}
Let Assumptions \ref{hypovint1} and \ref{hypovintmu} hold. Define  the function
\begin{equation}
\label{eq:vintageVF}
v(x) = \langle\bar{\alpha}, x\rangle+\frac{(\bar\alpha(0)-q_0)^2}{4\rho\beta_0}+\int_0^{\bar s}\frac{(\bar\alpha(s)-q_1(s))^2}{4\rho\beta(s)}\ud s
\end{equation}
where 
$$
\bar{\alpha}(s)=[ R(\rho, \A^*) \alpha](s)=\int_s^{\bar{s}} e^{-(\rho+\mu)(\sigma-s)} \alpha(\sigma) d \sigma.
$$
Then $Dv(x)\in D(\A^*)$ and  $v$ 
is a classical  solution to the \emph{HJB} equation \eqref{eq:HJBvintage} on $X$.
\end{Proposition}

\begin{proof}
    By the Riesz representation theorem, it is straightforward to see that $Dv(x)=\bar \alpha$, which is  a function in $D(\A^*)$ by direct proof. Then, comparing \eqref{eq:vintageVF} and \eqref{eq:HJBvintage}, one sees that  $v$ solves \eqref{eq:HJBvintage} if and only if 
    $$\rho\langle \bar\alpha, x\rangle=\langle\alpha+\A^*\bar\alpha, x\rangle\iff \langle (\rho-\A^*)\bar\alpha -\alpha, x\rangle=0.$$
    Note that the latter is equivalent to establishing that $\bar \alpha=R(\rho,\A^*)\alpha$ or also that $\bar\alpha$  is
 the unique solution to the backward Cauchy problem
    $$\begin{cases}
        \rho f(s)-f'(s)+\mu f(s)=\alpha(s), &s\in[0,\bar s)\\
        f(\bar s)=0
    \end{cases}.$$
    which is, by the variation of constant formula
    $\bar \alpha$, 
    concluding the proof.
\end{proof}

\subsubsection{The solution of the model}  We want to show that $V$ coincides with the explicit solution $v$ of the HJB equation  of Proposition \ref{prop:vintageHJBsol}, so that the candidate optimal (feedback) controls are obtained by replacing $p$ in  \eqref{eq:maxconvintage} with the gradient $\bar\alpha$ of $v$. 
Note that such controls are actually independent of time, of  the trajectory $y(t)$, and of the initial datum,  so that the closed loop equation, obtained by replacing  $u_0$ and $u_1$ with $u_0^*$ and $u_1^*$ in \eqref{eq:esvintage}, has a unique solution in $X$.

 Our theoretical results of Subsection \ref{sub:infstat} apply leading to the following.

\begin{Theorem}
Let Assumptions \ref{hypovint1} and \ref{hypovintmu} hold. The following statements hold:
\begin{enumerate}[(i)]
    \item The value function $V$ coincides with $v$  defined  in Proposition \ref{prop:vintageHJBsol}.
    \item The unique optimal control is 
\begin{equation}\label{eq:optcontrvintage}
    u_0^*=\frac{\bar\alpha(0)-q_0}{2\beta_0}, \quad u_1^*(s)=\frac{\bar\alpha(s)-q(s)}{2\beta(s)},
\end{equation} 
and the associated  optimal trajectory of the original problem is\footnote{Note that when $z_0$ is nonnegative, a sufficient condition for  the optimal trajectory to be nonnegative is that the optimal controls are nonnegative, that is
    $$q_0\le \bar\alpha(0),\quad q(s)\le \bar\alpha(s), \ \ \textrm{for a.a. } s\in [0,\bar s].$$}\begin{equation}\label{eq:optimaltrjvintage}
    z^*(t,s)=\begin{cases}
        \displaystyle{e^{-\mu t} z_0(s-t)+\int_0^te^{-\mu \sigma}\frac{\bar\alpha(s-\sigma)-q_1(s-\sigma)}{2\beta_1(s-\sigma)}d\sigma},& s\ge t,\\\\
\displaystyle
{       e^{-\mu s}\frac{\bar\alpha(0)-q_0}{2\beta_0}\int_0^se^{-\mu \sigma}\frac{\bar\alpha(s-\sigma)-q_1(s-\sigma)}{2\beta_1(s-\sigma)}d\sigma},&s<t.
    \end{cases}
\end{equation}
\end{enumerate}
\end{Theorem}
\subsubsection{Variants and Literature} 
A variant of the model with a boundary condition of Neumann type is addressed in \cite{Barucci1998}. Extensions of the model for a objective functional  concave  in the state and control variables are addressed in a series of papers by Faggian with finite horizon, and by Faggian and Gozzi with infinite horizon. For finite horizon, in  \cite{faggian2008application} it is shown the value function is the unique strong solution  (namely, the limit of classical solutions of approximating equations) of the HJB equation associated to the problem, althouygh no explicit formula is available. The optimal control is characterized by means of the gradient of the value function. Concurrently,  \cite{faggian2005regular} establishes existence and uniqueness of {strong} solutions of a class of HJ equation to which the value function of the control problem belongs. The same optimal control problem, but with the addition of state constraints, is then addressed  in \cite{faggian2008hamilton},  together with existence and uniqueness of the a weak solution (defined as, the limit of strong solutions in the above sense) of the associated HJB equation. 
For the infinite horizon case one can see, e.g., \cite{faggian2010optimal} and \cite{faggian2021optimal}.

Models with the same evolutionary and payoff structure are used in various applications.  In \cite{Faggian2013} the authors analyze a model for optimal advertising where    capitals are interpreted as the so called ``goodwill" (the intangible value of a product  that comes from its established reputation) of products of different vintages, and controls represent the intensity of advertising. In \cite{fabbri2015mitra} capitals represent a forest of trees   distributed in time among different ages, while the pay-off represents utility from harvesting, and is a concave functional solely of the control.


\subsection{Time-to-build: the linear pointwise delay case}
\label{sub:time-to-build}

\subsubsection{The economic problem}

In this subsection  work we present a class of models with a \emph{time-to-build} structure, i.e. a model which explicitly takes into account the time that elapses between the investment decision and the moment when the new generation of capital goods becomes productive. 

Time-to-build models are not directly a sub-family of vintage capital models but it is clear that the two have an important feature in common (see also \citealp{Benhabib1991}): in both cases, past investments have different effects on capital and output depending on when they were made. For this reason, an irregular pattern in the history of past investment implies shocks and an irregularity in the present capital stock both in vintage capital models and in time-to-build models.

In this subsection, we examine the model introduced by \cite{Kalecki35}, later studied, among others, by \cite{Asea1999}, \cite{Bambi08} and \cite{BambiFabbriGozzi10}, which presents the simplest case (albeit technically somewhat delicate) where investment becomes suddenly and fully productive after a certain (exogenous and constant) time $d$ (the \emph{gestation period}). For the reformulation of the problem as an optimal control problem in Hilbert spaces, we will follow \cite{BambiFabbriGozzi10}.

The described time-to-build structure is integrated, following \cite{Bambi08}, into an AK endogenous growth model, assuming that production at time $t$ is given by:
\[
q(t) = A k(t-d),
\]
where $A$ is the capital productivity (we use the notation $q$ for the production instead of $y$ to avoid confusion with the name of the state variable in the Hilbert space formulation of the problem). So, under the standard closed-economy, no-state-intervention budget constraints, we have
\[
i(t) = q(t) - c(t) = A k(t-d) - c(t),
\]
which we assume to be positive to incorporate an irreversibility constraint: once capital is installed, it has no alternative value outside production. Including a capital depreciation rate $\delta \geq 0$, the accumulation of capital is governed by:
\begin{equation}
{k}'(t) = i(t) - \delta k(t-d) = A k(t-d) - c(t) - \delta k(t-d) = \tilde{A}k(t-d) - c(t), \qquad \forall t \geq 0, \label{eq:state-delay} 
\end{equation}
where the parameter $\tilde{A} = A - \delta > 0$ reflects the productivity level $A$ net of the capital depreciation rate $\delta$. The irreversibility constraint can therefore be written as:
\begin{equation}
\dot{k}(t) \geq -\delta k(t-d), \qquad \forall t \geq 0, \label{eq:irrnew}
\end{equation}
and is coupled with the non-negativity constraints on $c$ and $k$.

The planner maximizes:
\begin{equation}
\label{eq:functional-firstversion}
\int_{0}^{\infty }\frac{c(t)^{1-\sigma }}{1-\sigma }e^{-\rho t}\ud t,
\end{equation}
where, as usual, $\rho > 0$ represents the intertemporal discount rate, and $\sigma > 0$ (with $\sigma \neq 1$) is the inverse of the elasticity of substitution.

We assume that the initial condition for $k$ is slightly more regular than in Subsection \ref{sub:vintagecapitaldelay}. Specifically, we assume that $k_0(\cdot) \in H^1([-d,0];\mathbb{R}^+)$ (the Sobolev space of order 1). We also assume that $c(\cdot) \in L^2_{loc}([0,+\infty);\mathbb{R})$. Under these assumptions, equation (\ref{eq:state-delay}) admits a unique continuous solution, and this solution belongs to $H^1_{loc}([-d,+\infty);\mathbb{R})$, as proved in \cite{BDDM07}, page 287.

Before reformulating the problem in the Hilbert space ${X}$ introduced in Subsection \ref{sub:vintagecapitaldelay}, we perform a change of variables. The optimal control problem is rewritten in terms of output, $q(t) = A k(t-d)$, and adjusted net investment, $u(t) := (A/\tilde{A})\dot{k}(t) = q(t) - \frac{A}{\tilde{A}} c(t)$.

In this way the constraints
\[
i(t) \geq 0, \qquad c(t) \geq 0
\]
read as
\begin{equation}
\label{eq:constraintssecond-1}
u(t) \geq \left(1-\frac{A}{\tilde{A}}\right)q(t) \text{ and } u(t) \leq q(t) \iff u(t) \in
\left[\left(1-\frac{A}{\tilde{A}}\right)q(t),q(t)\right],
\end{equation}
and maximizing the functional (\ref{eq:functional-firstversion}) is equivalent to maximizing:
\begin{equation}
\label{eq:functional-secondversion}
\int_0^\infty e^{-\rho t}
\frac{\left (q(t) - u(t) \right )^{1-\sigma}}{1-\sigma} \ud t
\end{equation}
subject to the state equation:
\begin{equation}
\label{eq:secondstateeq}
\left \{
\begin{array}{l}
{q}'(t) = \tilde{A} u(t-d), \qquad t \ge 0, \\
u(s) = u_0(s) = A \dot{k}_0(-d-s), \qquad s \in [-d,0), \\
q(0) = q_0 := A k(-d),
\end{array}
\right .
\end{equation}
and the constraint (\ref{eq:constraintssecond-1}).

\subsubsection{Rewriting the problem in a Hilbert space setting}
Given any initial data $q_0 \in \mathbb{R}$, $u_0 \in L^2([-d,0);\mathbb{R})$, and any control strategy $u(\cdot) \in L^2_{loc}([0, + \infty);\mathbb{R})$, we denote by $q^{(q_0,u_0(\cdot)),u(\cdot)}(\cdot)$ the unique output trajectory associated with $u(\cdot)$. This trajectory is the unique absolutely continuous solution of (\ref{eq:secondstateeq}), as shown in \cite{BDDM07}, Theorem 4.1, page 222. 

Following the approach used in Subsection \ref{sub:vintagecapitaldelay}, and based on the results of \cite{VinterKwong}, we define the structural state of the system. This state is again an element of the Hilbert space ${X}$ introduced in Subsection \ref{sub:vintagecapitaldelay}.

\begin{Definition}
\label{structuralstate}
Given the initial data $q_0 \in \mathbb{R}$ and $u_0 \in L^2([-d,0];\mathbb{R})$, and a control strategy $u(\cdot) \in L^2_{loc}([0,+ \infty);\mathbb{R})$, we define the \emph{structural state} of the system at time $t \geq 0$ as the pair:
\[
{y}^{(q_0,u_0(\cdot)),u(\cdot)}(t) =
({{y}^{(q_0,u_0(\cdot)),u(\cdot)}}^0(t),
{{y}^{(q_0,u_0(\cdot)),u(\cdot)}}^1(t))
\]
\[
\qquad \qquad \qquad := (q^{(q_0,u_0(\cdot)),u(\cdot)}(t), \gamma(t)(\cdot)) \in {X},
\]
where $\gamma(t)(\cdot)$ is the element of $L^2([-d,0];\mathbb{R})$ defined as:
\begin{equation}
\label{eq:defdellostato}
\left \{
\begin{array}{l}
\gamma(t)(\cdot) \colon [-d,0] \to \mathbb{R}, \\[4pt]
\gamma(t)(s) := \tilde{A} u(t-d-s).
\end{array}
\right .
\end{equation}
\end{Definition}

For notational convenience we will often skip the double superscript and we will just write ${y}(t) =
({{y}}^0(t),
{{y}}^1(t))$.

Before writing the state equation as an equation in ${X}$, we introduce some additional notation. We start by defining the unbounded operator $\A^*$ on ${X}$:
\[
\left\{ \begin{array}{ll}
D(\mathcal{A}^*) := \{ (\psi^0,\psi^1) \in {X} : \psi^1 \in
  W^{1,2}([-d,0];\mathbb{R}), \psi^0 = \psi^1(0) \}, \\
\mathcal{A}^* \colon D(\mathcal{A}^*) \rightarrow {X}, \\
\mathcal{A}^*(\psi^0,\psi^1) := (0,\frac{d}{d s}\psi^1).
\end{array} \right .
\]
The operator $\mathcal{A}^*$ and its adjoint $\A$ are generators of a $C_0$ semigroup on ${X}$ \citep[see][Section 4.6, page 242]{BDDM07}. We define Dirac's delta $\delta_{-d}$ on $D(\mathcal{A}^*)$ as:\footnote{By the Sobolev embedding theorem, $W^{1,2}([-d,0];\mathbb{R}) \subseteq C([-d,0]; \mathbb{R})$, so this definition is well-posed.}
\begin{equation}\label{eq:defdeltanew}
\delta_{-d}(\psi^0,\psi^1) = \delta_{-d}\psi^1 = \psi^1(-d) \in \mathbb{R}.
\end{equation}

We are now ready to rewrite the state equation of our problem as an evolution equation in ${X}$. Using Theorem 5.1, page 258 of \cite{BDDM07}, we obtain the following result:

\begin{Theorem}
Given any initial data $q_0 \in \mathbb{R}$, $u_0 \in L^2([-d,0];\mathbb{R})$, and any control strategy $u(\cdot) \in L^2_{loc}([0, + \infty);\mathbb{R})$, the structural state ${y}^{(q_0, u_0(\cdot)),u(\cdot)}(\cdot)$, as defined in (\ref{eq:defdellostato}), is the unique solution in the space:
\[
\Pi := \bigg \{ f \in C(0,+\infty; {X}) : f' \in L^2_{loc}(0,+\infty;\,D(\A^*)')) \bigg \}.
\]
of the equation:
\begin{equation}
\label{eqBDDMbis}
\left\{ \begin{array}{ll}
 {y}'(t) = \mathcal{A} {y}(t) + {u}(t) \tilde{A} \delta_{-d},
\hspace{.05in} \hspace{.05in} t\geq 0, \\\\
{y}(0) = x = (q_0, \gamma(0)(\cdot)),
\end{array} \right.
\end{equation}
where $\gamma(0)(\cdot)$ is defined in terms of $u_0$ as in  (\ref{eq:defdellostato}).
\end{Theorem}

Note (see page 258 of \citealp{BDDM07}) that (\ref{eqBDDM}) has a unique solution for every initial datum $x \in {X}$ and control strategy $u(\cdot) \in L^2_{loc}([0, + \infty);\mathbb{R})$. We denote this solution by $y^{p,u(\cdot)}(\cdot)$. 

Once we have rewritten the state equation using the Hilbert space formalization, we can consequently reformulate the other elements of the optimization problem in a consistent manner.

The constraint (\ref{eq:constraintssecond-1}) can be rewritten as:
\[
u(t) \in \left [ \left ( 1-\frac{A}{\tilde{A}} \right) y^0(t) , y^0(t) \right ], \qquad t\geq 0,
\]
so the set of admissible control strategies for a given initial datum $p \in {X}$ is given by:
\begin{equation}
\label{eq:defA0}
\mathcal{U}(x) := \Big \{ u \in L^2_{loc}([0,+\infty);\mathbb{R}^+) : u(t) \in \left [ \left ( 1-\frac{A}{\tilde{A}} \right) y^0(t), y^0(t) \right ] \text{ for all } t\geq 0 \Big \}.
\end{equation}
If $y^0(t)<0$, then the interval $\left [ \left ( 1-\frac{A}{\tilde{A}} \right) y^0(t), y^0_{p,u(\cdot)}(t) \right ]$ is empty. Thus, admissibility requires $y^0(t) \geq 0$ for all $t\geq 0$.

The functional to be maximized becomes:
\begin{equation}
\label{eqfunzionalehilbert}
J (x,u(\cdot)) := \int_0^\infty e^{-\rho s} \frac{(x^0(t) - u(t))^{1-\sigma}}{1-\sigma} \ud s.
\end{equation}

\subsubsection{The infinite-dimensional HJB equation and its explicit solution}
\label{subsubsec:HJB-timetobuild}

Following the Hilbert space formulation, we can now write the associated Hamilton-Jacobi-Bellman (HJB) equation which, in line with (\ref{eq:HJBbis}) can be written, in terms of the structural state as

\begin{equation}
\label{eq:HJB-timetobuild}
\rho v(x) = \langle x,  \mathcal{A} Dv(x) \rangle_{{X}} +  \sup_{u \in \left [ \left ( 1-\frac{A}{\tilde{A}} \right) x^0, x^0 \right ]} \left [ u \tilde{A} \delta_{-d}(D(v(x)) + \frac{(x^0 - u)^{1-\sigma}}{1-\sigma} \right ]
\end{equation}

Similarly to what we did in Subsection \ref{sub:vintagecapitaldelay} for the $AK$ vintage model, we denote here $\xi$ the unique positive real root of the characteristic equation associated with the delay dynamics:
\begin{equation}
\label{eq:characteristic-eq-delay}
z = \tilde{A} e^{-zd}.
\end{equation}
Also in this case $\xi$ can be interpreted as the maximal attainable growth rate (or the constant real interest rate of the economy). So the standard assumption for endogenous growth and for the finiteness of the value function translates here in the following:
\begin{Assumption}
\label{hyp:ttb_finiteness}
$\rho > \xi(1-\sigma)$.
\end{Assumption}
It ensures that the planner's discount rate $\rho$ is sufficiently high compared to the economy's maximal growth rate $\xi$ to prevent the utility integral \eqref{eqfunzionalehilbert} from diverging.

To find an explicit solution of the HJB equation, we try to find a function homogeneous in a linear functional of the state. Since $\xi$ can be interpreted as the interest rate 
it is natural to weight the past actions using the exponential of $\xi$. All in all (and similarly to what we had in the vintage capital model that we presented in the previous subsection) we look for a solution that is in the form of of power of some ``equivalent capital''. So we define the functional $\Gamma: {X} \to \mathbb{R}$ as
\begin{equation}
\label{eq:defGamma_ttb}
\Gamma(x) := x^0 + \int_{-d}^0 e^{\xi s} x^1(s) \,{\mathrm{d}}s,
\end{equation}
where $x=(x^0,x^1) \in {X}$. Here, $x^0(t) = q(t)$ is the current output produced by the current installed capital, while the integral term represents the discounted present value of the output that will be generated by capital goods that are already in the pipeline but not yet productive. 

The solution to the HJB equation can be defined on a specific subset of ${X}$. Let us define the open set $\mathcal{Q} := \{ x \in {X} : \Gamma(x) > 0 \}$ and, setting
\[
\alpha = \frac{\rho - \xi(1-\sigma)}{\sigma \xi},
\]
the open set
\[
\mathcal{O} := \left\{ x = (x^0,x^1) \in \mathcal{H} :\ \ \Gamma(x) < x^0 \left(\frac{A}{\alpha \tilde{A}}\right) \right\}.
\]
The set $\mathcal{O}$ represents the region of the state space where the optimal control is an interior solution to the constrained problem, as we will see in the proof.
\begin{Proposition}
\label{prop:HJB-solution-ttb}
Under Assumption \ref{hyp:ttb_finiteness}, the function $v: \mathcal{Q} \to \mathbb{R}$ given by
\begin{equation}
\label{eq:value-function-ttb}
v(x) := \nu \frac{[\Gamma(x)]^{1-\sigma}}{1-\sigma}
\end{equation}
with
\[
\nu = \alpha^{-\sigma} \frac{1}{\xi}
\]
is differentiable in all $x \in \mathcal{Q}$ and is a solution of the HJB equation \eqref{eq:HJB-timetobuild} on the set $\mathcal{O}$ in the sense of (\ref{eq:HJBbis}).
\end{Proposition}
\begin{proof}
The proof is constructive: we verify that \eqref{eq:value-function-ttb} solves the HJB equation \eqref{eq:HJB-timetobuild}.

First, we compute the Fréchet derivative of $v(x)$. Since $\Gamma$ is a linear functional, $v$ is differentiable for all $x \in \mathcal{Q}$, and its derivative is:
\[
Dv(x) = \nu [\Gamma(x)]^{-\sigma} D\Gamma(x),
\]
where $D\Gamma(x) = \psi := (1, \theta(\cdot))$ with $\theta(s) = e^{\xi s}$ which can be easily seen to belong to $D(\A^*)$.

Next, we find the candidate feedback optimal control by maximizing 
\[
u \mapsto u \tilde{A} \delta_{-d}(Dv(x)) + \frac{(x^0 - u)^{1-\sigma}}{1-\sigma}
\]
over the set $u \in \left [ \left ( 1-\frac{A}{\tilde{A}} \right) x^0, x^0 \right ]$. Since this function is concave the first-order condition is necessary and sufficient it identify a point which respects the constraints. We get
\[
\tilde{A} \delta_{-d}(Dv(x)) - (x^0 - u^*)^{-\sigma} = 0 \implies u^*(x) = x^0 - (\tilde{A} \delta_{-d}(Dv(x)))^{-1/\sigma}.
\]
We compute the term $\delta_{-d}(Dv(x))$:
\[
\delta_{-d}(Dv(x)) = \nu [\Gamma(x)]^{-\sigma} \delta_{-d}(\psi) = \nu [\Gamma(x)]^{-\sigma} e^{-\xi d}.
\]
Using the characteristic equation \eqref{eq:characteristic-eq-delay}, $\tilde{A}e^{-\xi d} = \xi$, so,
\[
\tilde{A} \delta_{-d}(Dv(x)) = \xi \nu [\Gamma(x)]^{-\sigma}.
\]
Substituting this into the expression for $u^*(x)$ gives the candidate feedback policy:
\[
u^*(x) = x^0 - (\xi \nu [\Gamma(x)]^{-\sigma})^{-1/\sigma} = x^0 - (\xi \nu)^{-1/\sigma} \Gamma(x).
\]
Using the definitions of $\nu$ and $\alpha$, we find that $(\xi\nu)^{-1/\sigma} = \alpha$, so
\[
u^*(x) = x^0 - \alpha \Gamma(x).
\]
This maximizer is admissible (i.e., it satisfies the constraints in \eqref{eq:constraintssecond-1} so the first condition identify an internal maximum in $\left [ \left ( 1-\frac{A}{\tilde{A}} \right) x^0, x^0 \right ]$) precisely when $x \in \mathcal{O}$.

Finally, we substitute the solution $v(x)$ and the candidate optimal control $u^*(x)$ back into  \eqref{eq:HJB-timetobuild}. The equation holds if:
\[
\rho v(x) = \langle x, \mathcal{A}^* Dv(x) \rangle_{{X}} + u^*(x) \tilde{A} \delta_{-d}(Dv(x)) + \frac{(x^0 - u^*(x))^{1-\sigma}}{1-\sigma}.
\]

Let's evaluate the right side of this equation: since $Dv(x) = \nu [\Gamma(x)]^{-\sigma} \psi$, $\mathcal{A}^*\psi = (0, \xi e^{\xi s})$ and $u^*(x) = x^0 - \alpha \Gamma(x)$, we have
\begin{multline}
\langle \mathcal{A}^* Dv(x), x \rangle_{{X}} + u^*(x) \tilde{A} \delta_{-d}(Dv(x)) + \frac{(x^0 - u^*(x))^{1-\sigma}}{1-\sigma} \\
= \left\langle \nu [\Gamma(x)]^{-\sigma} (0, \xi e^{\xi s}), (x^0, x^1) \right\rangle_{{X}} +
(x^0 - \alpha \Gamma(x))(\xi \nu [\Gamma(x)]^{-\sigma}) + \frac{(\alpha \Gamma(x))^{1-\sigma}}{1-\sigma}\\
= \nu [\Gamma(x)]^{-\sigma} \left( \xi \int_{-d}^0 e^{\xi s} x^1(s) \,{\mathrm{d}}s \right) +
x^0 \xi \nu [\Gamma(x)]^{-\sigma} - \alpha \xi \nu [\Gamma(x)]^{1-\sigma} + \frac{\alpha^{1-\sigma}}{1-\sigma} [\Gamma(x)]^{1-\sigma}.
\end{multline}
Using the definition of $\Gamma(x)$ from \eqref{eq:defGamma_ttb}, we can write $\int_{-d}^0 e^{\xi s} x^1(s) \,{\mathrm{d}}s = \Gamma(x) - x^0$. Thus the previous expression becomes
\begin{multline}
\nu \xi [\Gamma(x)]^{-\sigma} \left( \Gamma(x) - x^0 \right) + x^0 \xi \nu [\Gamma(x)]^{-\sigma} - \alpha \xi \nu [\Gamma(x)]^{1-\sigma} + \frac{\alpha^{1-\sigma}}{1-\sigma} [\Gamma(x)]^{1-\sigma}\\
=[\Gamma(x)]^{1-\sigma} \left( \xi\nu(1-\alpha) + \frac{\alpha^{1-\sigma}}{1-\sigma} \right).
\end{multline}
Using the definitions of $\alpha = \frac{\rho-\xi(1-\sigma)}{\sigma\xi}$ and with some algebra one can prove that $\left( \xi\nu(1-\alpha) + \frac{\alpha^{1-\sigma}}{1-\sigma} \right)$ is indeed equal to $\rho \left( \frac{\nu}{1-\sigma} \right) = \rho v(x)$ so the HJB equation is verified.
\end{proof}

\subsubsection{The solution of the model}
\label{subsubsec:solution-timetobuild}
The explicit solution to the HJB equation allows us to identify the candidate for the optimal policy in feedback form. The candidate optimal feedback map, provided by the maximizer of the Hamiltonian, is the linear map
\begin{equation}
\label{eq:Feedback-delay}
 \phi(x) := x^0 - \alpha \Gamma(x).
\end{equation}
We have the following result.
\begin{Proposition}\label{prop:ggg}
Let Assumption \ref{hyp:ttb_finiteness} hold. Let $\phi: \mathcal{O} \to \mathbb{R}$ be the (linear) feedback map defined by \eqref{eq:Feedback-delay} and  let $x \in \mathcal{O}$ be an initial state with $x^0>0$ and $x^1 > 0$. If the unique solution $y(\cdot)$ to the closed-loop system
\begin{equation}
\label{eqBDDM-ffedbackopt}
\left\{ \begin{array}{ll}
 {y}'(t) = \mathcal{A} {y}(t) + {\phi(y(t))} \tilde{A} \delta_{-d},
\hspace{.05in} \hspace{.05in} t\geq 0, \\\\
{y}(0) = x = (q_0, \gamma(0)(\cdot)),
\end{array} \right.
\end{equation}
remains in $\mathcal{O}$ for all $t \ge 0$, then the feedback control $u^*(t) = \phi(y(t))$ is optimal and the value function is given by the function $v$  provided in  Proposition \ref{prop:HJB-solution-ttb}.
\end{Proposition}
\begin{proof}[Sketch of the proof]
The result follows from applying Theorem \ref{teo:verst}. The demanded transversality can be proved by using Assumption \ref{hyp:ttb_finiteness}.
\end{proof}

If one comes back to the original DDE formulation of the problem, one can rewrite the results in terms of the original variables. For instance, the optimal control in open loop form is the solution of the following DDE:
\begin{equation*}
 \label{eq:per-u-ottimo}
 \left\{ \begin{array}{l}
\displaystyle{\frac{\ud}{\ud t}{ u}^*(t) = \tilde A { u}^*(t-d) \left ( 1- \alpha \right )  -\alpha \left ( \xi \tilde A e^{\xi t}\int_{-d-t}^{-t}
e^{\xi s} { u^*}(-d-s) \ud s + \tilde A ( -{ u^*}(t-d)+
e^{-d \xi} { u^*}(t) ) \right ),}\\\\
 {u}^*(s)= u_0(s) \qquad \text{for } s\in [-d,0),\\\\
 u^* (0) = \left ( 1- \alpha \right ) q_0 - \alpha
  \int_{-d}^0 e^{\xi s} u_0(-d-s)(s) \ud s.
 \end{array} \right.
 \end{equation*}
Similar results can be stated for the trajectory  $t\mapsto q(t)$ .

\begin{Remark}\label{rem:adm}
    One of the main issue for the application of Proposition \ref{prop:ggg} is the requirement that the solution to the closed loop equation remains in $\mathcal{O}$, that is, the admissibility of the candidate optimal feedback map. This may be really difficult: some sufficient conditions are provided at a theoretical level in the paper \cite{bambi2017generically}, dealing with a similar problem, but with distributed delay.\hfill$\square$
\end{Remark}

\subsubsection{Variants and Literature}
The economic importance of time-to-build was already clear, for instance, to \cite{Jevons71} (see, in particular, the section "Capital is concerned with Time" in Chapter 7 of \emph{Theory of Capital}). However, the first study linking this concept to aggregate economic fluctuations is attributed to \cite{Kalecki35}. Since then, numerous economists have explored this theme. Among them, the fundamental work of \cite{kydland1982time} demonstrated how, in a dynamic model with time-to-build, exogenous technological fluctuations can be amplified and become a key driver of economic cycles. The time-to-build phenomenon is well-documented in the empirical literature. For example, the dataset of \cite{Koeva00}, derived from a sample of firms in the Compustat database, indicates that the average time required to install a facility is about two years, with no significant dependence on the business cycle.

In this subsection, we focused on the formalization introduced by \cite{Kalecki35} and studied, for instance, by \cite{Asea1999, Bambi08, BambiFabbriGozzi10}, where investment becomes productive after a fixed, exogenous lag modeled as a pointwise delay in the state equation. In other time-to-build models addressing different aspects (such as firm and household investments or the presence of habits), other authors, including \cite{zhou2000time, gomme2001home, edge2007time, jeon2023time}, often use a smoother dependence on the entire history of past investments, frequently in a discrete-time framework.  On the one hand, this approach can be seen as more general; on the other, it is technically simpler. This is because, instead of involving an unbounded operator in the state equation (as is the case here, where the operator $\delta_{-r}$ appears in the evolution equation in ${X}$, with the relevant norm being that of $L^2(-d,0)$), it features a bounded operator (for instance the $L^2$ inner product with a weight function for the ``history'' of investment). In the next subsection, we explore another possible approach to obtaining a different infinite-dimensional formalization of the problem in the distributed delay case.

\section{Examples without explicit solution arising in economic theory}\label{sec:exnonexpl}
\subsection{Time-to-build: the nonlinear distributed delay case}\label{sub:ttbnonlinear}
In this subsection, we continue the study of time-to-build models, but instead of focusing, as in the previous case, on the scenario where machines become fully productive after a fixed time $d$ from purchase, we will consider a case where the process occurs more gradually. The problem described here is inspired by the economic problem proposed in \cite{Asea1999}, which is a nonlinear version of the one addressed in Subsection \ref{sub:time-to-build}, by replacing the term $Ak(t-d)$ by $P(k(t-d))$, with $P$ increasing and concave, and the power utility function with a more general one.  What we illustrate here is a smoothed version of this problem, extensively studied by means of viscosity solutions in the two papers \cite{federico2010hjb, federico2011hjb},  simplifying it in some aspects\footnote{We do not consider here the utility on the current state, present in the aforementioned papers.}.


\subsubsection{The optimal control problem with delay}
From a mathematical point of view, the aim is to study the optimal control of the $1$-dimensional  ordinary differential equation with delay
\begin{equation}
\label{statedelay}
\begin{cases}
\displaystyle{k'(t)=P\Big(k(t), \int_{-d}^0 a(\xi)k(t+\xi) d \xi\Big)-u(t),  }   \\  k(0)=x_0;\quad k(s)=x_1(s),\,\,s\,\in\,[-d,0).
  \end{cases}
\end{equation}
where \begin{enumerate}[(i)]
\item 
$d\in (0,\infty)$ is the maximal time delay of the  controlled ODE;
\item  $x= (x_0,x_1(\cdot))$ is such that $x_0>0$ and $x_1\in L^2([-d,0];\R)$;
\item  
$P: \R^+\times \R\to \R$;
\item  $u(\cdot)$ is the control variable taking values in some interval $\mathcal{C}\subseteq \R^+$. 
\end{enumerate}
The goal is to maximize a functional of the form\footnote{Notice that $\mathcal{J}$ does not actually depend on $\alpha$ here.}
\begin{equation}\label{functttb}
\mathcal{J}(c(\cdot))=\int_0^\infty e^{-\rho t} h(u(t))\ud t,
\end{equation} over the set of admissible controls
$$
\mathcal{U}(x)=\Big\{u(\cdot)\in L^1_{loc}(\R^+;\mathcal{C}): \ k^{x,u(\cdot)}(\cdot)>0\Big\}.
$$
The meaning of the above objects is the following.
\begin{enumerate}[(i)]
\item  $k$ is the state variable representing capital,  required to stay positive (state constraint); 
\item $u$ is the control variable representing the consumption rate;
\item 
$P$ is a production function, which takes the current and past values and capital as inputs and returns the production as output;
\item  $\rho>0$ is a discount factor and $w: \mathcal{C}\to\R$ is a  utility function.
\end{enumerate}

\subsubsection{Rewriting the problem in a Hilbert space setting absorbing the delay}
The optimal control problem with delay  described above can be reformulated as an optimal control problem without delay by lifting it to infinite dimension. The idea is to look at the whole path $k(t+\cdot)|_{[-d,0]}$ as a state variable in order to absorb the delay. More precisely, to give a formulation in a Hilbert space, it is useful to distinguish the present of the trajectory of  $k$ from its past by considering the dynamics of the couple $\left(k(t),k(t+\cdot)|_{[-d,0]}\right)$ as an element of the Hilbert space $X=\R\times L^2(-d,0;\R)$ endowed with the inner product 
\[
\langle x,z\rangle=  {x}_0 {z}_0 + \langle {x}_1, {z}_1 \rangle _{L^2([-d,0];\R)}, \ \ \ \ \ \ x=({x}_0, {x}_1)\in X, \ z=({z}_0, {z}_1) \in {X}.
\]
To this purpose, set
$$
y(t)=(y_0(t),y_1(t))=\left(k(t), k(t+\cdot)|_{[-d,0]} \right)\in X.
$$ 
The state equation (without delay) formally associated to the state variable $y(\cdot)$  in the space $X$ is 
\begin{equation}\label{ref:eqinf}
y'(t)=\mathcal{A}y(t)+F(y(t))-u(t)\hat n, \ \ \ y(0)=x,
\end{equation}
where $\hat n=(1,0)\in X$,
$$
\mathcal{A}: D(\mathcal{A})= \Big\{x\in \R\times W^{1,2}([-d,0];\R): \ x_1(0)=x_0\Big\}\subseteq X \to X, \ \ \ \ \ \ \mathcal{A} x= (x_{0}, x_1'(\cdot)),$$
$$
F:X\to X, \ \ \ \ \ F(x)= (F_{0}(x),F_{1}(x))=\left(P\Big(x_0,\, \int_{-d}^0 a(\xi)x_1(\xi)d\xi\Big)-x_{0},\,\, 0\right).
$$
The unbounded operator $\mathcal{A}$ is closed, densely defined, and generates a strongly continuous semigroup of linear operators of $\mathcal{L}(X)$. Moreover, it is invertible (with bounded inverse), which allows to define the weaker norm $|x|_{-1}=|A^{-1}x|$ in the space $X$. 
The objective functional remains \eqref{functttb} and the set of admissible controls is just rewritten as
 $$
\mathcal{U}(x)=\Big\{u(\cdot)\in L^1_{loc}(\R^+;\mathcal{C}): \ y_0^{x,u(\cdot)}(\cdot)>0\Big\}.
$$
We actually have an equivalence between solutions to \eqref{statedelay} and mild solutions to \eqref{ref:eqinf} (see \citealp[Proposition 3.3]{federico2010hjb}), which enables us to establish a rigorous correspondence of the original finite dimensional control problem with delay and the lifted control problem without delay in the Hilbert space $X$, posing ourselves in our theoretical  framework  of Subsection \ref{sec:infer} with 
$$f_{0}(x,u)=F(x)-u \hat n, \ \ \ g_{0}(x,u)=h(u)$$ and therefore posing the basis to employ the dynamic programming techniques in infinite dimension to approach the problem.
\subsubsection{Preliminary properties of the value function}
Some preliminary properties of the value function
$$V(x)=\sup_{u(\cdot)\in\mathcal{U}(x)} \mathcal{J}(u(\cdot)), \ \ \ x\in X,$$
 are established a priori in \cite{federico2010hjb}, Secs.\,2--3, dealing both with the original problem with delay in dimension one  and with the lifted problem in $X$ without delay. Specifically, under reasonable assumptions on the problem data\footnote{Basically, it is assumed that: (i)  $P$ is nondecreasing in its second variables, Lipschitz continuous, and jointly concave; (ii)  $a \in W^{1,2}([-d,0];\R^{+})$ with $a(-d)=0$; (iii)  $h$ satisfies standard assumptions for utility functions. For more details, we refer to the paper.}, it is shown that $V$ is:
\begin{enumerate}[(i)]
\item Nondecreasing with respect to the partial order in $X$;  
\item Concave and finite on a convex subset $\mathcal{O}$ of the positive cone of $X$, with $\mathcal{O}$ open with respect to $|\cdot|_{-1}$; the finiteness is consequence of the important estimate 
\begin{equation}\label{stimanorma}
|x_0|\leq C|x|_{-1}, \ \ \ \forall x=(x_0,x_1(\cdot)).   
\end{equation}
for some $C>0$.
\end{enumerate}
The latter result (ii) implies that $V$ is locally Lipschitz continuous with respect to $|\cdot|_{-1}$ within the domain  $\mathcal{O}$, leading to an important regularity result for the superdifferential $D^{+}V$. Specifically, $D^{+}V(x) \subset D(\mathcal{A}^{*})$ and this result enables a proper treatment of the term $\langle \mathcal{A}x,Dv(x)\rangle $ in the HJB equation by eliminating the unboundedness of $\mathcal{A}$ passing to the adjoint through formal equality $\langle \mathcal{A}x,Dv(x)\rangle =\langle x,\mathcal{A}^*Dv(x)\rangle$.    
\subsubsection{The \emph{HJB} equation in the Hilbert space: viscosity property and partial regularity of the value function}
The stationary HJB equation associated to the optimal control problem in the space $X$ is  \eqref{eq:HJBst} with 
\begin{equation}\label{HJBgoldys}
\mathcal{H}(x,p)=F_{0}(x)p_{0}+ \mathcal{H}^{o}(p_{0}) , \ \ \ x\in\mathcal{O}, \ \ p=(p_{0},p_{1})\in X.
\end{equation}
where 
$$ \mathcal{H}^{o}(p_{0})= \sup_{u\in\mathcal{C}} \big\{h(u)-up_{0}\big\}, \ \ \ p_{0}\in\R.$$
What is important to notice is that only the first (one dimensional) component $p_{0}$ of $p\in H$ is involved in this term, in particular in the nonlinear part. Since $p$ is the formal entry of $Dv(x)$ in the HJB equation, this means that only the one-dimensional directional derivative $v_{x_{0}}(x)=\frac{\partial v}{\partial x_{0}} (x_0,x_{1})$, corresponding to the formal writing $\langle Dv(x),\hat n\rangle$, is needed to give sense to the nonlinear term in HJB equation, which can be more explicitly written as \begin{equation}\label{HJBgoldysexpl} \rho v(x)=\langle \mathcal{A}x,Dv(x)\rangle + F_{0}(x)v_{x_{0}}(x)+ \mathcal{H}^{o}(v_{x_{0}}(x)). \end{equation} Particularly important is that the nonlinear term $\mathcal{H}^{o}$ only depends on $v_{x_{0}}(x)$, since this means that the candidate optimal feedback map associated to a solution to the HJB equation $v$ can be formally defined by only using $v_{x_0}$: assuming existence and uniqueness\footnote{Conditions for that can be given: we refer to  \cite{federico2011hjb} and refrain from providing them here, for the sake of brevity. To fix the ideas, here we may just assume that $\mathcal{C}$ is compact and $h$ is strictly concave.} of the $\argmax$, the candidate optimal feedback map is \begin{equation}\label{feedgoldys} \phi(x)=\mbox{argmax}_{u\in\mathcal{C}}\left\{h(u)- u v_{x_{0}}(x)\right\}. \end{equation} In \cite[Sec.\,4]{federico2010hjb}, the following main results, collected here in a single statement, are proved: \begin{Theorem}\label{th:goldys} $V$ is a viscosity solutions to \eqref{HJBgoldysexpl} over $\mathcal{O}$ and it is continuously differentiable, in classical sense, along the direction $\hat n$ therein; that is, $V_{x_{0}}(x)$ exists in classical sense for all $x\in\mathcal{O}$ and the map $x\mapsto V_{x_{0}}(x)$ is continuous. \end{Theorem}

\subsubsection{Verification theorem and optimal feedback controls in the context of viscosity solutions with partial regularity} 
  The second claim of Theorem \ref{th:goldys} is the notable one: it is a partial regularity result for viscosity solutions which allows to define the feedback map \eqref{feedgoldys} associated to the solution $V$ in classical sense. Indeed, the map
  \begin{equation}\label{feedgoldysV}
  \phi: \mathcal{O} \to \mathcal{C}, \ \ \ x\mapsto \argmax_{u\in\mathcal{C}}\big\{h(u)- u V_{x_{0}}(x)\big\}
  \end{equation}
  is well defined as a single valued map and continuous. The point is now to establish if it produces an optimal feedback control. In  \cite[Sec.\,3]{federico2011hjb}, the following main result is proved.

\begin{Theorem}\label{teo:vergoldys}
Let $x\in\mathcal{O}$ and let $u^*(\cdot)\in \mathcal{U}(x)$ be such that the following feedback relationship holds true: 
\begin{equation}\label{feedbackrelgoldys}
u^{*}(s)=\phi(y^{x,u^{*}(\cdot)}(s)), \ \ \ \ \mbox{for a.e.} \ s\geq 0.
\end{equation}
Then, $u^{*}(\cdot)$ is optimal starting at $x$.
\end{Theorem}
Theorem \ref{teo:vergoldys} is basically  half verification theorem,  as it basically corresponds to  the statement (ii) of Theorem \ref{teo:verst}; indeed, \eqref{feedbackrelgoldys} is nothing but  \eqref{feedverst} in the present example. However, its proof is much delicate, as it has to rely only on the existence of  $V_{x_{0}}$, whereas the proof of the classical verification theorem passes through the use of the (stationary version) of the infinite dimensional chain rule of Proposition \ref{prop:chain} and the latter requires much more regularity, i.e., that $DV\in C(\mathcal{O};D(\mathcal{A}^{*}))$. 

Moreover, in \cite[Sec.\,4]{federico2011hjb} a careful study of the closed loop equation associated with the feedback map $\phi$, i.e. 
  \begin{equation}\label{CLEgoldys}
y'(t)=\mathcal{A}y(t)+F(y(t))-\phi(x)\hat n, \ \ \ y(0)=x.
  \end{equation}
is performed to show the existence of a control satisfying \eqref{feedbackrelgoldys}. The argument requires an approximation through the introduction of a current additive utility on the state, $h_{0}(y_{0}(t))$, in the functional: it is eventually shown that, under some additional assumptions, there exists a unique optimal control and it is characterized by the feedback relation \eqref{feedbackrelgoldys}. Finally, in \cite{federico2011hjb}, Sec.\,5, an approximation argument is used to link the problem studied to the original economic problem with a pointwise delay $P(k(t-d))$ proposed by \cite{Asea1999}. 

\subsection{Optimal advertising}
The model we present here is an optimal control problem in which the state equation is linear and contains distributed delays in the control variable. Since the objective functional does not admit an explicit solution of the related HJB equation, the viscosity approach is employed. The problem is investigated in \cite{federico2014dynamic}; we here briefly illustrate how the infinite dimensional methodology is employed to succedefully address the problem.  

The motivating application is \emph{optimal advertising}. There is a substantial body of mathematical economics literature on this subject dating back to the 1960s. Key references include the very first contribution employing optimal control theory to describe this kind of application by \cite{nerlove1962optimal}, followed by papers that highlighted the need for delayed effects in the dynamics, e.g. \cite{pauwels1977optimal}. The survey by \cite{feichtinger1994dynamic} provides an excellent overview of the preceding literature on this rich mathematical economics topic. We also mention the stochastic extensions of these models introduced in the works \cite{gozzi200513, gozzi2009controlled}. 

\subsubsection{The economic problem}
The aim is to study the optimal control of the $1$-dimensional linear controlled ODE
\begin{equation}
\label{state}
\begin{cases}
\displaystyle{G'(t)=\,a_0G(t)\,+\,b_0u(t)\,+\,\int_{-r}^0b_1(\xi)u(t+\xi)d\xi, }       \\  G(0)=\alpha_0;\quad u(s)=\,\alpha_1(s),\,\,s\,\in\,[-d,0),
  \end{cases}
  \end{equation}
where  $u(\cdot)$ is the control variable taking values in some interval $\mathcal{C}\subseteq\R^+$; 
$d\in(0,\infty)$ is the delay lenght; $a_0, b_0\in\mathbb{R}^+$,  $\alpha=(\alpha_0,\alpha_1(\cdot))\in \R\times L^2([-d,0];\R^+)$, and $b_1\in L^2([-d,0];\R^+)$. It has be noticed that the delay appears here in the control variable, whose past values  still continue to have effects on the dynamics of the state variable $G(\cdot)$. Clearly, to be \eqref{state} well posed, the past values of the control variable need to be specified, explaining the presence of the initial condition $u(s)=\,\alpha_1(s)$ for $s\,\in\,[-r,0)$ in \eqref{state}. 
The goal is to maximize a functional of the form
\begin{equation}
\label{funcintroduad}
\mathcal{J}(\alpha; u(\cdot))=
\int_0^{\infty} e^{-\rho t}\ell(G(t),u(t))\ud t,
\end{equation}
where $\rho>0$ is a discount factor and $f:\R\times \mathcal{C}\to \R$ is a given measurable function,
over a set of admissible controls, subset of 
$L^1_\loc(\R^+;\mathcal{C})$.

 The interpretation of the mathematical object is the following. 
 The state variable $G(\cdot)$ represents the so-called \emph{goodwill} of a firm and the control variable $u$ represents the effort in advertising in order to increase the goodwill itself. The more is the effort in advertising, the more is the increase of the goodwill; that is, $b_0$ and $b_1(\cdot)$ are nonnegative (and nonidentically vanishing). The term $b_0 u(t)$ measures the immediate impact of the advertising effort, whereas the delay integral term $\int_{-d}^0b_1(\xi)u(t+\xi)d\xi$  accounts for the (empirically evident) fact that there is some memory of the past advertising on the current increase of the goodwill. 
 As for the objective functional, in this application  the function $\ell$ typically assumes a separable form of revenue/cost type in the state and control variables, i.e.
$\ell(G,u)=\ell_0(G)-\ell_1(u)$,
where $\ell_0$ is a revenue function and $\ell_1$ is a cost function satisfying the usual economic assumptions (nondecreasing and concave, nondecreasing and convex, respectively).  The simplest case is when $\ell_0$ is linear and $\ell_1$ is quadratic, leading to a linear-quadratic problem, solvable explicitly; however, the case when $\ell_0$ is nonlinear is  meaningful, since in reality this function is observed to have, in general, genuinely decreasing return to scale. 

\subsubsection{Rewriting the problem in a Hilbert space setting}
The lifting to infinite dimension of the above problem can be done similarly as for the problem illustrated in Subsection \ref{sub:vintagecapitaldelay} through the construction of the structural state (\citealp{VinterKwong}) with a representation through the construction of the structural state $$ y(t)=\left(G(t), \int_{-d}^\cdot b_1(\tau) u(\tau-\cdot)\ud\tau\right). $$ One would be tempted to do that lifting in the product Hilbert space $X=\mathbb{R}\times L^2([-d,0];\mathbb{R})$, as in Subsection \ref{sub:vintagecapitaldelay}. The state equation would be formally rewritten as $$ y'(t)=\mathcal{A}y(t)+ b u(t), \ \ y(0)=x_\alpha, $$ where $b=(b_0,b_1(\cdot))\in X$, $$x_\alpha= \left(\alpha_0,\int_{-d}^\cdot b_1(\tau) \alpha_1(\tau-\cdot)\ud\tau\right) \in X,$$ and $$ \mathcal{A}: D(\mathcal{A})\subseteq X \to X, \ \ \mathcal{A} x= (a_0 x_0+x_1(0), -x_1'(\cdot)), $$ with
$$D(\mathcal{A})= \Big\{x=(x_0,x_1(\cdot))\in \mathbb{R}\times W^{1,2}([-d,0];\mathbb{R}): \ x_1(-d)=0\Big\}.$$ The objective functional would be then reformulated, with abuse of notation, as $$ \mathcal{J}(x_\alpha; u(\cdot))=\int_0^\infty e^{-\rho t} g(y(t), u(t))\ud t, $$ where $g(x,u)=\ell(x_0,u))$, $(x,u)\in X\times \mathcal{C}$. Then, one could try to follow the approach employed in Subsection \ref{sub:ttbnonlinear} for the problem studied there. Unfortunately, an issue arises: in this case the estimate \eqref{stimanorma} does not hold and the argument breaks down. For this reason, in order to restore the approach, in \cite{federico2014dynamic} the problem is embedded in $X= \mathbb{R}\times W^{1,2}_0([-d,0];\mathbb{R})$, where $W^{1,2}_0([-d,0];\R)= \{f\in W^{1,2}([-d,0];\mathbb{R}): \ f(-d)=0\Big\}$. Basically, the representation is done one level of regularity more with respect to the  standard one performed in    $\mathbb{R}\times L^2([-d,0];\mathbb{R})$: the latter is replaced by $\mathbb{R}\times W^{1,2}_0([-d,0];\mathbb{R})= D(\mathcal{A})$ and $D(\mathcal{A})$ is replaced by $D(\mathcal{A}^2)$.
In this way, the estimate \eqref{stimanorma} is restored (see the crucial Remark 5.4 in \citealp{federico2014dynamic}) and the argument may proceed as in Subsection \ref{sub:ttbnonlinear} getting similar results.

\subsection{Vintage capital with transport state equation and a general concave pay-off} \label{vcteimpl} Let us now consider a generalization of the optimal investment problem with vintage capital studied in Subsection \ref{sec:vcteexpl}, in which we have the same state equation, but the objective functional is more generally concave in the state and control variables.
Differently from Subsection \ref{sec:vcteexpl}, we consider here the finite horizon problem

\subsubsection{The economic problem} We consider the following finite horizon problem ($T<+\infty$ is the horizon), having the same state equation \eqref{eq:esvintage} as that of the problem in Subsection \ref{sec:vcteexpl}
\begin{equation}
\label{eq:SEvintagenew}
\begin{cases}\frac{\partial z(\tau, s)}{\partial \tau}+\frac{\partial z(\tau, s)}{\partial s}=-\mu z(\tau, s)+u_1(\tau, s) & \tau \in(t,T), s \in(0, \bar{s}], \\ z(\tau, 0)=u_0(\tau) & \tau \in(t,T), \\ z(t, s)=z_0(s) & s \in[0, \bar{s}],\end{cases}  
\end{equation}
with the same meaning and assumptions there detailed, but a different objective functional 
\[
\mathcal{J}(t, x; u(\cdot)) = \int_t^T e^{-\rho \tau} \big[ R(Q(\tau)) - c(\tau, u(\tau)) \big] \, d\tau 
    + e^{-\rho T} R_0(Q(T))
\]
where $u(\cdot)=(u_0(\cdot),u_1(\cdot))$ and
\[
Q(\tau) = \int_0^{\bar{s}} \alpha(\tau, s) z(\tau, s) \, \ud s
\]
is the output rate (linear in $z(\tau,\cdot)$), i.e., the quantity of product of all different ages available at time $\tau$; 
$R$ and $R_0$ are \emph{concave} revenues from the output, rather than linear as in Section \ref{sec:vcteexpl},  and a general \emph{convex} investment costs 
\[
c(\tau, u(\tau)) = \int_0^{\bar{s}} c_1(\tau, s, u_1(\tau, s)) \, \ud s 
    + c_0(\tau, u_0(\tau)),
\]
with $c_1$ indicating the investment cost rate for technologies of age $s$, while $c_0$ is the investment cost in new technologies, including adjustment–innovation. 

\subsubsection{Rewriting the problem in a Hilbert space setting} 
The problem can be formulated, rather than in $L^2(0,\bar s)$, in the larger space $D(\A^*)'$,  by extending the translation semigroup $e^{t\A}$ of Subsection \ref{sec:vcteexpl}  to the semigroup $e^{t\A_{-1}}$ on $D(\A^*)'$, as defined  in Subsection \ref{SSSE:Bunbounded}. Thus, as in Subsection \ref{sec:vcteexpl}, setting $y(t):=z(t,\cdot)$ the state equation
writes
\begin{equation}\label{eq:clsolvint2}
    \begin{cases}
        y'(\tau)=\A_{-1} y(\tau)+\mathcal{B} u(\tau),&t>0,\\
        y(t)=x=z_0(\cdot),
    \end{cases}
\end{equation}
where $\mathcal{B}$ is as in Subsection \ref{sec:vcteexpl}.
Its mild form in $D(\A^*)'$ reads as 
$$
y(\tau )= e^{\mathcal{A}_{-1}(\tau-t)} x+ \int_t^\tau e^{(s-t)\mathcal{A}_{-1}} \mathcal{B}u(s)\ud s.
$$
If $\langle\cdot,\cdot\rangle$ denotes the duality pairing between $D(\A^*)$ and $D(\A^*)'$,  and    the coefficient $\alpha(\tau)$ is regular, i.e. in $D(\A^*)$, we define the concave function $g$  of $x\in D(\A^*)'$ as
$$
g(\tau, x)=R(\langle \alpha(\tau, \cdot), x\rangle), \quad g_0(k)=R_0(\langle \alpha(T), x\rangle)
$$  so that the objective functional is 
\[
J(t, x; u) = \int_t^T e^{-\rho \tau} \big[ g(\tau, y(\tau)) - c(\tau, u(\tau)) \big] \, d\tau 
    + e^{-\rho T} g_0( y(T))
\]
It is important to note that this approach applies in the enlarged state space by virtue of the stronger  regularity of the data $g$ and $g_0$.

The set of admissible controls is $\mathcal U(t,x)=L^p([t,T]; U)$
The value function is then defined, as usual as 
$$
V(t,k_0)=\sup_{u\in \mathcal U(t,x)}J(t, k_0; u)
$$

\subsubsection{The infinite-dimensional HJB equation: strong solutions} 
The HJB equation associated to the problem is 
$$-v_t(t,k)=\langle k, \A^*Dv(t,k)\rangle+g(k)+c(t)^*(\mathcal B^* Dv(k))$$
coupled with the final condition
$$v(T,k)=g_0(k),$$
and where, for all $t\ge 0$, $c(t)^*$ is  the convex conjugate of $c(t, \cdot)$, namely
$$[c(t)]^*(u)=\sup_{v\in U}\{(\mathcal B^* v,u)_U-c(t,u)\}, \quad u\in U$$
while $(\cdot,\cdot)_U$ represents the inner product in the Hilbert space $U$. 

We  need the following definitions.
If $X,Y$ are Banach spaces, then:

$\displaystyle Lip(X;Y)\,:\,
\left \{f:X\to Y \,:\, [f]_L:=\sup_{x,y\in X,\; x\neq y}
\frac{\vert f(x)-f(y)\vert_{Y}}
{\vert x-y\vert_X} <+\infty\right\}$;

 $\displaystyle{C_{Lip}^1}(X):=\{f\in C^1(X, \mathbb R)~:~ [
f^\prime]_L<+\infty\}$;

 $C_p(X,Y):=\left\{f:X\to Y\ ~:~\vert
f\vert_{C_p}:=\displaystyle\sup_{x\in X} {\vert f(x)\vert_Y\over 1+\vert
x\vert_X^p}<+\infty\right\}; \quad C_p(X):=C_p(X,\mathbb R)$.

Note that $C_p(X)$ is a Banach space, with respect to the associated norms,  and that $Lip(X;Y)$ and
${C_{Lip}^1}(X)$ are not.
\begin{Assumption} \label{ass:vcconc}
\begin{itemize}
\item[]
    \item[$(i)$] $g \in C([0, T], C_2(V')), g_x \in C([0, T], C_1(V', V)),$\\
    and for all $t \in [0, T], g(t, \cdot) \in C^1_{Lip}(V'), g(t, \cdot)$ convex, 
$t \mapsto [g_x(t, \cdot)]_L \in L^1(0, T)$;
     \item[$(ii)$] $g_0 \in C^1_{Lip}(V'), g_0$ concave
\item[$(iii)$] $c : [0, T] \times U \to \mathbb{R}, c(t, \cdot)$ is convex, l.s.c, and  
$\partial_u c(t, \cdot)$ is injective, for all $t \in [0, T].$

\item[$(iv)$] $[c(t)]^* \in C([0, T], C_2(U)), D_u [c(t)]^*\in C([0, T], C_1(U, U))$  \\
and for all $t \in [0, T], [c(t)]^* \in C^1_{Lip}(U), [c(t)]^*( 0) = 0$,\\
$\sup_{t \in [0, T]} [D_u[c(t)]^*]_L < \infty$.
\end{itemize}
\end{Assumption}
Unlike the case where $R$ is the identity function, discussed in Section \ref{sec:vcteexpl}, the HJB associated to the problem with a general concave objective functional admits no explicit classical solution. Nevertheless, one can show that the value function is the unique \emph{strong} solution of the associated HJB equation, as specified below.

\begin{Definition}
    \emph{(Strong solution of HJB)}
A function  $v\in C([0,T],C_2({\D}))$ is a {\rm strong}
solution of
\begin{equation}\label{eq:hjbvintconc}\begin{cases}-v_t(t,k)=\langle   x,\A^*Dv(t,k)\rangle +g(t,k)-[c(t)]^*(\mathcal B^* Dv(t,k))=0 , &(t,x)\in[0,T]\times {\D}\\
v(T,k)=g_0(k)&\\
\end{cases}
    \end{equation}

if there exists a family
$\{v^\e\}_{0<\e<\bar\e} \subset C([0,T],C_2({\D}))$ such
that:
\begin{itemize}
    \item[$(i)$] $v^\e(t,\cdot)\in {C_{Lip}^1}({\D})$ and
$v^\e(t,\cdot)$ is concave for all $t\in[0,T]$;
$v^\e(T,x)=g_0(x)$ for all $x\in{\D}$.

\item[$(ii)$]  there exist  constants $\Gamma_1,\Gamma_2>0$ such that
$$\sup_{t\in[0,T]}[Dv^\e(t)]_L\le \Gamma_1,~
\sup_{t\in[0,T]}\vert Dv^\e(t,0)\vert_{D(\A^*)}\le \Gamma_2,~\forall
\e\in]0,\bar\e[;$$

\item[$(iii)$]  for all $x\in D(A)$, $t\mapsto v^\e(t,x)$ is
continuously
 differentiable;

\item[$(iv)$]  $v^\e\to\phi$, as $\e\to 0+$,
 in $C([0,T],C_2({\D}))$;

\item[$(v)$] there exists  $g_\e\in C([0,T];C_2(\D))$ such that,
for all $t\in[0,T]$ and $x\in D(A)$,
$$-v^\e_t(t,k)=\langle   x,\A^*Dv^\e(t,k)\rangle+g_\e(t,k) -[c(t)]^*(\mathcal B^*Dv^\e(t,k))$$
with $g_\e(t,k)\to g(t,k)$ pointwise,  and $\int_0^T\vert
g_\e(t)-g(t)\vert_{C_2}\ud s\to 0$, as $\e\to 0+.$ 
\end{itemize}
\end{Definition}  
Then the main result is the following.

\begin{Theorem}
    Let Assumptions \ref{ass:vcconc} be satisfied. There exists a unique strong solution $v$ to \eqref{eq:hjbvintconc} with the following properties:
    \begin{itemize}
        \item [$(i)$] for all $k\in D(\A^*)'$, $v(\cdot, k)$ is Lipschitz continuous;
        \item[$(ii)$] for all fixed $t\ge0$, $v(t, \cdot)$ is a convex function in $ C^1_{Lip}(\D)$. 
    \end{itemize}
\end{Theorem}
 
 The long and detailed proof is the contents of  \cite{faggian2005regular}, in particular of Theorem 7.3.   There the reader will find the definition and the properties, extended to $\D$, of the so called Hamilton-Jacobi semigroup implicated in the construction of the approximating equations. \qed
 
\begin{Theorem}
    The value function $V$ of the optimal control problem coincides with the unique strong solution of to \eqref{eq:hjbvintconc}. Moreover, the unique optimal control problem in closed-loop form is given by
    $$u^*(\tau)=D_u[c(t)^*] (\mathcal B^*DV(\tau,k^*(\tau))$$
\end{Theorem}
The proof of this statement constitutes the contents of \cite{faggian2008application}, and in particular of Theorem 5.1, and Proposition 5.7.\qed


		\appendix\label{sec:app}

\section{Strongly Continuous Semigroups, Generators, Resolvents}
\label{sec:semigroup_generator}

In this section, we review the fundamental properties of strongly continuous semigroups of operators and their generators. The material is mainly taken from the books \cite{engel-nagel} and \cite{pazy1983semigroups}.

\begin{Definition}    \label{def:scs}
\emph{($C_0$-semigroup)} A map $T:[0,\infty) \rightarrow \mathcal{L}(X)$ is called a $C_0$ semigroup (or a strongly continuous semigroup) on $X$ if the following three conditions are satisfied:
\begin{enumerate}[(i)]

    \item $T(0) = I$, where $I$ is the identity operator on $X$;
        \item \emph{(Semigroup property)} $T(t+s) = T(t)T(s)$ for all $t, s \ge 0$.
    \item \emph{(Strong Continuity)} For every $x \in X$, the map $[0, \infty)\to X$,  $t \mapsto T(t)x$ is continuous.
\end{enumerate}
\end{Definition}
For $C_0$-semigroups we will use the notation $\{T(t)\}_{t \geq 0}$ or simply $T(t)$.

\begin{Definition}
\emph{(Generator of a $C_0$-semigroup)} Let $T(t)$ be a $C_0$-semigroup on $X$. The linear operator $\A: D(\A) \subset X \rightarrow X$ defined as
$$
\left\{\begin{array}{l}
D(\A):=\left\{x \in X: \ \displaystyle{\lim\limits_{t\to0^+}\frac{T(t) x-x}{t}} \text { exists in } X \right\} \\\\
\displaystyle{\A x:=\lim \limits_{t \rightarrow 0^{+}} \frac{T(t) x-x}{t}}.
\end{array}\right.
$$
is called the infinitesimal generator of $T(t)$.
\end{Definition}

\begin{Proposition}\label{lem:Momega}
    Let $T(t)$ be a $C_0$-semigroup on $X$. Then there exist $M \geq 1$ and $\omega \in \mathbb{R}$ such that
\begin{equation}
    \label{typeofsem}
|T(t)| \leq M e^{\omega t}, \quad \text { for } t \geq 0
\end{equation}
\end{Proposition}
\noindent 
\begin{proof}
    See, for instance, \cite{pazy1983semigroups}  Theorem 2.2, Chap. 1, p. 4.
\end{proof}
The quantity $\omega_0$ defined as
$$\omega_0=\inf\{\omega : \exists M>0,  |T(t)|\le Me^{\omega t}\}$$
is called the \emph{type} of the semigroup. We have $\omega_0 \in[-\infty,+\infty)$. If $\omega_0<0$ we say that the $C_0$-semigroup $T(t)$ is of negative type and if $\omega_0>0$ we say that the $C_0$-semigroup $T(t)$ is of positive type.
\noindent \begin{Definition}
    A $C_0$-semigroup $T(t)$ on $X$ is called: 
   \begin{enumerate}[(i)]  \item 
    \emph a \emph{Contraction Semigroup}   if \eqref{typeofsem} holds with $M=1, \omega_0=0$;
\item  a \emph{Pseudo-contraction semigroup}  if \eqref{typeofsem} holds with $M=1$ for some $\omega_0 \in \mathbb{R}$;
\item a \emph{Uniformly bounded semigroup} if \eqref{typeofsem} holds with $\omega_0=0$ for some $M \geq 1$.
\end{enumerate}
\end{Definition}

\begin{Example}
    A simple and intuitive example of a strongly continuous semigroup can be found on a finite-dimensional Banach space, such as $X = \mathbb{R}^n$ equipped with any norm. In this setting, the generator of the semigroup is always a bounded linear operator, which can be represented by an $n \times n$ matrix $A$. More in detail, let $A$ be an $n \times n$ matrix. We can define a family of operators $\{T(t)\}_{t \ge 0}$ by setting $T(t)$ as the matrix exponential of $tA$:
$$T(t) = e^{tA} = \sum_{k=0}^{\infty} \frac{t^k}{k!} A^k$$
This family of operators forms a strongly continuous semigroup on $\mathbb{R}^n$, whose generator is the matrix $A$ itself.  \qed

The previous example provides a foundational understanding of the relationship between a semigroup of operators and the associated generator, before addressing  the more complex case  in infinite-dimensional spaces where generators can be unbounded. It also explains why a  semigroup of generator $\A $ is  often conventionally expressed as $T(t) = e^{t\A}$.
\end{Example}\color{black}

We also recall that the resolvent set $\rho(\A)$ of a linear operator $\A$ on a Banach space $X$ is the  set of all $\lambda \in \mathbb{C}$ for which $\lambda I - \A$ is invertible with inverse $R(\lambda, \A)\equiv (\lambda I-\A)^{-1}$ belonging to the space $\mathcal L(X)$ of linear continuous operators on $X$.

The main properties of semigroups, generators, and resolvents are recollected in the next theorem.
\begin{Theorem}\label{thm:propsemigr}
    Let $T(t)$ be a $C_0$-semigroup on the Banach space $X$ and let $\A$ be its generator. We have the following claims.
    \begin{enumerate}[(i)]
 \item   For $x \in X$, we have $\int_0^t T(s) x \ud s \in D(A)$ and $$\A\left(\int_0^t T(s) x \ud s\right)=T(t) x-x.
$$
 \item If $x\in D(\A)$, then $T(t)x\in D(\A)$.
         \item If $x\in D(\A)$, then $t\mapsto T(t)x$ is differentiable and,  for all $t\ge0$, we have  $$\frac{\ud}{\ud t}T(t)x=T(t)\A x=\A T(t)x.$$ 
        
         \item $\A$ is a closed and  densely defined linear operator.
 \item The half-plane $\{ \emph{Re}\,\lambda>\omega_0\}$ is contained in $\rho(\A)$ and if  $\lambda$ belongs to such half-plane, we have 
$$R(\lambda, \A)x = \int_0^\infty e^{-\lambda t} T(t)x \, \ud t.$$

    \end{enumerate}
\end{Theorem}
\begin{proof}
    See, for instance, \cite[ Ch.\,1]{pazy1983semigroups}: precisely, Th.\,2.4 and Corollary 2.5, for (i)--(iv), and the proof of Theorem 5.1 for (v). \end{proof}
    
The Hille-Yosida Theorem, characterizes the linear operators $\A$ which are generators os $C_0$-semigroups.
\begin{Theorem}\emph{(Hille–Yosida)}
Let \( \A : D(\A) \subset X \to X \) be a densely defined, closed linear operator on a Banach space \( X \).  
Then \( \A \) is the generator of a $C_0$ semigroup $T(t)$ of bounded linear operators with \( \|T(t)\| \leq M e^{\omega t} \) for some constants \( M \geq 1 \), \( \omega \in \mathbb{R} \), if and only if the following conditions hold:
\begin{enumerate}
  \item[$(i)$] $\A:D(\A)\to X$ is closed and $D(\A)$ is dense in $X$;
   \item[$(ii)$] \( \rho(\A) \) contains the half-line \( (\omega, \infty) \) and, for all \( \lambda > \omega \),
  \[
  \|R(\lambda, \A)^n\| \leq \frac{M}{(\lambda - \omega)^n}, \quad \forall n \in \mathbb{N}.
  \]
\end{enumerate}
\begin{proof}
    See for instance \cite{pazy1983semigroups}, Chapt. 1, Theorem 5.3. \end{proof}
\end{Theorem}

We recall that for an operator $\A$ on a Banach space $X$ with dual space $X'$, and $\langle \cdot,\cdot\rangle $ represents the duality pairing between $X,X'$, the \emph{adjoint operator} $\A^*$ is defined on $X'$ as follows. 
\begin{align}
\nonumber
D(\A^*) 
&=
\left\{ \varphi \in X' \;\middle|\; \exists C>0 
 \text{ such that } 
\forall x \in D(\A),\ \langle \varphi, \A x \rangle \le C | x | 
\right\}
\\
&=
\left\{ \varphi \in X' \;\middle|\; \exists \psi \in X' \text{ such that } 
\forall x \in D(\A),\ \langle \varphi, \A x \rangle = \langle \psi, x \rangle \right\}
\label{eq:defadjointdomain}
\end{align}

and for all \( \varphi \in D(\A^*) \), one sets $\A^* \varphi := \psi$, that is $\A^*$ is determined through the position \[
\langle \A^* \varphi, x \rangle = \langle \varphi, \A x \rangle \quad \text{for all } x \in D(\A)
\]

\begin{Definition}\emph{(Adjoint semigroup)}
\label{df:adjointsemigroup}
If $\A$ is the generator of a strongly continuous semigroup $T(t)$ on the Banach space $X$, then its adjoint $\A^*$ is the generator of the \emph{adjoint semigroup} $\{T(t)^*\}_{t \ge 0}$, which is a  semigroup on $X'$ defined by
\[
\langle T(t)^*\varphi, x \rangle := \langle \varphi, T(t)x \rangle
\quad \text{for all } \varphi \in X',\ x \in X.
\]
    \end{Definition}
 The adjoint semigroup $T(t)^*$ is not necessarily strongly continuous with respect to the topology of $X'$, but it is continuous with respect to the weak-* topology on $X'$, that is, for all \( \varphi \in X^* \) and \( x \in X \), the map
$
t \mapsto \langle T^*(t)\varphi, x \rangle
$
is continuous on \( [0, \infty) \).

However,   if \( X \) is a Hilbert space, and we identify \( X \cong X' \),  the duality is simply the inner product in $X$ and the following result holds.
\begin{Corollary}
\label{cr:adjointHilbert}
Assume $T(t)$ is a \( C_0 \)-semigroup on \( X \), with $X$ a Hilbert space.
    Then the adjoint semigroup \( T(t)^* \)  is a \( C_0 \)-semigroup on \( X \).
\end{Corollary} 
\begin{proof}
  It follows from  \cite{pazy1983semigroups}, Chapt.1, Theorem 10.1, and Theorem 3.1, and the fact that the identification $X=X'$ preserves the norm, implying $\|R(\lambda, \A^*)\| = \|R(\lambda, \A)\|$.
\end{proof}

\begin{Example}[Translation semigroup] \label{ex:trsem}
We here introduce the example of a semigroup and its generator that is used in some examples of the paper, the translation
semigroup. It can be defined in the Hilbert space $X=L^2(0,\bar s)$, for $\bar s\in (0,+\infty)$, as
\[
T(t)f(s) := 
\begin{cases}
f(s - t), & s > t \\
0, & 0 < s \leq t
\end{cases}
\quad \text{or equivalently} \quad
T(t)f(s) := f(s - t)\, \chi_{[t, \bar s]}(s).
\]
and it is  a \( C_0 \)-semigroup of contractions, meaning 
$\|T(t)\| \leq 1$, in $\mathcal{L}(X)$.

\noindent Its generator is the (negative) derivative operator:
$$D(\A) = \left\{ f \in H^{1}(0, \bar s) \;\middle|\; f(0) = 0 \right\}, \quad \A f = -f',$$
where \( H^{1}(0, \bar s) \) denotes the Sobolev space of weakly differentiable functions with derivative in \( L^2(0,\bar s)\).
 The resolvent  $R(\lambda, \A)$  is defined by solving in $X$ with respect to $f$ the Cauchy problem:
$$\lambda f(s) + f'(s) = g(s),\qquad f(0)=0,$$
whose unique solution is,    for $\text{Re}\lambda > 0$,
$$
f(s)=[R(\lambda, \A)g](s) = 
\int_0^s e^{-\lambda u} g(s-u) \, d u.
$$ 
Since $L^2(0,\bar s)$ is a Hilbert space, the adjoint semigroup is 
\[
[T(t)^*\varphi](s) = \begin{cases} \varphi(s+t), & s <\bar s -t,\\ 0, & s \geq \bar s-t, \end{cases}
\]
with generator 
\[
D(\A^*) = \left\{ \varphi \in H^{1}(0, \bar s) \;\middle|\; \varphi(\bar s) = 0 \right\}, \quad
\A^*\varphi = \varphi',
\]
and resolvent
\[ [R(\lambda, A^*)g](s) = \int_s^{\bar s}e^{\lambda(s-u)} g(u) \,\ud u. \vspace{-.7cm}\]

\qed
\end{Example}

{
		}
\begin{spacing}{1.45}

\cite{*}

{\small 
\bibliographystyle{apalike}
\bibliography{biblio}
}

\end{spacing}

\end{document}